\documentclass[12pt,openany]{article}
\usepackage[utf8]{inputenc}
\usepackage[spanish]{babel}
\usepackage{amsmath,times}
\usepackage{amssymb,amsmath}
\usepackage{graphicx}
\usepackage{amsthm}
\usepackage{multicol}
\usepackage[all]{xy}
\usepackage{cancel}
\usepackage{tikz}
\usepackage{vmargin}
\spanishdecimal{.}

\title{El monoide de endomorfismos de $G$-conjuntos: estructuras y otras propiedades algebraicas}
\author{Ramón H. Ruiz-Medina\footnote{Email: harath.ruiz@academicos.udg.mx}  \\[1em]
\small{Centro Universitario de Ciencias Exactas e Ingenierías}, \\ 
\small{Universidad de Guadalajara, Guadalajara, México.}}
\date{}

\newtheorem{teorema}{Teorema}[section]
\newtheorem{lema}[teorema]{Lema}
\newtheorem{corolario}[teorema]{Corolario}
\newtheorem{proposicion}[teorema]{Proposición}
\newtheorem{ejemplo}[teorema]{Ejemplo}
\newtheorem{definicion}[teorema]{Definición}
\newtheorem{observacion}[]{Observación}

\newcommand{\EndG}{\mathrm{End}_{G}(X)}
\newcommand{\AutG}{\mathrm{Aut}_{G}(X)}

\setpapersize{A4}
\setmargins{2.5cm}       
{2cm}                        
{16.5cm}                      
{23.42cm}                    
{10pt}                           
{1cm}                           
{0pt}                             
{2cm}                           

\begin{document}

\pagenumbering{roman}

\maketitle

\textbf{Abstract}\\
Given the action of a group $G$ on a set $X$, an endomorphism of $X$ is a function $f:X \rightarrow X$ which is $G$-equivariant, that is, it commutes with the action, i.e., $f(g\cdot x)= g\cdot f(x)$, for all $x\in X$. The set of endomorphisms of a $G$-set $X$ is a monoid, with the composition of functions , which we will denote $\EndG$. Given subsets $U,N\subseteq M$, we say that $U$ generates $M$ modulo $N$ if it is satisfied that $M= \langle U \cup N \rangle$. The relative rank of M modulo N is the minimum cardinality of a set $U$ to generate $M$ modulo $N$. In this work we address the particular case in which $G$ and $X$ are finite to calculate the relative rank of the endomorphism monoid $\EndG$ modulo its group of units, denoted by $\AutG$. We also address structure situations, such as isomorphisms of $\AutG$ and $\EndG$ with other known structures.\\

\textbf{Resumen}\\
Dada la acción de un grupo $G$ sobre un conjunto $X$, un endomorfismo de $X$ es una función $f:X \rightarrow X$ la cual es $G$-equivariante, es decir que conmuta con la acción, i.e., $f(g\cdot x)= g\cdot f(x)$, para todo $x\in X$. El conjunto de endomorfismos de un $G$-conjunto $X$ es un monoide, con la composición de funciones, al que denotaremos $\EndG$. Dados subconjuntos $U,N\subseteq M$, decimos que $U$ genera $M$ modulo $N$ si se satisface que $M= \langle U \cup N \rangle$. El rank relativo de M modulo N es la cardinalidad mínima de un conjunto $U$ para generar a $M$ modulo $N$. En este documento abordamos el caso particular en el que $G$ y $X$ son finitos para calcular el rank relativo del monoide de endomorfismos $\EndG$ modulo su grupo de unidades, denotado por $\AutG$. Además abordamos situaciones de estructura, como isomorfismos de $\AutG$ y $\EndG$ con otras estructuras conocidas.

\newpage
\section*{Agradecimientos}
\begin{flushright}
\textit{La vida me ha llenado de personas maravillosas las cuales han sido importantes para mi, tantas que la verdad sería difícil mencionarlas a todas. Sin embargo es importante mencionar al Dr. Alonso Castillo quien ha sido un maravilloso mentor. A mi madre, quien es la persona a la que debo todo lo que tengo y todo lo que soy. Es ella la persona más importante de mi vida y sin ella absolutamente nada de esto habría sido posible.  A mi novia Lore, que me ha acompañado durante todo este tiempo, apoyándome y motivándome a siempre ser la mejor versión posible de mi. Su amor y su cariño me han llevado hasta donde estoy.\\
Mi familia, mis amigos, mis compañeros, mis profesores y sinodales también han sido parte importante de todo este proceso y jamás tendré suficiente tiempo para expresarles el cariño y agradecimiento que les tengo. }
\end{flushright}
\newpage

\maketitle

\tableofcontents

\newpage

\setlength{\parindent}{0in}

\pagenumbering{arabic}
\setcounter{page}{1}
\section{Introducción}

Un monoide es una estructura algebraica que se compone de un conjunto dotado de una operación binaria la cual es asociativa y cumple con tener un elemento neutro. Si dicha operación cumple también con la propiedad de que todo elemento tiene un elemento inverso, entonces la estructura cambia su nombre a \emph{grupo}. Dado un grupo $G$, un $G$-conjunto es un conjunto $X$ el cual tiene definida una acción con $G$; una acción es una aplicación del producto cartesiano $G\times X$ al mismo conjunto $X$, la cual satisface un par de sencillos axiomas. Dentro de la categoría de los $G$-conjuntos los llamados morfismos son aquellas funciones que conmutan con la acción, esto quiere decir que si sobre un elemento cualquiera dentro de $X$ queremos aplicar tanto la función como la acción de un elemento de $G$, es indiferente para el resultado en qué orden estas aplicaciones se realicen. El conjunto de morfismos de $G$- conjuntos, también llamados funciones $G$-\emph{equivariantes}, es un monoide con la composición de funciones.    \\

Dentro de la teoría de monoides y grupos el estudio de los llamados \emph{rank} y \emph{rank relativo} es un campo de gran interés. El primero representa el cardinal mínimo de algún conjunto generador dentro del monoide; subconjuntos los cuales al operar todos sus elementos logramos obtener todo el monoide completo. El rank relativo es por su parte la cantidad mínima de elementos que se deben agregar a un subconjunto para convertirlo en un conjunto generador, decimos que dicha cantidad es el rank relativo del monoide módulo el subconjunto.\\   

Los autómatas celulares (AC) surgieron en la década de 1940 con John Von Neumann. Estos son objetos de gran importancia en las ciencias computacionales y a su vez son la motivación de nuestro trabajo. Dado un conjunto $A$, denominado \emph{alfabeto}, el espacio de configuraciones es el conjunto de funciones que trasforman elementos de un grupo cualquiera $G$ en elementos del alfabeto. Sobre este mismo espacio se puede definir una acción llamada \emph{la acción de traslación}, la cual es un objeto bastante estudiado en diversas ramas de las matemáticas como la dinámica simbólica, los sistemas dinámicos, entre otros, además de tener bastantes aplicaciones en otras ciencias como la física o la biología. Los autómatas celulares son funciones definidas a partir de restricciones dentro del grupo a un conjunto finito llamado \emph{conjunto memoria}. Por su definición los autómatas celulares conmutan con la acción de traslación, convirtiendo así al espacio de configuraciones en un $G$-conjunto. El Dr. Alonso Castillo, junto con algunos otros investigadores, han abordado el concepto de autómata celular desde un punto de vista meramente algebraico y son parte de nuestro trabajo al ser un ejemplo particular de un monoide de endomorfismos de $G$-conjuntos. \\

Uno de los objetivos de este trabajo es determinar el rank relativo del monoide de endomorfismos de $G$-conjuntos con respecto a su grupo de unidades, el grupo de los llamados automorfimos de $G$-conjuntos. Se definen los conceptos y las herramientas necesarias para determinarlo. El otro objetivo es determinar similitudes de la estructura algebraica del monoide, o de submonoides dentro de él, con objetos matemáticos más simples o más conocidos. \\

Los resultados de esta tesis forman parte del artículo \cite{raa}, el cual se encuentra publicado en la revista Semigroup Forum. Antes de enunciar el resultado principal de esta tesis es necesario establecer notación. Sea $G$ un grupo y $H$ un subgrupo denotamos por $[H]$ a la clase de conjugación de $H$, el conjunto de los subgrupos conjugados de $H$ por todos los elementos de $G$. Llamamos caja $\mathcal{B}_{[H]}$ al conjunto de puntos en $X$ cuyos estabilizadores son conjugados de $H$. Dado un subconjunto $N$ de $G$, definimos la $N$-clase de conjugación de $H$ como $[H]_{N}$, el cual es el conjunto de subgrupos conjugados de $H$ únicamente por los elementos de $N$. Definimos el conjunto $U(H)$ como el conjunto de $N$-clases de conjugación de los subgrupos de $G$ que contienen a $H$ y denotamos por $\kappa_{G}(X)$ al conjunto de subgrupos de $G$ los cuales tienen una sola $G$-órbita en su correspondiente caja. Entonces demostramos en esta tesis que el rank relativo del monoide de endomorfismos de un $G$-conjuntos $X$, $End_{G}(X)$, módulo su grupos de unidades, el grupo de automorfismos del $G$-conjunto $X$, $Aut_{G}(X)$, está dado por:
$$\sum_{i=i}^{r}{|U(H_{i})|}-|\kappa_{G}(X)|.$$

Este trabajo consta de cuatro partes, además de la sección de apéndices: en la segunda sección se enlistan y se demuestran una seria de conceptos preliminares los cuales son de suma importancia para comprender y desarrollar nuestros objetivos. Se profundiza lo necesario en los conceptos de monoide, grupo, acción y endomorfismo. En la tercer sección se desarrolla la parte medular de nuestro trabajo. Se definen las características especiales de nuestro trabajo y se demuestran todos los resultados sobresalientes. La cuarta sección es un apartado en el cual se toma el monoide de autómatas celulares para ejemplificar todos los resultados que aparecen en la sección número tres. En la quinta sección encontramos las conclusiones, donde se plasman de manera breve y concisa los resultados sintetizados y se dan algunos comentarios adicionales a nuestro trabajo. Finalmente este documento cuenta con una sección de apéndices en la cual se detallan algunos conceptos que se consideran relevantes para el desarrollo de este trabajo.

\newpage

\section{Conceptos preliminares}

\subsection{Los monoides como estructura algebraica}
El concepto básico más necesario para nuestro trabajo es el concepto de \emph{monoide}. Trabajaremos con distintas categorías, todas relacionadas, las cuales son construidas a partir de este concepto. En esta sección daremos los conceptos básicos que abarcan desde la definición de monoides hasta las definiciones de otras estructuras, morfismos, acciones, producto corona, entre otras, así como los resultados necesarios para la construcción de nuestro trabajo. 

\begin{definicion} {(Monoide)}
Un conjunto $M$ se dice \emph{semigrupo} si está dotado de una operación binaria $\cdot: M\times M \rightarrow M$ asociativa: $$(x\cdot y)\cdot z=x \cdot (y\cdot z),\ \forall x,y,z\in M.$$
Si además existe un elemento $e\in M$ tal que $$x\cdot e= e\cdot x= x,\ \forall x\in M, $$ decimos que $M$ es un \emph{monoide}.
\end{definicion}

El elemento $e$ recibe el nombre de elemento identidad. Nombramos \emph{unidades} a los elementos $x\in M$ para los cuales existe un elemento $y\in M$ tal que $x\cdot y=y\cdot x= e$. Denotaremos por $M^{*}$ al conjunto de unidades de $M$, también llamados elementos invertibles. Además denotamos $y=x^{-1}$ al nombrado elemento inverso de $x$.\\

Existen diversas notaciones para la operación definida en un conjunto, aunque este no cumpla con todas las características para ser un monoide, tales como $x\cdot y$, $x+y$, o incluso simplemente $xy$. Pero todas significan lo mismo, la operación dentro del conjunto de los elementos $x$ y $y$. Existen otras representaciones especiales para las operaciones en diversos conjuntos (aunque no sean monoides) como los conjuntos de funciones con la composición , $f\circ g$, o $\mathbb{R}^{3}$ con el producto cruz, $u \times v$.

\subsection{El monoide de transformaciones}
Recordamos que una función, transformación, mapeo o aplicación de un conjunto $X$ a un conjunto $Y$ es una relación la cual asigna a cada elemento de $X$ un único elemento en $Y$. Para una notación más agradable, cuando el conjunto $X$ es finito, podemos escribir
$$f=\left( \begin{array}{ccccc}
x_{1}& x_{2} & \hdots & x_{n} & \hdots \\
y_{1}& y_{2} & \hdots & y_{n} & \hdots \end{array} \right),$$
donde decimos que $y_{i}=f(x_{i})$ es la imagen del elemento $x_{i}$. También solemos denotar por $f:X\rightarrow Y$, a las funciones que van de $X$ en $Y$.\\

Llamaremos funciones \emph{constantes} a las funciones que para un $a\in Y$ fijo, $f(x)=a$, $\forall x\in X$. Recordemos que una función se dice \emph{inyectiva}, si se tiene que $f(x)=f(x')$ implica que $x=x'$. Y se dice \emph{sobreyectiva} si para cualquier elemento $y\in Y$, existe un elemento $x\in X$ tal que $f(x)=y$. Una función que cumple con ser inyectiva y sobreyectiva se dice \emph{biyectiva}. \\

Definimos una \emph{operación} sobre el conjunto de funciones. Dadas dos funciones $f:X\rightarrow Y'$ y $g:Y\rightarrow Z$, donde $Y' \subseteq Y$, definimos la \emph{composición} de $g$ con $f$ como:
$$g\circ f(x)= g(f(x)).$$
Notemos que $g \circ f$ define una función que toma elementos de $X$ y, dado que $f(x)$ es un elemento en $Y'$, les asigna elementos en $Z$, i.e., $g\circ f:X\rightarrow Z$. No es complicado verificar que la composición de funciones biyectivas, es también una función biyectiva.\\

Debemos notar que no existe una \emph{cerradura} bajo estas condiciones para la composición de funciones, pues las funciones que van de $X$ en $Y$, las funciones que van de $Y$ en $Z$ y las funciones de $X$ en $Z$ son objetos matemáticos completamente distintos. También debemos notar que la composición de $f$ con $g$ no es posible, pues $f(g(x))$ no tiene sentido, $g(x)$ es un elemento en $Z$ y $f$ toma elementos de $X$. Sin embargo, si $X$, $Y$ y $Z$ fueran el mismo conjunto, digamos $X$, entonces la composición de dos funciones $f,g:X\rightarrow X$ da como resultado una función $g\circ f:X\rightarrow X$ e incluso la composición $f\circ g$ es posible.\\

Llamamos \emph{identidad} a la función que satisface que 
$$f(x)=x,\ \forall x\in X,$$ denotada como $id$.

\begin{lema}
Dado un conjunto $X$, el conjunto de todas las funciones de $X$ en $X$,
$$Trans(X):=\{f:X\rightarrow X\},$$ dotado de la composición de funciones, es un monoide.\end{lema}

\begin{proof}
Debemos recordar que dos funciones $f,g:X\rightarrow Y$ son iguales si $f(x)=g(x)$, $\forall x\in X$. Entonces dadas tres funciones $f,g,h:X\rightarrow X$ tenemos que 
$$\begin{array}{rl}
(f\circ g)\circ h (x) &= (f\circ g)(h(x))\\
&=f(g(h(x)))\\
&=f(g\circ h(x))\\
&=f\circ (g\circ h)(x),\  \forall x\in X. 
\end{array}$$
Entonces $(f\circ g)\circ h=f\circ (g\circ h)$ y la operación es asociativa. La función identidad $id(x)=x$, cumple con ser la identidad del monoide.
$$\begin{array}{rl}
f\circ id (x)&= f(id(x))\\
&=f(x)\\
id\circ f (x)&=id(f(x)), \forall x\in X,

\end{array}$$
es decir que $f\circ id = id\circ f = f$.
\end{proof}

También definimos el concepto de \emph{transformación} o \emph{función parcial} de un conjunto como una función $f$ que va de un subconjunto $A$ de $X$, es decir, $f:A\subseteq X\rightarrow X$. Denotaremos por $PT(X)$ al conjunto de transformaciones parciales de un conjunto $X$ en sí mismo, es decir:
$$PT(X):=\{f:A \rightarrow X|\ A\subseteq X\}.$$
Es importante mencionar que para los elementos $x$ que no se encuentren en $A$ decimos que la función no está definida y podemos utilizar el símbolo del vacío para denotar esta situación, esto es
$$f(x)= \emptyset,\ \forall x\notin A.$$

\begin{ejemplo}

Sea $X$ un conjunto con 2 elementos, digamos $x\in \{a,b\}$. Utilizando la notación de arreglos, tenemos que los únicos elementos en $PT(X)$ son:
\begin{tiny}
$$PT(X)=
\left\{
\begin{array}{c}

f_{1}=\left( \begin{array}{cc} a&b \\ a& a \end{array} \right),\ 

f_{2}=\left( \begin{array}{cc} a&b \\ a& b \end{array} \right),\ 

f_{3}=\left( \begin{array}{cc} a&b \\ b& a \end{array} \right),\ 

f_{4}=\left( \begin{array}{cc} a&b \\ b& b \end{array} \right),\ 

\\

f_{5}=\left( \begin{array}{cc} a&b \\ a& \emptyset \end{array} \right),\ 

f_{6}=\left( \begin{array}{cc} a&b \\ b& \emptyset \end{array} \right),\ 

f_{7}=\left( \begin{array}{cc} a&b \\  \emptyset& a \end{array} \right),\ 

f_{8}=\left( \begin{array}{cc} a&b \\ \emptyset & b \end{array} \right),\ 

f_{9}=\left( \begin{array}{cc} a&b \\ \emptyset & \emptyset  \end{array} \right)

\end{array} \right\}.
$$ \end{tiny}
Realizamos, utilizando diagramas como herramienta visual, la composición de dos elementos de $PT(X)$.

$$f_{5} \circ f_{3}= \xymatrix{ 
a\ar[rd] &a \ar[r]  & a  \\
b\ar[ru]& b\ar[rd]  & b \\
 & & \emptyset 
}$$
Notemos lo que esta operación nos arroja como resultado:
$$f_{5} \circ f_{3} = \left( \begin{array}{ccc}  a \longrightarrow \emptyset \\
b \longrightarrow a\end{array}\right) = f_{7}.$$
Notemos que si $f_{i}(x)=\emptyset$, diremos que por definición $f_{i}(\emptyset)=\emptyset$, entonces tendrá sentido componer, por ejemplo $f_{5}$ con $f_{3}$. Esto para que al momento de componerlas en este orden la composición se logre realizar.
$$f_{3} \circ f_{5}= \xymatrix{ 
a\ar[r] &a \ar[rd]   & a  \\
b\ar[rd]& b \ar[ru] & b \\
\emptyset \ar[r]& \emptyset \ar[r]  & \emptyset 
}$$
Pues en general $f_{3}(\emptyset)$ no está definido. Con esta convención tiene sentido decir que 

$$f_{3} \circ f_{5} = \left( \begin{array}{ccc}  a \longrightarrow b \\
b \longrightarrow \emptyset \end{array}\right) = f_{6}.$$
Entonces existe una cerradura y por lo tanto $PT(X)$ es un monoide. También es importante notar que $f_{1}$, $f_{2}$, $f_{3}$ y $f_{4}$ son funciones definidas en todo $X$, por lo tanto están en $Trans(X)$ y $Trans(X) \subset PT(X)$.  

\end{ejemplo}

Debemos mencionar también que $$f_{3} \circ f_{5} \neq f_{5} \circ f_{3},$$
lo cual significa que la composición no tiene la propiedad \emph{conmutativa}. No todos los monoides cuentan con esta propiedad.\\

El siguiente resultado es bastante conocido para la caracterización de las unidades dentro del monoide de transformaciones. 

\begin{proposicion}
Dado un elemento $f\in Trans(X)$, $f$ es una unidad si y solo si es biyectivo.
\end{proposicion}

\begin{definicion}
El conjunto de unidades de $Trans(X)$, es llamado el conjunto de simetrías, o permutaciones, de $X$ y lo denotaremos $Sym(X)$.  
\end{definicion}

Debido a que la composición de funciones biyectivas es una biyección, tenemos que el conjunto de permutaciones de un conjunto $X$ es un monoide, más específicamente un grupo, como veremos en secciones posteriores.\\

Toda permutación $\sigma \in Sym(X)$ de un conjunto finito, si utilizamos la notación de arreglos, puede ser visualizada como
$$\sigma= \left(  \begin{array}{cccc}  x_{1} & x_{2} & \cdots & x_{n} \\ y_{1} & y_{2} & \cdots & y_{n} \end{array}\right).$$
Además se cumple que cada $y_{i}$ es distinto e incluye todos los elementos de $X$.\\

Ambos conjuntos, el de transformaciones y el de simetrías, son estudiados de manera particular en secciones posteriores dentro de los apéndices de esta tesis por algunas de sus características las cuales son relevantes para nuestro trabajo. Entre estas características encontramos su cardinalidad, su cerradura y, lo más importante para nuestro trabajo, sus conjuntos generadores.

\subsubsection{Homomorfismos de monoides}
En esta sección introduciremos la noción de morfismo, la cual es pieza clave para trabajar con cualquier estructura algebraica, como por ejemplo para demostrar el Lema de Cayley, herramienta importante la cual nos permitirá trabajar con cualquier monoide a partir de algún monoide de transformaciones.

\begin{definicion}{(Homomorfismo)}

Dados dos monoides $(M,\cdot)$ y $(N,\circ)$, un \emph{homomorfirmo} es una función $\varphi:M\rightarrow N$ la cual satisface que 
$$\varphi(x \cdot y)=\varphi(x)\circ \varphi(y),\ \forall x,y\in M.$$
\end{definicion}

Si $\varphi$ es una función biyectiva, la cual cumple que $\varphi^{-1}$ es también un homomorfismo, entonces decimos que $\varphi$ es un \emph{isomorfismo} entre $M$ y $N$, decimos entonces que $M$ y $N$ son \emph{isomorfos} y denotamos $M\cong N$.\\

Dado un homomorfismo inyectivo $\psi$ entre 2 monoides $M$ y $N$, se tiene que $M$ es isomorfo a su imagen, i.e., $M\cong \psi(M)$. Y dado que la imagen de cualquier monoide bajo cualquier homomorfismo es un monoide en si, entonces se tiene que $\psi(M)$ es un submonoide de $N$, lo cual denotamos $\psi(M) \leq N$ y en particular nos permite decir que $$M \leq N,$$ a pesar de que como conjuntos son completamente distintos y no existe una contención real entre ellos.

\begin{ejemplo}
Tomemos el monoide de matrices de $2\times 2$, $M_{2\times 2}$ , con entradas reales y el grupo de transformaciones de un conjunto $X=\{a,b\}$.\\
Notemos que $Trans(X)$ solo tiene 4 elementos, de los cuales $f_{1}$ y $f_{4}$ son funciones constantes las cuales denotaremos por el elemento de su imagen, es decir 
$$\begin{array}{c}
f_{1}=a\\
f_{4}=b.
\end{array}$$
También se tiene que $f_{2}=id$ y denotaremos a $f_{3}=(1\ 2)$, usando la notación usual de ciclos cuando se trabaja el monoide de transformaciones (ver apéndice A). Entonces 
$$Trans(X)=\{id, (1,2),a,b\}.$$
Además expresaremos todas las operaciones de $Trans(X)$ en la siguiente matriz, a la cual llamaremos tabla de Cayley:

$$\begin{array}{c|cccc} 
\circ & id &(1\ 2) &a &b \\
\hline
id &id &(1\ 2) &a &b \\
 (1\ 2)&(1\ 2) &id &b &a \\
 a& a&a &a &a \\
 b& b& b&b &b 

\end{array}$$
Definimos entonces el mapeo $\varphi:Trans(X) \rightarrow M_{2\times 2}$ definido por:

$$\varphi= \left\{ \begin{array}{ccc}
 id& \longrightarrow & \left( \begin{array}{cc}1&0 \\0&1 \end{array}\right) \\
 (1\ 2)& \longrightarrow & \left( \begin{array}{cc}0&1 \\1&0 \end{array}\right)   \\
 a& \longrightarrow & \left( \begin{array}{cc}1&1 \\0&0 \end{array}\right)  \\
 b& \longrightarrow & \left( \begin{array}{cc}0&0 \\ 1&1 \end{array}\right)   \\

\end{array}\right. .$$
Es fácil ver que, por ejemplo, que 
$$\begin{array}{ccc}  
\left( \begin{array}{cc}0&1 \\1&0 \end{array}\right)   \left( \begin{array}{cc}1&1 \\0&0 \end{array}\right) &=& \left( \begin{array}{cc}0&0 \\1&1 \end{array}\right)\\
 \varphi((1\  2)) \varphi(a)&=&\varphi((1\ 2)\circ a)=\varphi(b).  \end{array}$$
Mediante cálculos sencillos se puede notar que estas igualdades se cumplen para cualquier producto en la tabla de Cayley de $Trans(X)$. 
Entonces podemos decir que $Trans(X)$ es un submonoide de $M_{2\times 2}$ y lo denotamos 
$$Trans(X) \leq M_{2\times 2},$$ a pesar de que son dos conjuntos completamente distintos.

\end{ejemplo}

\begin{teorema}{(Lema de Cayley)}\label{lema de cayley}

Todo monoide $M$ es isomorfo a un submonoide del monoide de transformaciones de algún conjunto. Es decir, existe un conjunto $X$ tal que $$M \leq Trans(X).$$

\end{teorema}

La prueba del Lema de Cayley requiere del concepto de \emph{acción de grupo}. Por lo que la demostración será incluida en la Sección \emph{Acciones y grupos}.

\subsubsection{Monoides finitos}
Abordaremos en esta sección un caso particular del monoide de transformaciones, el caso donde $X$ es un conjunto finito.\\

Dado un conjunto finito $X$, el conjunto de transformaciones de $X$ en $X$, $Trans(X)$, es también un conjunto finito. Específicamente si se tiene que $|X|=n$, entonces debe ocurrir que $|Trans(X)|=n^{n}$. Además se tiene que $|PT(X)|=(n+1)^{n}$.

\begin{ejemplo}\label{eje}
Sea $X=\{a,b\}$. Tenemos que los únicos elementos en $Trans(X)$ son:
$$\begin{array}{cc}   

f_{1}=\left\{ \begin{array}{ccc} a \rightarrow a \\   b\rightarrow a \end{array} \right.
& 
f_{2}=\left\{ \begin{array}{ccc} a \rightarrow a \\   b\rightarrow b \end{array} \right. 
\\

f_{3}=\left\{ \begin{array}{ccc} a \rightarrow b \\   b\rightarrow a \end{array} \right.
 &
 
  f_{4}=\left\{ \begin{array}{ccc} a \rightarrow b \\   b\rightarrow b \end{array} \right.

\end{array}$$

\end{ejemplo}

\begin{definicion}{(Ideal)}

Decimos que un submonoide $I$ de un monoide $M$ es un ideal de $M$ si para cualquier $a\in I$ se tiene que $xa\in  I$, para cualquier $x\in M$.\\
\end{definicion}

Presentamos el siguiente lema, el cual es un resultado conocido en el estudio del monoide de transformaciones, como herramienta para demostrar el teorema posterior.

\begin{lema}\label{asd}
Sea $X$ un conjunto finito y $f: X\rightarrow X$ una transformación en $Trans(X)$. Entonces $f$ es inyectiva si y solo si es sobreyectiva.
\end{lema}

\begin{teorema}
Sea $M$ un monoide finito, entonces el conjunto de no-unidades de $M$ es un ideal.
\end{teorema}
\begin{proof}
Sea $f:X \rightarrow X$ una no-unidad y sea $\sigma$ una unidad, entonces tenemos que verificar dos casos: si $f$ no es ni sobreyectiva o si no es inyectiva.\\

Caso 1. Supongamos que $f$ no es inyectiva. Entonces existen $x,y\in X$, distintos, tales que $f(x)=f(y)$, por lo tanto tenemos que $\sigma \circ f(x)=\sigma \circ f (y)$, por lo que $\sigma \circ f$ no es inyectiva, por lo tanto no es unidad. \\

Caso 2. Supongamos que $f$ no es sobreyectiva existe un $t\in X$ tal que $f(x)\neq t$, para todo $x\in X$. Entonces $\sigma(t)$ satisface que  $f\circ \sigma (x) \neq \sigma(t)$ para todo $x\in X$. 
\end{proof}
 Usaremos los datos del Ejemplo \ref{eje} para ejemplificar el Lema \ref{asd}.

\begin{ejemplo}
Sea $X=\{a,b\}$, tenemos:
$$\begin{array}{cc}   

f_{1}=\left\{ \begin{array}{ccc} a \rightarrow a \\   b\rightarrow a \end{array} \right.
& 
f_{2}=\left\{ \begin{array}{ccc} a \rightarrow a \\   b\rightarrow b \end{array} \right. 
\\

f_{3}=\left\{ \begin{array}{ccc} a \rightarrow b \\   b\rightarrow a \end{array} \right.
 &
 
  f_{4}=\left\{ \begin{array}{ccc} a \rightarrow b \\   b\rightarrow b \end{array} \right.

\end{array}.$$

De teoría básica de funciones, notamos que las únicas no-unidades son $f_{1}$ y $f_{4}$, pues $f_{2}$ y $f_{3}$ son unidades; $f_{2}$ es directamente la identidad y $f_{3}$ es su propio inverso. Queremos ver que no importa con quien compongamos a las no-unidades, estas seguirán siendo no-unidades. A continuación mostramos lo resultados de todas las posibles composiciones de parejas de elementos:

$$\begin{array}{c|c|c|c}  f_{1}\circ f_{1}=f_{1} &
f_{1}\circ f_{2}=f_{1} &
f_{1}\circ f_{3}=f_{1} &
f_{1}\circ f_{4}= f_{1}\\
f_{4}\circ f_{1}= f_{4}&
f_{4}\circ f_{2}= f_{4}&
f_{4}\circ f_{3}= f_{4}&
f_{4}\circ f_{4}=f_{4} 
\end{array}$$
Incluso si realizamos la composición por el otro lado:
$$\begin{array}{c|c|c|c}  f_{1}\circ f_{1}=f_{1} &
f_{2}\circ f_{1}=f_{1} &
f_{3}\circ f_{1}=f_{4} &
f_{4}\circ f_{1}= f_{4}\\
f_{1}\circ f_{4}= f_{1}&
f_{2}\circ f_{4}= f_{4}&
f_{3}\circ f_{4}= f_{1}&
f_{4}\circ f_{4}=f_{4} 
\end{array}$$

\end{ejemplo}

 Presentamos un contraejemplo del teorema para el caso donde $M$ no es un monoide finito.

\begin{ejemplo}
Sea $M=Trans(\mathbb{Z})$, tomamos los elementos:

$$\begin{array}{cc}

f(n)=\left\{\begin{array}{cc} n+1 & ,n>0 \\  n & ,n\leq 0\end{array}    \right.,    &    g(n)=\left\{\begin{array}{cc} 1-n & ,n>0 \\  -n & ,n\leq 0\end{array}    \right. 

\end{array}.$$

Es fácil ver que $f(n)$ no es una unidad, no es invertible, pues no es sobreyectiva. No existe $z\in \mathbb{Z}$ tal que $f(z)=1$. Del mismo modo se puede notar que $g(n)$ tampoco es una unidad, pues no es inyectiva, pues se tiene que $g(0)=g(1)$. Dicho de otro modo, no existen $f^{-1}$ ni $g^{-1}$. Sin embargo, se puede verificar fácilmente que $$g \circ f(n)=-n, $$ donde $h(n)=-n$ es una unidad de $Trans(\mathbb{Z})$, es decir, el producto de dos no-unidades da como resultado una unidad.

\end{ejemplo}

\subsubsection{Rank de monoides}

En esta sección definimos los conceptos de \emph{rank} y \emph{rank relativo}, los cuales son la base de nuestro trabajo. \\

Sea $T$ un subconjunto de un monoide $M$, definimos a $\langle T \rangle$ como el monoide más pequeño que contiene a $T$. Es fácil verificar que se satisfacen las siguientes propiedades:
\begin{enumerate}
\item[a)] $\langle T \rangle$ es la intersección de todos los submonoides de $M$ que contienen a $T$. 
\item[b)] Se cumple que:
$$\langle T \rangle= \{t_{1}t_{2}\cdots t_{k}:\ t_{i}\in T,\ k\in \mathbb{N}\}.$$
\end{enumerate}
Si se tiene que $\langle T \rangle =M$ decimos que $T$ genera a $M$, o que es un conjunto generador.

\begin{definicion}{(Rank)}

Dado un monoide $M$ definimos su rank como:
$$rank(M)=\min\{|T|:\ \langle T \rangle = M\}.$$

\end{definicion}

 Los siguientes son ejemplos bien estudiados del concepto de rank de un monoide. Algunos se pueden encontrar desarrollados en la sección de apéndices del presente documento.

\begin{ejemplo}[El rank de $Trans(X)$]
Sea $X$ un conjunto finito. Entonces se tiene que $rank(Trans(X))=3$. Además se tiene que $rank(Sym(X))=2$.

\end{ejemplo}

El rank de monoides no se comporta de manera agradable cuando tomamos submonoides, o subgrupos; en otras palabras, si $K$ es un submonoide de $M$, puede no haber relación entre el $rank(K)$ y el $rank(M)$.

\begin{ejemplo}[El cubo de rubik]
El cubo de rubik es un polígono de 6 lados (cubo) seccionados en 9 secciones congruentes cada lado. Es decir que es una figura con 54 elementos, los cuales, a grandes rasgos, es posible permutar. Es decir hablamos del conjunto de simetrías de 54 elementos, $S_{54}$.\\
\begin{figure}[ht]
\centering
\includegraphics[width=2in]{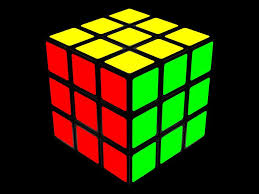}
\end{figure}

Sin embargo, el cubo de rubik es conocido por tener la limitación que solo pueden utilizarse las permutaciones ilustradas en la imagen. \\
\begin{figure}[ht]
\centering
\includegraphics[width=4in]{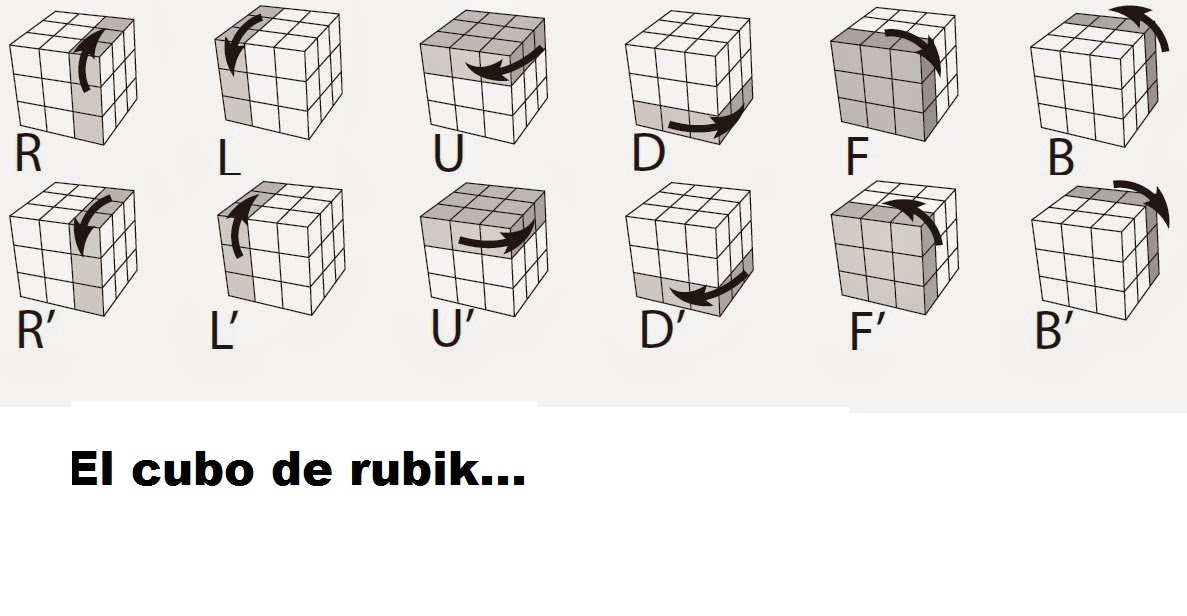}
\end{figure}

Para los que no lo han trabajado, se sabe que no todas las permutaciones existentes son posibles dentro del cubo de rubik. Por ejemplo, los colores verde, rojo y amarillo deben de permanecer siempre juntos para la pieza mostrada en la esquina superior al frente, o el rojo y el verde deben mantenerse juntos para la pieza del medio. En general, para ser más ilustrativos, hay una pieza de dos lados verde con rojo y una la cual tiene verde y naranja, es decir que una combinación de verde con azul o verde con naranja es completamente imposible. Esto nos muestra que el cubo de rubik es un subconjunto propio de $S_{54}$.\\

Sin embargo, es fácil ver que la idea del cubo de rubik es que es generado por las permutaciones 
$$Rubik=\langle\{R,R',L,L',U,U',D,D',F,F',B,B'\}\rangle,$$
por lo cual es un submonoide de $S_{54}$. Por su construcción podemos decir que 
$$rank(Rubik)\leq 12.$$
Más aún, es fácil identificar que $$R'=R^{3},$$
lo cual se satisface para cada par de elementos en el conjunto generador, por lo que entonces también se satisface la siguiente igualdad:
$$Rubik=\langle\{R,L,U,D,F,B\}\rangle,$$
lo cual indica que $$rank(Rubik)\leq 6.$$
Si numeramos los cuadros del cubo, podemos ver que cada elemento del conjunto generador deja fijos una cantidad específica de ellos, 33 para ser exactos, En particular existen por lo menos 2 movimientos, digamos $M$ y $N$, que dejan fijo a un elemento seleccionado $\alpha$, i.e., $M(\alpha)=N(\alpha)=\alpha$. Si tomamos un movimiento que no fije el elemento seleccionado $\alpha$, digamos $K$,i.e., $K(\alpha) \neq \alpha$, es imposible generar a $K$ como combinación de $M$ y $N$, i.e., $K\notin \langle M,N \rangle$. Esto significa que 2 elementos del conjunto generador no son suficientes para generar a todo el submonoide \emph{rubik}. Y por lo tanto tenemos que 

$$2<rank(Rubik).$$
Sin embargo, por el ejemplo anterior, sabemos que $rank(S_{54})=2$. Lo que nos da un contra-ejemplo para decir que el rank de un monoide no debe conservarse a cualquier submonoide.\\

\end{ejemplo}

El siguiente lema nos genera herramientas elementales para acotar el rank en algunos casos.

\begin{lema}
Sean $G$ y $H$ grupos y sea $N$ un subgrupo normal de $G$. Entonces:
\begin{enumerate}
\item[i)] $rank(G/N)\leq rank(G)$.
\item[ii)] $rank(G\times H)\leq rank(G)+rank(H)$.
\item[iii)] $rank(G\wr S_{n})\leq rank(G)+rank(S_{n})$, $\forall n\geq 1$ .
\item[iv)] $rank(\mathbb{Z}_{d}\wr S_{n})=2 $, $\forall n,d \geq 2$.

\end{enumerate}

\end{lema}

 Los pormenores y demostración de este lema pueden consultarse en documentos como \cite{cas8}.\\

Otro caso posible al momento de tratar de determinar el rank de un monoide es tener un monoide $M$ el cual es finitamente generado, pero algún submonoide $K$ no lo sea.

\begin{ejemplo}[El conmutador de un grupo libre]
Dado el grupo libre generado por dos símbolos $F_{2}$, digamos $F_{2}=\langle a,b\rangle$, podemos tomar el subgrupo
$$S=\langle \{b^{n}ab^{-n}:n\in \mathbb{N}\}\rangle.$$
Por su construcción, es fácil ver que $S\leq F_{2}$ y es un subgrupo generado por una cantidad infinita de elementos.
\end{ejemplo}

Dados dos subconjuntos $A$ y $U$ de un monoide $M$, decimos que $U$ genera a un monoide $M$ módulo $A$ si se satisface que $$M=\langle U \cup A \rangle.$$
Definimos pues el concepto de \emph{rank relativo}.

\begin{definicion}{(Rank relativo)}

Sea $M$ un monoide y $U$ un subconjunto de $M$, definimos el \emph{rank relativo} de $U$ sobre $M$ como:
$$rank(M:U)=min\left\{|V|:\ M=\langle U\cup V \rangle\right\}.$$
\end{definicion}
 Y recordemos que el conjunto de unidades es el conjunto de todos los elementos invertibles de $M$. Entonces presentamos el siguiente resultado.

\begin{teorema}[Rank relativo]
Sea $M$ un monoide y $U$ su conjunto de unidades, entonces 
$$rank(M)\leq rank(U)+rank(M:U).$$
\end{teorema}

Este teorema es bastante abordado en varias subramas de la teoría de grupos y se ha probado la igualdad para varios casos especiales, ver por ejemplo \cite{Ara}, \cite{caslineal} o \cite{finite monoids}, entre otros. Este resultado genera una herramienta importante para encontrar el rank de un monoide. 

\subsection{Acciones de grupos}

Esta sección es un repaso básico de dos conceptos conocidos en cualquier rama de las matemáticas: \emph{acción y grupo}. Nuestros objetos principales de estudio, los endomorfismos de $G$-conjuntos, se construyen sobre estos conceptos.

\begin{definicion}{(Grupo)}

Un grupo es un conjunto $G$, equipado con una operación binaria $\cdot:G\times G \rightarrow G$, la cual satisface las siguientes propiedades:
\begin{enumerate}
\item[i)] \emph{Asociativa} $$(g_{1}\cdot g_{2}) \cdot g_{3} = g_{1} \cdot (g_{2} \cdot g_{3}),\ \forall g_{1},g_{2},g_{3}\in G.$$
\item[ii)] \emph{Elemento neutro} $$\exists e\in G, e\cdot g= g\cdot e =g,\ \forall g\in G.$$

\item[iii)] \emph{Inversos} $$\forall g\in G, \exists g^{-1}\in G,\ g\cdot g^{-1}=g^{1}\cdot g= e.$$
\end{enumerate}

\end{definicion}

 Notemos que la diferencia entre un grupo y un monoide es que dentro de un grupo todos sus elementos deben de ser unidades, es decir deben de tener un inverso.

\begin{ejemplo}
El conjunto de simetrías de un conjunto, $Sym(X)$, es un grupo, bajo la composición de funciones.
\end{ejemplo}
Esto se sigue por definición. Como $Sym(X)$ está definido como el conjunto de unidades del monoide $Trans(X)$, por lo que las propiedades de grupos se cumplen automáticamente. \\

Dado un elemento $h \in G$ definimos el conjugado de $h$ por $g\in G$ como el elemento $g^{-1}hg$, decimos que dos elementos $h,k\in G$ son conjugados si existe un $g\in G$ tal que $h=g^{-1}kg$. Y dado un subconjunto $K$ de $G$ definimos el conjunto de elementos conjugados de $K$ como:
$$g^{-1}Kg:=\{g^{-1}kg:\ g\in G,\ k\in K\}.$$

\begin{proposicion}
Si $H$ es un subgrupo de $G$, entonces $g^{-1}Hg$ es un subgrupo de $G$, para cualquier $g\in G$.
\end{proposicion}

\begin{proof}
Se tiene que $e\in H$ y es fácil ver que $e=g^{-1}eg \in g^{-1}Hg$.\\
Además dado $h\in H$, se tiene que $h^{-1}\in H$, entonces $(g^{-1}hg)(g^{-1}h^{-1}g)=e$, por lo tanto $(g^{-1}hg)^{-1}=g^{-1}h^{-1}g\in g^{-1}Hg$.
\end{proof}

\begin{definicion}{(Normalizador de $H$ en $G$)}

Sea $G$ un grupo y $H$ un subconjunto de $G$. Entonces el normalizador de $H$ en $G$ se define como el conjunto 
$$N_{G}(H):=\{g\in G:\ gHg^{-1}=H\}.$$
\end{definicion}

\begin{lema}
$N_{G}(H)$ es un subgrupo de $G$.
\end{lema}
\begin{proof}
Realizamos esta demostración con el test de subgrupo. Se tiene que si $h\in N_{G}(H)$, se cumple que  $H=eHe=h^{-1}hHh^{-1}h=h^{-1}Hh$, entonces $h^{-1}\in N_{G}(H)$. Además si $g,h\in N_{g}(H)$ tenemos que $$\begin{array}{rl} (gh^{-1})H(gh^{-1})^{-1}&=ghHh^{-1}g^{-1}\\ &= gHg^{-1} \\ &=H \end{array},$$ entonces $gh^{-1}\in N_{G}(H)$.
\end{proof}

\begin{lema}
Si $H\leq G$, entonces $H \unlhd N_{G}(H)$.
\end{lema}

Recordemos que un subgrupo es normal ($\unlhd$) en un grupo si la conjugación por todo elemento en el grupo deja invariante al subgrupo, lo cual se cumple por definición del normalizador.

\begin{definicion}{(Clase de conjugación)}

Para un elemento $h$ o para un subgrupo $H$ de un grupo $G$ definimos su clase de conjugación como:
$$[h]=\{ghg^{-1}: g\in G\}$$
$$[H]=\{gHg^{-1}: g\in G\}.$$
\end{definicion}

\begin{lema}
El conjunto de clases de conjungación de elementos de $G$, 
$$\{ [h]:\ h\in G\},$$
forma una partición para $G$.
\end{lema}

\begin{lema}
El conjunto de clases de conjugación de subgrupos de $G$, 
$$\{[H]:H\leq G\},$$
es una partición para el conjunto de subgrupos de $G$.
\end{lema}

 Estos lemas se siguen del hecho que la relación de conjugación es una relación de equivalencia:
 \begin{enumerate}
\item[] (Reflexiva) $H=e^{-1}He$, $\forall H\leq G$. 
\item[] (Simétrica) Si $H_{1}=g^{-1}H_{2}g$, entonces se tiene que $H_{2}=gH_{1}g^{-1}$. 
\item[] (Transitiva) Si $H_{1}=g^{-1}H_{2}g$ y $H_{2}=h^{-1}H_{3}h$, entonces $H_{1}=(gh)^{-1}H_{3}(gh)$. 
\end{enumerate}

Denotamos por $Conj(G)$ el conjunto de clases de conjugación de todos los subgrupos de $G$,
$$Conj(G):=\{ [H]:\ H\leq G \}.$$
Para el caso en el que $G$ es un grupo finito podemos definir un orden parcial en $Conj(G)$ como sigue: dados $[H_{1}],[H_{2}]\in Conj(G)$,
$$[H_{1}]\leq [H_{2}] \mbox{ si y solo si existe } g\in G \mbox{ tal que } H_{1}\subseteq gH_{2}g^{-1}.$$
Debemos tomar en consideración el grafo de este orden parcial $(Conj(G),\varepsilon_{G})$, donde
$$\varepsilon_{G}:=\{([H_{i}],[H_{j}])\in {Conj(G)}^{2}:\ [H_{i}]\leq [H_{j}]\},$$
pues este objeto será utilizado en secciones posteriores.

\begin{definicion}{(Homomorfismo)}

Dados dos grupos $G$ y $H$, un homomorfismo es una función $\varphi:G\rightarrow H$ tal que 
$$\varphi(g_{1}\cdot g_{2})=\varphi(g_{1})\cdot\varphi(g_{2}),\ \forall g_{1},g_{2}\in G.$$
\end{definicion}

Recordando que un \emph{isomorfismo} es un homomorfismo $\varphi$ biyectivo, el cual cumple que su inversa es también un homomorfismo.   Podemos definir entonces dos conjuntos especiales dentro del monoide de transformaciones.

$$End(G):=\{\varphi:G\rightarrow G| \ \varphi(gh)=\varphi(g)\varphi(h), \ \forall g,h\in G\},$$
el conjunto de homomorfismos de $G$ en $G$ y 
$$Aut(G):=\{\varphi\in End(G)|\  biyectivo\}=End(G)\cap Sym(G).$$
Notemos que $$Aut(G) \leq End(G) \leq Trans(G),$$ pues si tenemos dos homomorfismos $\phi,\psi:G \rightarrow G$, tenemos que 
$$\begin{array}{ccc}
\phi \circ \psi (g\cdot h)&=& \phi(\psi(g\cdot h))\\
&=& \phi(\psi(g)\cdot\psi(h))\\
&=& \phi(\psi(g))\cdot\phi(\psi(h))\\
&=&\phi\circ\psi(g)\cdot\phi\circ\psi(h).

\end{array}$$

Entonces $End(G)$ es cerrado bajo composición. Y de igual manera, la composición de funciones biyectivas da una función biyectiva, entonces $Aut(G)$ es también cerrado bajo composiciones, es decir, ambos son submonoides de $Trans(G)$.\\
 Notemos que las definiciones de $End(G)$ y $Aut(G)$ tienen sentido únicamente por la estructura de grupo que $G$ posee.

\begin{definicion}{(Acción)}

Dado un grupo $G$ y un conjunto $X$ definimos una \emph{acción} como una aplicación $\cdot: G\times X \rightarrow X$, la cual cumple las siguientes propiedades:
\begin{enumerate}
\item[i)] $e\cdot x=x$, para todo $x\in X$.
\item[ii)] $(g_{1}g_{2})\cdot x=g_{1}\cdot(g_{2}\cdot x)$, $\forall g_{1},g_{2}\in G,\ \forall x\in X$.
\end{enumerate}
\end{definicion}

\begin{ejemplo}{(La acción Shift})

Sea $G$ un grupo y $A$ un conjunto, denotamos por $A^{G}$ al conjunto de funciones que van de $G$ en $A$,
$$A^{G}:=\{x:G\rightarrow A\}.$$
Definimos una acción de $G$ sobre $A^{G}$ como:
$$g\cdot x (t)=x(g^{-1}t),\ \forall g,t\in G.$$
\end{ejemplo}

Sea $G$ un grupo que actúa sobre un conjunto $X$, es posible definir, para cada elemento $g\in G$, una función $\varphi_{g}:X\rightarrow X$ dada por:
$$\varphi_{g}(x)=g\cdot x.$$
Esta función recibe el nombre de \emph{función de acción} del elemento $g$.\\
Notemos que el concepto de acción puede ser también definido sobre un monoide $M$. Entonces con estas herramientas podemos presentar una demostración para el Lema de Cayley.

\begin{proof}[Demostración. (Lema de Cayley)]
Dado un grupo $G$ (puede ser un monoide) y un conjunto no vacío $X'$. Tomamos un elemento $x\in X'$ sobre el cual $G$ actúa y definimos el siguiente conjunto:
$$X:=\{g\cdot x |\ g\in G\}.$$
Entonces $X$ es un conjunto no vacío. Es fácil ver que podemos definir una acción de $G$ sobre $X$ dada por
$$h\cdot (g\cdot x)= hg\cdot x.$$
Y definimos la siguiente función $\varphi:G \rightarrow Trans(X)$:

$$\varphi(g)=\varphi_{g},\ \mbox{tal que}\  \varphi_{g}(y)=g\cdot y, \ \forall y\in X.$$
La cual asigna a cada elemento $g$ su función de acción. Es fácil ver que $\varphi$ es un homomorfismo entre $G$ y $Trans(X)$. 
Sean $g,h\in G$ tenemos que 
$$\begin{array}{cl}  \varphi(gh)(y)&=(gh)\cdot y \\
& = g\cdot (h\cdot y)    \\
& =\varphi_{g}(\varphi_{h}(y)) \\
&= \varphi_{g} \circ \varphi_{h}(y)\\
&=\varphi(g) \circ \varphi(h) (y),
\end{array}$$
por tanto $\varphi(gh)=\varphi(g) \circ \varphi(h)$. 

Si $\varphi(g)=\varphi(h)$, para cualesquiera $g,h\in G$, se tiene que $g\cdot y = h\cdot y$ para todo $y\in X$, lo cual implica que $g=h$.
Por consecuencia podemos afirmar que $$G \leq Trans(X).$$
\end{proof}

Recalcamos que la función que asigna a cada elemento en un grupo su función de acción es un homomorfismo entre el grupo y el monoide de transformaciones del conjunto sobre el que actúa. En particular, a continuación probamos que si tomamos un grupo $G$ que actúa sobre sí mismo por conjugación, entonces la función de acción es un isomorfismo de $G$ con él mismo, es decir, es parte del conjunto de automorfismos de $G$.

\begin{lema}
Sea $G$ un grupo, el cual actúa sobre sí mismo por conjugación,
$$g\cdot h= g^{-1}hg,$$
entonces la función de acción $\varphi_{g}$ es un isomorfismo de $G$ en $G$.
\end{lema}

\begin{proof}
Primero se debe probar que $\varphi_{g}$ es un homomorfismo. Sean $h,k\in G$ tenemos que :
$$\begin{array}{ccl}
\varphi_{g}(hk)&=&g^{-1}(hk)g\\
&=&g^{-1}h(gg^{-1})kg\\
&=&(g^{-1}hg)(g^{-1}kg) \\
&=&\varphi_{g}(h)\varphi_{g}(k).

\end{array}$$
También debemos probar que es biyectiva. Supongamos que se tiene que $\varphi_{g}(h)=\varphi_{g}(k)$, entonces:
$$\begin{array}{ccl} 
\varphi_{g}(h)&=&\varphi_{g}(k)\\
g^{-1}hg &=& g^{-1}kg\\
\cancel{gg^{-1}}hg&=&\cancel{gg^{-1}}kg\\
h\cancel{gg^{-1}}&=&k\cancel{gg^{-1}}\\
h&=&k.
\end{array}$$
Entonces $\varphi_{g}$ es inyectiva. Además notemos que para cada $h\in G$, existe $k=ghg^{-1}\in G$ tal que 
$$\varphi_{g}(k)=g^{-1}(ghg^{-1})g=h,$$ por lo que $\varphi_{g}$ es sobreyectiva. Entonces $\varphi_{g}$ es biyectiva y por tanto es un isomorfismo.

\end{proof}

\subsubsection{La categoría de $G$-conjuntos}
En esta sección definimos una categoría matemática de especial interés en nuestro trabajo, la categoría de $G$-conjuntos.  Se definen todos los conceptos y se prueban algunos resultados que ésta incluye los cuales son de suma importancia para el desarrollo de nuestro trabajo.

\begin{definicion}{($G$-conjunto)}

Un $G$-conjunto es un conjunto $X$ sobre el cual $G$ actúa, es decir que existe una acción de $G$ sobre $X$. 
\end{definicion}

Los $G$-conjuntos son una categoría sumamente estudiada, en la cual adentramos nuestra investigación. Definimos pues los morfismos dentro de esta categoría.

\begin{definicion}{(Función $G$-equivariante)}

Dado un grupo $G$ y dos conjuntos $X$ y $Y$ sobre los que $G$ actúa. Una función $f:X\rightarrow Y$ se dice $G$-equivariante si se tiene que $$g\cdot f(x)=f(g\cdot x),\ \forall g\in G,\ \forall x\in X.$$
\end{definicion}

\begin{definicion}{(Acciones equivalentes)}

Dados dos conjuntos $X$ y $Y$ sobre los que $G$ actúa. Decimos que la acción de $G$ sobre $X$ es \emph{equivalente} a la acción de $G$ sobre $Y$ si existe una función $\varphi:X\rightarrow Y$ biyectiva de $X$ en $Y$ tal que:
$$\varphi(g\cdot x)=g\cdot \varphi(x),\ \forall g\in G,\   \forall x\in X.$$

\end{definicion}

Las funciones $G$-equivariantes son conocidas como los morfismos en esta categoría. Entonces definimos el conjunto de endomorfismos de un $G$-conjunto $X$ como:
$$End_{G}(X):=\{f:X\rightarrow X |\ g\cdot f(x)=f(g\cdot x),\ \forall g\in G,\ \forall x\in X\}.$$
Si además tenemos que $f$ es una biyección decimos que es un \emph{automorfismo}:
$$Aut_{G}(X):=End(X) \cap Sym(X).$$

\begin{definicion}{(Isomorfismo)}

Un isomorfismo entre dos $G$-conjuntos $X$ y $Y$ es una función $G$-equivariante la cual es biyectiva.
\end{definicion}
Dos $G$-conjuntos $X$ y $Y$ se dicen \emph{isomorfos} si existe un isomorfismo entre ellos. Además son isomorfos si y solo si la acción es equivalente en ambos.\\

Si una función $G$-equivariante tiene inversa, ésta es también una función $G$-equivariante, pues para $\lambda\in Aut_{G}(X)$ y para todo $y\in X$ existe $\lambda^{-1}\in End_{G}(X)$ tal que si denotamos por $y=\lambda(x)$, $x\in X$, se satisface que
$$\lambda^{-1}(g\cdot y)=\lambda^{-1}(g\cdot \lambda(x))=\cancel{\lambda}^{-1}(\cancel{\lambda}(g\cdot x))=g\cdot x= g\cdot \lambda^{-1}(y).$$

\subsubsection{Estabilizadores y órbitas}
Dado un grupo $G$ que actúa sobre un conjunto $X$, definimos para cada elemento $x\in X$ los siguientes conjuntos.

\begin{definicion}{(Estabilizador)}

$$G_{x}:=\{g\in G: \ g\cdot x=x\}.$$
\end{definicion}

\begin{definicion}{(Órbita)}
$$Gx:=\{g\cdot x:\  g\in G\}.$$
\end{definicion}

 En esta tesis hablaremos de $G$-órbitas y las $Aut_{G}(X)$-órbitas, las cuales son las órbitas de elementos en $X$ donde los elementos de $G$ actúan y en donde los elementos de $Aut_{G}(X)$ actúan respectivamente.\\
Y también mencionamos que existen distintas notaciones para estos conjuntos, por ejemplo $G_{x}=sta(x)$ y $Gx=orb(x)$.\\

\begin{proposicion}

Sea $G$ un grupo que actúa sobre un conjunto $X$. 
\begin{enumerate}

\item[i)] Los estabilizadores $G_{x}$ son subgrupos de $G$, $\forall x \in X$.

\item[ii)] Se tiene que las órbitas forman una partición para $X$.

\end{enumerate}
\end{proposicion}

\begin{proof}

\begin{enumerate}
\item[i)]Si $g\in G_{x}$, entonces $g\cdot x=x$ y actuando por $g^{-1}$ tenemos que  $g^{-1}\cdot x=g^{-1}\cdot(g\cdot x)=(g^{-1}g)\cdot x=e\cdot x=x$, por lo que $g^{-1}\in G_{x}$.\\ Entonces por el test de subgrupo, dados $g,h\in G_{x}$ tenemos que 
$$(g^{-1}h)\cdot x= g^{-1}\cdot(h\cdot x)=g^{-1}\cdot x=x,$$
entonces $g^{-1}h\in G_{x}$ y por lo tanto $G_{x}$ es un subgrupo de $G$.

\item[ii)] Es fácil ver que la relación $x\sim y$ si y solo si $\exists g\in G$ tal que $x=g\cdot y$ es una relación de equivalencia. \\
$e\cdot x=x$ implica que $x\sim x$. (Reflexiva)\\
Como $g\cdot y=x$, al accionar por $g^{-1}$ tenemos que, $g^{-1}\cdot (g \cdot y)=y=g^{-1}\cdot x$ y se tiene que $y \sim x$. (Simétrica)\\
Además si $x=g\cdot y$ y $y=h\cdot z$ se tiene que $x=g\cdot (h\cdot z)=(gh)\cdot z$, por lo que $x\sim z$. (Transitiva)\\
Entonces las clases de equivalencia $$\overline{x}:=\{y\in X,\ x\sim y\}$$ forman una partición para $X$.\\ Además es fácil visualizar que $$Gx=\overline{x}.$$

\end{enumerate}

\end{proof}

Definimos otros conjuntos los cuales son de gran importancia para nuestro trabajo, los conjuntos de puntos fijos:

$$X^{g}=Fix(g):=\{x\in X:\ g \cdot x= x\}$$
$$Fix(H):=\{x\in X: h\cdot x=x,\ \forall h\in H\}.$$

Ahora, para cada subgrupo $H$ de $G$, donde $G$ actúa sobre $X$, definimos los siguientes conjuntos a los cuales denominamos \emph{cajas}:
$$\mathcal{B}_{H}:=\{x\in X:\ H=G_{x}\}.$$
A partir de la relación de equivalencia que definimos sobre $Conj(G)$ podemos definir también los siguientes conjuntos, los cuales denominamos cajas:

$$\mathcal{B}_{[H]}:=\{x\in X:\ [G_{x}]=[H]\}.$$

Tenemos algunas propiedades a partir de estos conjuntos.

\begin{proposicion}
Dados $x\in X$ y $K,H\leq G$, tales que $K\in [H]$. Entonces:
\begin{enumerate}
\item[i)] $\mathcal{B}_{K} \subseteq \mathcal{B}_{[H]}$.
\item[ii)] $\mathcal{B}_{K}=\mathcal{B}_{H}$ ó $\mathcal{B}_{K}\cap \mathcal{B}_{H}=\emptyset$.
\item[iii)] $\mathcal{B}_{G_{x}}= Aut_{G}(x)$. (La órbita de $x$ bajo la acción de $Aut_{G}(X)$)
\end{enumerate}
\end{proposicion}

La demostración de las afirmaciones i) y ii) se sigue de la definición de estos conjuntos, para la afirmación iii) se presentará en la siguiente sección utilizando herramientas que facilitan su demostración.

\begin{teorema}{(Órbita-estabilizador)}\label{orbest}

Dado un grupo $G$ que actúa sobre un conjunto $X$. Para cualquier elemento $x\in X$ se cumple que:
$$|Gx|=|G/G_{x}|.$$
\end{teorema}

\begin{proof}
Cualquier elemento de $Gx$ puede verse como $g\cdot x$ para algún $g \in G$. Definimos una asignación biyectiva de $g\cdot x$ a $gG_{x}$ o bien, $\phi(g\cdot x) := gG_{x}$, la cual está bien definida y es una biyección: 

\begin{enumerate}
\item[1.] $\phi$ está bien definida: \\ Sean $g\cdot x , h \cdot x \in Gx$ tales que $g\cdot x = h\cdot x$. Entonces $h^{-1} g\cdot x= x$ y $h^{-1}g \in G_{x}$. Esto implica que $gG_{x}=hG_{x}$ y $\phi(g\cdot x)= \phi(h \cdot x)$.
\item[2.] $\phi$ es inyectiva: \\ Supongamos que $\phi(g\cdot x)= \phi(h \cdot x)$. Entonces $gG_{x}=hG_{x}$ lo que implica que $h^{-1}g \in G_{x}$, pues $h^{-1}g \cdot x = x$ y así $g\cdot x= h\cdot x$.

\item[3.] $\phi$ es sobreyectiva: \\ Si $gG_{x}$ es una clase lateral genérica en $G/G_{x}$, es claro que su preimagen bajo $\phi$ es $g\cdot x$.

\end{enumerate}

\end{proof}

Tenemos también una proposición de utilidad para las clases de conjugación de los estabilizadores de dos elementos. 

\begin{proposicion}\label{equivalencia}
Sea $G$ un grupo que actúa sobre un conjunto $X$ y dados 2 elementos $x,y$ en $X$, entonces se tiene que la acción de $G$ es equivalente sobre $Gx$ y sobre $Gy$ si y solo si  $[G_{x}]=[G_{y}]$.

\end{proposicion}

\begin{proof}
Supongamos que la acción de $G$ es equivalente sobre $Gx$ y $Gy$, entonces existe una función $\lambda:Gx \rightarrow Gy$, la cual es biyectiva y $G$-equivariante.
 Sea $\alpha \in G_{x}$ y dado que $\lambda(x)$ es un elemento de $Gy$, tenemos que $\lambda(x)=g\cdot y$, para algún $g\in G$ y entonces se tiene que:

$$\begin{array}{rl} 
\alpha \cdot x&=x\\
\lambda(\alpha\cdot x)&=\alpha\cdot \lambda(x)=\lambda(x)\\
\alpha \cdot (g\cdot y)&=(\alpha g) \cdot y=g\cdot y \end{array}$$
y actuando por $g^{-1}$\\
$$g^{-1}\cdot ((\alpha g)\cdot y)= g^{-1}\cdot (g\cdot y).$$
Al final obtenemos que 
$$(g^{-1}\alpha g)\cdot y=y,$$
lo que implica que $g^{-1}\alpha g \in G_{y}$. Esto significa que $g^{-1}G_{x}g \subseteq G_{y}$. Por una construcción análoga utilizando $\lambda^{-1}$ se obtiene la otra contención, por lo tanto $g^{-1}G_{x}g=G_{y}$.\\

Por el otro lado, si suponemos que $[G_{x}]=[G_{y}]$, existe un $\beta\in G$ tal que $\beta^{-1}G_{x}\beta=G_{y}$. Dado un elemento en $Gx$, es de la forma $g\cdot x$, para algún $g\in G$. Entonces definimos una función $\lambda:Gx\rightarrow Gy$ dada por $\lambda(g\cdot x)=g\beta \cdot y$. Se debe verificar que $\lambda$ está bien definida y que es una biyección $G$-equivariante.\\
Supongamos que $g_{1}\cdot x=g_{2}\cdot x$, notemos que esto implica que $g_{2}^{-1}g_{1}\cdot x=x$. y por lo tanto $g_{2}^{-1}g_{1}\in G_{x}$. Entonces al aplicarle $\lambda$ tenemos que
$$\begin{array}{rl} \lambda(g_{2}^{-1}g_{1}\cdot x)&=\lambda(x) \\
g_{2}^{-1}g_{1}\beta\cdot y&= \beta \cdot y \end{array}$$
lo cual implica que

$$\begin{array}{rl}
g_{1}\beta\cdot y&=g_{2}\beta \cdot y\\
\lambda(g_{1}\cdot x)&=\lambda(g_{2}\cdot x),
\end{array}$$ por lo tanto $\lambda $ está bien definida. Además notemos que dado un elemento $g \cdot x \in Gx$ y un $h\in G$ se tiene que:
$$\begin{array}{rl} \lambda(h\cdot(g\cdot x))&= hg\beta\cdot y\\
&=h\cdot(g\beta\cdot y)\\
&=h\cdot\lambda(g\cdot x),\end{array}$$ por lo tanto es $G$-equivariante.\\

Es fácil ver que para cualquier $g\cdot y\in Gy$ existe $g\beta^{-1}\cdot x\in Gx$ el cual cumple que $\lambda(g\beta^{-1}\cdot x)=g\cdot y$, por lo tanto $\lambda$ es sobreyectiva. Y supongamos que $\lambda(g_{1}\cdot x)=\lambda(g_{2}\cdot x)$ entonces tenemos que 
$$ g_{1}\beta \cdot y = g_{2}\beta \cdot y . $$
Esto significa que $$\beta^{-1}g_{2}^{-1}g_{1}\beta \in G_{y},$$ y como $\beta^{-1}G_{x}\beta = G_{y}$, significa que $g_{2}^{-1}g_{1}\in G_{x}$. Luego $g_{1}\cdot x= g_{2}\cdot x$ y por tanto $\lambda$ es inyectiva.
\end{proof}

Una acción se dice libre, o semiregular, o libre de puntos fijos, si dados $g,h\in G$ y si $\exists x\in X$ tal que $g\cdot x= h\cdot x$, entonces se tiene que $g=h$.\\
O equivalentemente, si dado $g\in G$, $\exists x\in X$ tal que $g\cdot x=x$, implica que $g$ es la identidad en $G$.

\begin{lema}[Cauchy-Frobenius]
Sea $G$ un grupo finito que actúa sobre $X$. Para cada $g\in G$ tomemos $X^{g}:=\{x\in X: \ g\cdot x=x\}$. Y además denotemos por $|X/G|$ al número de órbitas de $X$ dadas por la acción de $G$. Entonces 
$$|X/G|=\frac{1}{|G|}\sum_{g\in G}{|X^{g}|}.$$
\end{lema}

\begin{proof}
Notemos que 
$$\sum_{g\in G}{|X^{g}|}=|\{(g,x)\in G\times X:\ g\cdot x=x\}|= \sum_{g\in G}{|G_{x}|}.$$
Por el Teorema \ref{orbest} existe una biyección natural entre $Gx$ y $G/G_{x}$. Y por el Teorema de Lagrange:
$$|Gx|=[G:G_{x}]=\frac{|G|}{|G_{x}|},$$
ó equivalentemente
$$|G_{x}|=\frac{|G|}{|Gx|},$$
por lo que obtenemos que 
$$\sum_{x\in X}{|G_{x}|}=|G| \sum_{x\in X}{\frac{1}{|Gx|}}.$$
Notemos que $X$ es la unión disjunta de las órbitas en $X/G$, entonces la suma por $X$ se puede separar como las sumas en cada órbita. Es decir, para $Gx$ una órbita, i.e., $A\in X/G$, tenemos que 
$$\sum_{x\in X}{\frac{1}{|Gx|}}=\sum_{A\in X/G}\sum_{x\in A}{\frac{1}{|A|}}=\sum_{A\in X/G}{1}=|X/G|.$$
Entonces, sustituyendo de las expresiones anteriores tenemos que $$\sum_{g\in G}{|X^{g}|}=|G||X/G|.$$
\end{proof}
Dicho de otra forma, el Lema de Cauchy-Frobenius nos dice que el número de órbitas es el promedio de puntos fijos para cada elemento en $G$.

Para un conjunto $X$ definimos dos conjuntos especiales de subgrupos de $G$ a partir de su acción: 
$$Stab_{G}(X):=\{K\leq G| \ K=G_{x} \mbox{ para algún }x\in X\},$$
y 
$$Conj_{G}(X):=\{[K]\in Conj(G)| \ K=G_{x} \mbox{ para algún }x\in X\}.$$

\begin{proposicion}
Dado un grupo $G$ que actúa sobre un conjunto $X$ y sea $\sim$ la relación de conjugación, entonces:

$$Conj_{G}(X)=\ Stab_{G}(X)/ \sim .$$
\end{proposicion}

\subsubsection{Acciones transitivas}

En esta sección incluiremos un teorema de gran importancia sobre las acciones transitivas.

\begin{definicion}{(Acción transitiva)}

Dado un grupo $G$ que actúa sobre un conjunto $X$, la acción de $G$ sobre $X$ se dice \emph{transitiva} si solo tiene una única $G$-órbita.
\end{definicion}

A continuación presentamos un ejemplo. 

\begin{ejemplo}
Dado un grupo $G$ y un subgrupo $H$. Tenemos el conjunto de clases laterales izquierdas
$$G/H:=\{kH:\ k\in G\}.$$

El grupo $G$ actúa sobre $G/H$ por multiplicación izquierda:
$$g\cdot (kH)=gkH.$$
Esta acción es transitiva. Dadas dos clases laterales $k_{1}H,k_{2}H\in G/H$ si tenemos un elemento en la $G$-órbita de $k_{1}H$, $G_{k_{1}H}$, éste puede escribirse como $gk_{1}h$, para algún $g\in G$ y para algún $h\in H$. Notemos que:
$$gk_{1}h=gk_{1}{k_{2}}^{-1}k_{2}h=g'k_{2}h.$$
Por lo tanto pertenece a la $G$-órbita de $k_{2}H$. Por lo tanto las $G$-órbitas son iguales y en consecuencia es una única $G$-órbita.
\end{ejemplo}

\begin{teorema}
Toda acción transitiva de un grupo $G$ sobre un conjunto $X$ es equivalente a la acción de multiplicación de $G$ sobre las clases laterales $G/H$, para algún subgrupo $H\leq G$.
\end{teorema}
\begin{proof}
Fijemos un elemento $x_{0}$ en $X$. Como la acción es transitiva todo elemento en $X$ es de la forma $g\cdot x_{0}$, para algún $g\in G$.\\
Dado $H=G_{x_{0}}$ definimos una función $\varphi:G/H \rightarrow X$ dada por $\varphi(gH)=g\cdot x_{0}$.\\
Notemos  que dados dos elementos $g,g'\in G$ se satisface que:
$$g'H=gH \leftrightarrow g^{-1}g'\in H \leftrightarrow g^{-1}g'\cdot x_{0}=x_{0}\leftrightarrow g'\cdot x_{0}=g\cdot x_{0} \leftrightarrow \varphi(g'H)=\varphi(gH).$$
Esto muestra que la función $\varphi$ está bien definida y además es inyectiva. Notemos además que para cualquier elemento $h\cdot x_{0}$, con $h\in G$, existe $hH\in G/H$ tal que $\varphi(hH)=h\cdot x_{0}$, por lo tanto es sobreyectiva. \\
Dado que $G$ actúa sobre $G/H$ por multiplicación izquierda tenemos que $h\cdot (gH)=(hg)H$, para todo $g,h \in G$. Entonces 
$$h\cdot \varphi(gH)=h\cdot(g\cdot x_{0})=(hg)\cdot x_{0}=\varphi((hg)H)=\varphi(h\cdot (gH)),$$ por lo tanto es $G$-equivariante y la acción de $G$ sobre $X$ es equivalente a su acción sobre $G/H$.

\end{proof}

\subsubsection{Isomorfismos entre acciones}

En este apartado mostramos una breve diferencia sobre el concepto de isomorfismo y la equivalencia entre acciones. Diferencias importantes para la construcción de nuestro trabajo. 

\begin{definicion} {(Acciones Isomorfas)}

Sea $G$ un grupo que actúa en $\Gamma$ y sea $H$ un grupo que actúa en $\Omega$. Se dice que éstas son acciones isomorfas si existe un isomorfismo $\varphi: G \rightarrow H$ y una biyección  $\lambda:\Gamma \rightarrow \Omega$ tal que $\lambda(g\cdot x)= \varphi(g)\cdot \lambda(x) $ para toda $g\in G$, $x\in \Gamma$.
\end{definicion}

\begin{definicion}{(Acciones equivalentes)}
Sea $G$ un grupo que actúa en $\Gamma$ y sea $H$ un grupo que actúa en $\Omega$. Se dice que éstas son acciones equivalentes si $G=H$ y una biyección  $\lambda:\Gamma \rightarrow \Omega$, tal que $\lambda(g\cdot x)= g\cdot \lambda(x) $ para toda $g\in G$, $x\in \Gamma$.
\end{definicion}

Si dos acciones son equivalentes, entonces son isomorfas. Sin embargo, el converso es falso. Hay acciones isomorfas (aún con $G = H$) que no son equivalentes (en tal caso el isomorfismo $\varphi: G \rightarrow G$ es un automorfismo no trivial del grupo).

\begin{ejemplo}
Sean $G$ y $H$ subgrupos de $Sym(\Omega)$. Las acciones de $G$ y $H$ sobre $\Omega$ son isomorfas si y solo si $G$ y $H$ son subgrupos conjugados en $Sym (\Omega)$.
\end{ejemplo}

\subsection{El producto corona}

Es posible definir variaciones sobre la estructura de un grupo, cuando este puede ser construido a partir de otros grupos más pequeños. Desarrollamos en esta sección varias construcciones que nos permiten conocer la estructura de un grupo. 

\begin{definicion}{(Producto directo)}

Sean $G$ y $H$ grupos con operaciones representadas por $\cdot$ y $+$ respectivamente. Tomamos el conjunto $G\bigotimes H= G \times H$ y lo dotamos con la operación binaria:
$$(g_{1},h_{1})\otimes(g_{2},h_{2})=(g_{1}\cdot g_{2},h_{1}+h_{2})\ \forall g_{1},g_{2}\in G,\ \forall h_{1},h_{2}\in H.$$

\end{definicion}

Sean $G$ y $H$ dos grupos y dado un homomorfismo $\varphi: H \rightarrow Aut(G)$, el \emph{producto semidirecto} esta determinado por el homomorfismo $\varphi$.

\begin{definicion}{(Producto semidirecto)}

Dados dos grupos $G$ y $H$ y un homomorfismo $\varphi:H \rightarrow Aut(G)$, definimos sobre el conjunto $G\times H$ la operación
$$(g_{1},h_{1})\rtimes_{\varphi}(g_{2},h_{2})=(g_{1}\cdot \varphi_{h_{1}}(g_{2}),h_{1}+h_{2}),$$
siendo $\varphi_{h_{1}}=\varphi(h_{1})$.
\end{definicion}

Cuando $H$ y $N$ son subgrupos de un grupo $G$, siendo $N$ normal en $G$, tiene sentido determinar $\varphi_{h}(n)=h^{-1}nh$. De esta manera se cumple que $G=H\rtimes N$ y decimos que $G$ es el producto semidirecto interno de $H$ y $N$.\\

Ahora sea $H$ un grupo y $\Omega$ un conjunto en el cual actúa. i.e., $\exists \cdot : H\times \Omega \rightarrow \Omega$. Notemos que una acción induce un homomorfismo natural
$$\varphi:H\rightarrow Sym(\Omega),$$
dada por $\varphi(h)=\varphi_{h}$, donde $\varphi_{h}(x)=h\cdot x$, $x\in \Omega$.\\

Notemos también que al no ser $\Omega$ un grupo, no podemos hablar sobre automorfismos. Y ahora tomamos el producto directo de $G$ indexado por $\Omega$:
$$K=\prod_{\alpha\in \Omega}{G},$$
es decir $|\Omega|$ copias de $G$. Entonces dos elementos  $x,y\in K$ pueden visualizarse como: 
$$\begin{array}{cc}x=\{x_{\alpha}\}_{\alpha\in \Omega} & y=\{y_{\alpha}\}_{\alpha\in \Omega} \end{array},\ x_{\alpha},y_{\alpha}\in G,\ \forall \alpha \in \Omega.$$
Por ejemplo, si $\Omega= \mathbb{N}$ podríamos identificar 
$$\begin{array}{rl}
xy=\{x_{\alpha}\}_{\alpha\in \mathbb{N}}\{y_{\alpha}\}_{\alpha\in \mathbb{N}}=&\{x_{1},x_{2},\cdots\}\{y_{1},y_{2},\cdots\}  \\
 =&\{x_{1}y_{1},x_{2}y_{2},\cdots\}=\{x_{\alpha}y_{\alpha}\}_{\alpha\in \mathbb{N}}
\end{array}$$
y como $H$ actúa sobre $\Omega$ podemos extender naturalmente la acción de $H$ a todo $K$ de la siguiente manera:

$$h\cdot x=h\cdot \{x_{\alpha}\}_{\alpha\in \Omega} =\{x_{h^{-1}\cdot \alpha}\}_{\alpha\in \Omega}.$$

\begin{definicion}{(Producto corona)}

Dados dos grupos $G$ y $H$ y $\Omega$ un conjunto sobre el cual actúa $H$. Siendo $\varphi$ el homomorfismo que asigna a cada elemento en $H$ su función de acción. 
Entonces definimos el \emph{producto corona}  como 
$$G\wr H= \left(\prod_{\alpha\in \Omega}{G} \right) \rtimes_{\varphi} H,$$
cuya operación está dada por:
$$(x,h_{1})\wr(y,h_{2})=(x\varphi(h_{1})(y),h_{1}h_{2}),$$
para todo $x,y\in G^{\Omega}$ y para todo $h_{1},h_{2}\in H$.

\end{definicion}
Notamos que la operación sobre $G\wr H$ puede verse como:
$$(\{x_{\alpha}\},h)\wr(\{y_{\alpha}\},h')=(\{x_{\alpha}y_{h^{-1}\cdot \alpha}\},hh'),\ \alpha \in \Omega.$$
En el caso más común tomamos $\Omega=H$, donde la acción de $H$ en $H$ es, de manera natural, operar consigo mismo con la multiplicación por inverso por la izquierda. i.e., $\alpha \cdot h= \alpha^{-1}h$. Y de esta manera, dados dos elementos en $G\wr H$, tenemos que son de la forma 
$$(x,h_{1}),(y,h_{2})\in G\wr H,\ x=\{x_{\alpha}\}_{\alpha \in H}, y=\{y_{\alpha}\}_{\alpha \in H},\ x_{\alpha},y_{\alpha}\in G,\ \forall \alpha \in H.$$
 Y al operarlos obtenemos
 $$\begin{array}{rl}  
(x,h_{1})\wr(y,h_{2}) =&(x (h_{1}\cdot y),h_{1}h_{2})   \\
 =& (\{x_{\alpha}\}h_{1}\cdot \{y_{\alpha}\},h_{1}h_{2})  \\
 =&   (\{x_{\alpha}\}\{y_{h_{1}\alpha}\},h_{1}h_{2}) \\
 =& (\{x_{\alpha}y_{h_{1}\alpha}\},h_{1}h_{2}).

 \end{array}$$
En general, sea $\varphi:H \rightarrow Aut(G)$ un homomorfismo, la operación en $G\wr H$ puede visualizarse como:
$$(x,h_{1})\wr(y,h_{2})=(\{x_{\alpha}\},h_{1})\wr(\{y_{\alpha}\},h_{2})=(\{x_{\alpha}\varphi(h_{1})(y_{ \alpha})\},h_{1}h_{2}).$$
Podemos también identificar $\prod{G}$ como un conjunto de funciones de $\Omega$ en $G$.  Es decir que un elemento $f\in K$ puede ser visto como $f:\Omega \rightarrow G$. Entonces el producto corona se puede visualizar como
$$(f,h_{1})\wr (g,h_{2})= (f\varphi(h_{1})(g),h_{1}h_{2}).$$
Además, si el homomorfismo que da el producto semidirecto  $\varphi:H \rightarrow Aut(K)$ es la función de la acción de $H$ sobre $\Omega$, entonces
$$(f,h_{1})\wr (g,h_{2})=(fg(\varphi_{h_{1}}),h_{1}h_{2}),$$
donde $\varphi_{h_{1}}(\omega)=h_{1}\cdot \omega$ es la función de acción de $h_{1}$.

\subsubsection{Producto corona permutacional}
Sea $\Omega$ un conjunto numerable, podemos identificarlo simplemente por 
$$\Omega=\{1,2,\dots,n,\dots\}.$$ 
Además si $\Omega$ es un conjunto finito de orden $n$ podemos fácilmente visualizar que 
$$\prod_{\alpha \in \Omega}{G_{\alpha}}=G^{n}$$
y sea $H$ el grupo de simetrías de $n$ elementos, $S_{n}$, la acción natural de $S_{n}$ en $\Omega$ está dada por evaluación:  $$\sigma \cdot n=\sigma(n),$$ recordando que $\sigma$ es una permutación. Entonces podemos visualizar al producto corona como $$G\wr S_{n}=G\rtimes S_{n}=\{(v,\sigma):\ v\in G^{n},\ \sigma\in S_{n}\},$$
donde la operación puede ser vista como
$$(g_{1},g_{2},\dots,g_{n})\sigma \wr (h_{1},h_{2},\dots,h_{n})\kappa=(g_{1}h_{\sigma^{-1}(1)},g_{2}h_{\sigma^{-1}(2)},\dots,g_{n}h_{\sigma^{-1}(n)})\sigma\kappa,$$
el cual trabaja con la acción natural sobre $G^{n}$ dada por:
$$v=(g_{1},g_{2},\dots,g_{n}) \mapsto \sigma \cdot v= (g_{\sigma^{-1}(1)},g_{\sigma^{-1}(2)},\dots,g_{\sigma^{-1}(n)}),$$
es decir:
$$(v,\sigma)\wr (w,\kappa)=(v(\sigma\cdot w), \sigma \kappa).$$

\subsubsection{La acción imprimitiva de $G\wr H$}
La acción imprimitiva está definida como una acción que preserva alguna partición no trivial de un conjunto.\\

Sea $G$ un subgrupo de un grupo de permutaciones $Sym(\Gamma)$. Entonces para un grupo $H$ actuando sobre sí mismo, se tiene que se puede definir una acción de $G\wr H$ sobre $\Gamma \times H$ de la siguiente manera:
$$(s_{h_{1}},...,s_{h_{m}})h\cdot (\alpha, k)=(s_{h}(\alpha),hk).$$
También tenemos que si $H$ es un subgrupo de algún grupo de permutaciones $Sym(\Omega)$, entonces 
$G\wr H$ actúa sobre $\Gamma \times \Omega$ como:
$$(s_{h_{1}},...,s_{h_{m}})r \cdot (\alpha, \omega)=(s_{r}(\alpha),r(\omega)),$$
donde $s_{x}\in G\leq Sym(\Gamma)$, $r\in\{h_{1},...,h_{m}\}= H \leq Sym(\Omega)$, $\alpha \in \Gamma$, $\omega \in \Omega$.\\

\subsubsection{Particiones uniformes}
Dado un conjunto $X$, una partición de $X$ es una colección $\mathcal{P}$ de subconjuntos disjuntos de $X$, cuya unión es igual al conjunto completo. 
La partición se dice uniforme si $|P_{i}|=|P_{j}|$, para todo $P_{i},P_{j}\in \mathcal{P}$.
Una partición $\mathcal{P}$ se dice invariante bajo una acción $\sigma$ si se tiene que $$\sigma(\mathcal{P})=\mathcal{P}.$$
Decimos que $f:X\rightarrow X$ preserva la partición $\mathcal{P}$ si para cada $P\in  \mathcal{P}$ se tiene que 
$$f(P)\subseteq P',\ \mbox{para algún } P'\in \mathcal{P}.$$ \\

Para un conjunto de la forma $X=Y\times Z$, digamos que $Y$ tiene $n$ elementos y $Z$ tiene $m$ elementos, entonces podemos identificarlos directamente con los números del 1 al $n$ y del 1 al $m$ respectivamente. Entonces la partición natural para $X$ está dada por:
$$\mathcal{P}=\{\{(1,1),...,(n,1)\},\{(1,2),...,(n,2)\}...,\{(1,m),...,(n,m)\}\}.$$

\begin{proposicion}
Sean $G\leq Sym(\Gamma)$ y $H\leq Sym (\Omega) $. Entonces la partición natural de $\Gamma \times \Omega$ es invariante bajo la acción imprimitiva de $G\wr H$.
\end{proposicion}

\begin{proof}

Queremos probar que $x\cdot \mathcal{P}=\mathcal{P}$, para cualquier $x\in G\wr H$.  Como la acción  está definida de $G\wr H \times \mathcal{P}$ en $\mathcal{P}$, entonces la primer contención es automática: $x\cdot \mathcal{P}\subseteq \mathcal{P}$.\\
Un elemento en $P\in \mathcal{P}$ es de la forma $(a,b)$, con $a\in \Gamma$ y $b\in \Omega$. Tenemos que $x$ es de la forma $(g_{\lambda})h$, donde $\lambda\in H$.\\
Sea $P'=\{(y,h^{-1}(b)): y\in \Gamma\}\in \mathcal{P}$, se tiene que $x\cdot P'=\{(g_{h}(y)),b):\ y\in \Gamma\}$. Como $g_{h}\in Sym(\Gamma)$, se tiene que este conjunto es igual a $P$. Es decir $P=x\cdot P'$ y por consecuencia $P\in x\cdot \mathcal{P}$.

\end{proof}

\begin{ejemplo}
Sea $g=(a\ b\ c\ d)$ y $G=\langle g \rangle\leq Sym(\{a,b,c,d\})$ y sea $H=S_{3}$, de tal manera que identificamos $\{Id,(1\ 2\ 3),(1\ 3\ 2)\} \mapsto \{h_{1},h_{2},h_{3}\}$. Entonces $G\wr H$ actúa sobre $\{a,b,c,d\}\times\{1,2,3\}$.\\
Por ejemplo:
$$[(g,g^{2},e),h_{2}]\cdot(c,2)=(s_{h_{2}}(c),h_{2}(2))=(g^{2}(c),(1\ 2\ 3)(2))=(a,3).$$

\end{ejemplo}
\begin{proposicion}\label{conservacion de particion}
La acción de evaluación del monoide $End_{G}(\mathcal{B}_{[H_{i}]})$ sobre $\mathcal{B}_{[H_{i}]}$ preserva la partición de $G$-órbitas. 
\end{proposicion}

\begin{proof}
Es fácil ver que un elemento $z\in \tau(Gx)$ es de la forma $z=\tau(g\cdot x)=g\cdot \tau(x)$. Esto muestra que $\tau(Gx)=G\tau(x)$. Y por tanto la acción de evaluación preserva la partición por $G$-órbitas.
\end{proof}

\subsection{La acción y el producto corona}

Ahora probaremos una isomorfía importante sobre los endomorfismos de $G$-conjuntos, pero en esta ocasión restringidos a un conjunto de configuraciones más grande, los conjuntos $\mathcal{B}_{[H]}$.\\

Se tiene que $G$ actúa sobre el conjunto de clases laterales $G/G_{x}$ por multiplicación izquierda:
$$h\cdot gG_{x}=hgG_{x}.$$
Denotemos por $G^{x}=G/G_{x}$ al cociente y por $gG_{x}=\overline{g}$a sus elementos. Entonces esta acción puede extenderse a $G^{x}\times X$ como:
$$h\cdot(\overline{g},x_{\lambda})=(h\cdot\overline{g},x_{\lambda})=(\overline{hg},x_{\lambda}).$$
Trabajaremos el caso donde $G_{x}=H$. Además $G$ también actúa sobre $Gx$ por multiplicación izquierda:
$$h\cdot(g\cdot x)=hg\cdot x.$$

\begin{proposicion}
La acción de $G$ sobre $G^{x}$ es equivalente a su acción sobre $Gx$, para todo $x\in X$. 
\end{proposicion}

\begin{proof}
Tenemos una función $\phi:G^{x} \rightarrow Gx$ dada por $\phi(\overline{g})=g\cdot x$. Notemos que 

 $$\overline{g}=\overline{h}\iff h^{-1}g \in G_{x}\iff h^{-1}g\cdot x=x\iff g\cdot x=h\cdot x.$$
 Por lo tanto $\phi$ está bien definida y además es inyectiva.\\
 También se tiene que para todo $g\cdot x\in Gx$, $\phi(\overline{g})=g\cdot x$, por lo tanto es sobreyectiva. \\
 Por otro lado: $$\phi(h\cdot \overline{g})= \phi (\overline{hg})=hg\cdot x= h\cdot (g\cdot x)= h\cdot \phi(\overline{g}).$$
 Esto muestra la equivalencia entre sus acciones. 
\end{proof}

\begin{proposicion}
La acción de $G$ sobre $G^{x}$ es isomorfa a su acción sobre el conjunto $G^{x}\times \{x\}$, para todo $x\in X$. 
\end{proposicion}
\begin{proof}
Tenemos una función $\phi_{x}:G^{x} \rightarrow G^{x}\times X$ dada por $\phi_{x}(\overline{g})=(\overline{g},x)$.
Por argumentos análogos a la demostración anterior se tiene que $\phi_{x}$ está bien definida y $G$ es isomorfo a sí mismo trivialmente. Además se tiene que:
 $$\phi_{x}(h\cdot \overline{g})= \phi_{x} (\overline{hg})=(\overline{hg}, x)= h\cdot (\overline{g}, x)= h\cdot \phi_{x}(\overline{g}).$$
\end{proof}

Este resultado nos permite visualizar $G^{x}$ dentro de $G^{x}\times X$, i.e., $G^{x}\leq G^{x}\times X$. Específicamente hay $|X|$ copias de $G^{x}$ dentro de $G^{x}\times X$.\\

Para este apartado $G$ actúa sobre un conjunto $Y$ y $X$ es un subconjunto de $Y$ el cual contiene un elemento de cada $G$-órbita en $Y$. Dado un subgrupo $H$ de un grupo $G$, denotamos por $X$ a algún conjunto de representantes de las $G$-órbitas de $\mathcal{B}_{[H]}$, tal que todos los elementos en $X$ tienen el mismo estabilizador, i.e., para cuales quiera $x,y\in X$ se tiene que $Gx\neq Gy$ y $G_{x}=G_{y}$.\\ Notemos que, por la $G$-equivarianza, cada función $G$-equivariante $\tau$ induce dentro de la caja una transformación en $Trans(X)$.\\
Es decir, dado 
$$\tau(g\cdot x_{\lambda})= h\cdot x_{\kappa},$$
existe una transformación $f:X\rightarrow X$ tal que 
$$f(x_{\lambda})=x_{\kappa}.$$
De igual manera, para cada $x_{\lambda}$, se induce una transformación $\varphi: G^{x}\rightarrow G^{x}$ dada por
$$\varphi(\overline{g})=\overline{h}.$$
Denotaremos a estas funciones por $f_{\tau}$ y $\varphi_{\lambda,\tau}$ respectivamente.\\

\begin{teorema}\label{abc1}
Sea $G$ un grupo actuando sobre un conjunto $Y$, sea $H$ un subgrupo de $G$ y $X$ el conjunto de representantes de $G$-órbitas de $\mathcal{B}_{[H]}$. Entonces la acción de evaluación de $End_{G}(\mathcal{B}_{[H]})$ sobre $\mathcal{B}_{[H]}$ es isomorfa a la acción de algún subgrupo de $Trans(Gx)\wr Trans(X)$ sobre $Gx\times X$, para todo $x\in X$. 
\end{teorema}

\begin{proof}
Primero definimos una biyección entre $\mathcal{B}_{[H]}$ y $G^{x} \times X$. Tomamos la función $\theta:g\cdot x_{\lambda} \mapsto (\overline{g},x_{\lambda})$ la cual satisface que:
$$g\cdot x_{\lambda}=h\cdot x_{\lambda} \iff\overline{g}=\overline{h}\iff(\overline{g},x_{\lambda})=(\overline{h},x_{\lambda})\iff \theta(g\cdot x_{\lambda}) = \theta(h\cdot x_{\lambda}).$$
Esto muestra que $\theta$ está bien definida y además es inyectiva. 
Y, por como está construida, es fácil identificar que es sobreyectiva, pues para cualquier elemento de la forma $(\overline{g},x_{\lambda})$ se tiene que su preimagen es de la forma $g\cdot x_{\lambda}$.\\

Notemos que para dos elementos $\pi, \tau \in End_{G}(\mathcal{B}_{[H]})$ que satisfacen
$$\tau(g\cdot x_{\lambda})=h\cdot x_{\kappa} \hspace{0.2in}\mbox{ y }\hspace{0.2in}\pi(h\cdot x_{\kappa})=k\cdot x_{\rho}$$
se tiene que:
$$f_{\tau}(x_{\lambda})=x_{\kappa},\hspace{0.2in} f_{\pi}(x_{\kappa})=x_{\rho}$$ y 
$$\varphi_{\lambda,\tau}(\overline{g})=\overline{h},\hspace{0.2in} \varphi_{\kappa,\pi}(\overline{h})=\overline{k}.$$
Por otro lado tenemos que 
$$\pi\tau(g\cdot x_{\lambda})=k\cdot x_{\rho},$$
lo que implica que $$f_{\pi \tau}(x_{\lambda})=x_{\rho},\hspace{0.2in} \varphi_{\lambda,\pi\tau}(\overline{g})=\overline{k}. $$ Tomamos una función $\sigma_{\tau}:X\rightarrow \{\varphi_{\lambda,\tau}\}$, dada por $\sigma_{\tau}(x_{\lambda})= \varphi_{\lambda,\tau}$. Y además definimos una función  $\varphi:End_{G}(\mathcal{B}_{[H]})\rightarrow Trans(G^{x})\wr Trans(X)$ dada por $\varphi(\tau)= (\sigma_{\tau},f_{\tau})$. \\
Entonces tenemos que:
$$\begin{array}{rll}
\varphi(\pi)\varphi(\tau)\cdot(\overline{g},x_{\lambda}) &= 
(\sigma_{\pi},f_{\pi})(\sigma_{\tau},f_{\tau})\cdot(\overline{g},x_{\lambda})&=
(\sigma_{\pi},f_{\pi})\cdot\left(\sigma_{\tau}(x_{\lambda})(\overline{g}),f_{\tau}(x_{\lambda})    \right)\\
& =
(\sigma_{\pi},f_{\pi})\cdot(\varphi_{\lambda,\tau}(\overline{g}),f_{\tau}(x_{\lambda}))&=
(\sigma_{\pi},f_{\pi})\cdot(\overline{h},x_{\kappa})\\
&=(\sigma_{\pi}(x_{\kappa})(\overline{h}),f_{\pi}(x_{\kappa}))&=
(\varphi_{\kappa,\pi}(\overline{h}),f_{\pi}(x_{\kappa}))\\
&= (\overline{k},x_{\rho}) &= (\varphi_{\lambda, \pi \tau}(\overline{g}),f_{\pi\tau}(x_{\lambda}))\\
&= (\sigma_{\pi\tau}(x_{\lambda})(\overline{g}),f_{\pi\tau}(x_{\lambda})) &= (\sigma_{\pi \tau},f_{\pi\tau})\cdot(\overline{g},x_{\lambda})\\
 & & = \varphi(\pi \tau)\cdot(\overline{g},x_{\lambda})
\end{array},$$
para cuales quiera elementos $\overline{g}\in G^{x},  x_{\lambda}\in X$. Entonces $\varphi(\pi)\varphi(\tau)=\varphi(\pi\tau)$, $\forall \pi,\tau\in End_{G}(\mathcal{B}_{[H]})$, por lo que $\varphi$ es un homomorfismo.\\

 Notemos que la identidad en $Trans(G^{x})\wr Trans(X)$ es de la forma $(Id,Id)$, donde la primer identidad es la función que manda todo $x_{\lambda}\in X$ a la identidad de $G^{x}$ en sí mismo. Entonces sea $\tau\in End_{G}(\mathcal{B}_{[H]})$ tal que $\varphi(\tau)=1$, se tiene que $\sigma_{\tau}(x_{\lambda})(\overline{g})=\varphi_{\lambda,\tau}(\overline{g})=\overline{g}$, para todo $\overline{g}\in G^{x}$ , entonces $\varphi_{\lambda,\tau}$ es la identidad en $G^{x}$, para todo $x_{\lambda}\in X$, entonces $\sigma_{\tau}$ es la identidad mencionada. Además $f_{\tau}=Id$. Entonces $\tau$ es la identidad en $End_{G}(\mathcal{B}_{[H]})$ y por tanto $\varphi$ es inyectiva.\\

Debemos mostrar también que $(\theta, \varphi)$ es un embebimiento de acción. Sea $g\cdot x_{\lambda} \in \mathcal{B}_{[H_{i}]}$ y $\tau\in End_{G}(\mathcal{B}_{[H_{i}]})$, tenemos que 
$$ \varphi(\tau)\cdot(\theta(g\cdot x_{\lambda}))=(\sigma_{\tau},f_{\tau})\cdot(\overline{g},x_{\lambda})=(\sigma_{\tau}(x_{\lambda})(\overline{g}),f_{\tau}(x_{\lambda}))$$
$$= (\overline{h},x_{\kappa})= \theta(h\cdot x_{\kappa}) =\theta(\tau(g\cdot x_{\lambda})).$$

\end{proof}

Como concecuencia de este resultado, podemos decir que se puede visualizar a $End_{G}(\mathcal{B}_{[H]})$ como un submonoide de $Trans(Gx)\wr Trans(X)$, i.e.:
$$End_{G}(\mathcal{B}_{[H]})\leq Trans(Gx)\wr Trans(X).$$

\begin{teorema}\label{abc2}
Sea $G$ un grupo actuando sobre un conjunto $Y$, sea $H$ un subgrupo de $G$ y $X$ el conjunto de representantes de $G$-órbitas de $\mathcal{B}_{[H]}$. Dada una función $\tau:\mathcal{B}_{[H]}\rightarrow \mathcal{B}_{[H]}$, vista como $\tau=(\sigma_{\tau},f_{\tau})$. Entonces $\tau\in End_{G}(\mathcal{B}_{[H]})$ si y solo si $\sigma_{\tau}(x_{\lambda}) \in End_{G}(Gx)$, para todo $x_{\lambda}\in X$.
\end{teorema}

\begin{proof}
Supongamos que $\sigma_{\tau}(x_{\lambda})\in End_{G}(Gx)$ para todo $x_{\lambda}\in X$. Si visualizamos los elementos en $Gx$ por su equivalencia con $G^{x}$ entonces tenemos que para cualquier $x\in \mathcal{B}_{[H]}$, existe un $g\in G$ y un $x_{\lambda}\in X$ tal que $x=g\cdot x_{\lambda}$. Entonces:
$$\begin{array}{rl}
\tau(h\cdot x)&= \tau(hg\cdot x_{\lambda})\\
& \approx  (\sigma_{\tau},f_{\tau})\cdot (\overline{hg},x_{\lambda})\\
&= (\sigma_{\tau}(x_{\lambda})(\overline{hg}),f_{\tau}(x_{\lambda}))\\
&=(h\cdot \sigma_{\tau}(x_{\lambda})(\overline{g}),f_{\tau}(x_{\lambda}))\\
&=h\cdot( \sigma_{\tau}(x_{\lambda})(\overline{g}),f_{\tau}(x_{\lambda}))\\
&= h\cdot [(\sigma_{\tau},f_{\tau})\cdot (\overline{g},x_{\lambda})]\\
& \approx h\cdot \tau(g\cdot x_{\lambda})\\
&= h\cdot \tau(x).

\end{array}$$

Supongamos ahora que $\tau\in End_{G}(\mathcal{B}_{[H]})$ y supongamos además que existe un $x_{\kappa}\in X$ tal que $\sigma_{\tau}(x_{\kappa})$ no es $G$-equivariante, es decir que $h\cdot \sigma_{\tau}(x_{\kappa})(\overline{g})\neq\sigma_{\tau}(x_{\kappa})(\overline{hg}) $. Entonces tenemos que:

$$\begin{array}{rcl}
h\cdot \tau(x)&=& \tau(h\cdot x)\\
h\cdot \tau (g\cdot x_{\lambda})&=& \tau(hg\cdot x_{\lambda})\\
h\cdot [(\sigma_{\tau},f_{\tau})\cdot(\overline{g},x_{\lambda})]&=& (\sigma_{\tau},f_{\tau})\cdot (\overline{hg},x_{\lambda})\\
h\cdot(\sigma_{\tau}(x_{\lambda})(\overline{g}),f_{\tau}(x_{\lambda}))&=& (\sigma_{\tau}(x_{\lambda})(\overline{hg}),f_{\tau}(x_{\lambda}))\\
(h\cdot\sigma_{\tau}(x_{\lambda})(\overline{g}),f_{\tau}(x_{\lambda}))&=& (\sigma_{\tau}(x_{\lambda})(\overline{hg}),f_{\tau}(x_{\lambda})).
\end{array}$$

Como $x$ es cualquier elemento en la caja, se tiene que estas ecuaciones deben satisfacerse para cualquier $g,h\in G$ y para cualquier $x_{\lambda} \in X$. De la última ecuación se infiere directamente que 
$$h\cdot \sigma_{\tau}(x_{\lambda})(\overline{g})=\sigma_{\tau}(x_{\lambda})(\overline{hg}),\ \forall x_{\lambda}\in X.$$

 \end{proof}

Como consecuencia de los Teoremas \ref{abc1} y \ref{abc2} tenemos los siguientes resultados. 

\begin{corolario}
Sea $G$ un grupo actuando sobre un conjunto $Y$, sea $H$ un subgrupo de $G$, $X$ el conjunto de representantes de $G$-órbitas de $\mathcal{B}_{[H]}$ y sea $x\in \mathcal{B}_{[H]}$. Se satisface que:
$$End_{G}(\mathcal{B}_{[H]})\cong End_{G}(Gx)\wr Trans(X).$$

\end{corolario}
\begin{corolario}
Sea $G$ un grupo actuando sobre un conjunto $Y$, sea $H$ un subgrupo de $G$ y $X$ el conjunto de representantes de $G$-órbitas de $\mathcal{B}_{[H]}$. Dada $\tau\in Aut_{G}(\mathcal{B}_{[H]})$, vista como $\tau=(\sigma_{\tau},f_{\tau})$, entonces se satisfacen las siguiente afirmaciones:
\begin{enumerate}
\item[i)] $$\sigma_{\tau}(x_{\lambda})\in Aut_{G}(Gx),\ \forall x_{\lambda}\in X. $$
\item[ii)] $$f_{\tau}\in Sym (X).$$
\end{enumerate}

\end{corolario}

\begin{proof}
Esto se verifica directamente del hecho que un para que un elemento $\tau=\sigma_{\tau},f_{\tau})$ sea biyectivo ambas partes deben ser funciones biyectivas.   
\end{proof}

\subsection{Algunos resultados adicionales}

Dado un grupo finito $G$, entonces podemos indexar sus clases de conjugación de subgrupos, incluso ordenarlas con base en su cardinal:
$$|H_{1}| \leq |H_{2}| \leq ... \leq |H_{r}|.$$

Sea $X_{i}$ un conjunto de representantes de las $G$-órbitas de la caja $\mathcal{B}_{[H_{i}]}$, tales que tienen el mismo estabilizador, es decir, sea $x\in \mathcal{B}_{[H_{i}]}$ :
$$X_{i}=\{y\in\mathcal{B}_{[H_{i}]}:\ Gx\neq Gy,G_{x}=G_{y}\}.$$

\begin{proposicion}
Sea $G$ un grupo finito que actúa sobre un conjunto $X$ y sea $H_{i}$ un subgrupo en $Sub_{G}(X)$, entonces $\mathcal{B}_{[H_{i}]}$ es isomorfo a $G/H_{i}  \times X_{i}$ y se tiene que la partición por $G$-órbitas dentro de la caja $\mathcal{B}_{[H_{i}]}$ es equivalente a la partición natural de $G/H_{i}\times X_{i}$. 
\end{proposicion}
\begin{proof}
Tenemos que 
$$\mathcal{B}_{[H_{i}]}=\{g\cdot x_{\lambda}:\ g\in G, x_{\lambda}\in X_{i}\}=\{(gH_{i},x_{\lambda}),\ g\in G,\ x_{\lambda}\in X_{i}\}= G/H_{i}\times X_{i},$$
y se tiene que la partición por $G$-órbitas dentro de la caja $\mathcal{B}_{[H]}$ es equivalente a la partición natural de $G/H_{i} \times X_{i}$. 
\end{proof}

El siguiente es una versión particular de un teorema que describe un monoide de transformaciones en términos de un producto corona (ver \cite{Csaba}). Más adelante se presentará una demostración de dicho resultado generalizado dentro de la estructura algebraica específica de nuestro interés. 

\begin{teorema}
Sea $G$ un grupo de permutaciones sobre un conjunto $\Omega$, sea $\mathcal{P}$ una partición $G$-invariante de $\Omega$ con una cantidad finita de bloques y sea $P\in \mathcal{P}$. Supongamos que $G$ induce un grupo de permutaciones transitivo sobre $\mathcal{P}$. Entonces $G$ es permutacionalmente isomorfo a un subgrupo de $Sym(P)\wr Sym(\mathcal{P})$ que actúa sobre $P\times \mathcal{P}$.
\end{teorema}

Sea $Y\subseteq X$, definimos los estabilizadores \emph{de conjunto} y \emph{puntual} de $Y$, los cuales son subgrupos de $G$, como sigue respectivamente:
$$G_{Y}:=\{g\in G: g\cdot Y = Y\} \hspace{0.3in}\mbox{ y }\hspace{0.3in} G_{(Y)}:=\{g\in G: \ g\cdot y = y,\ \forall y\in Y\}.$$

Una partición $\mathcal{P}$ se dice $G$-invariante si $g\cdot P \in \mathcal{P}$, para todo $P\in \mathcal{P}$. Un subconjunto $Y \subset X$ se dice $G$-invariante si $g\cdot y \in Y$, para todo $y\in Y$. Es fácil probar que $Y\subseteq X$ es $G$-invariante si y solo si $Y$ es la unión de $G$-órbitas. Deberemos restringir la acción de $G$ a $Y$ y considerar el monoide $End_{G}(Y)$ y el grupo $Aut_{G}(Y)$.

\begin{lema}\label{lema0}
Para cada subconjunto $Y \subseteq X$ $G$-invariante, se satisface lo siguiente:
\begin{enumerate}
\item[i)] $End_{G}(Y)$ es isomorfo a un submonoide de $End_{G}(Y)$ y $Aut_{G}(Y)$ es isomorfo a un subgrupo de $Aut_{G}(X)$. 
\item[ii)] Si además $Y$ es $Aut_{G}(X)$-invariante, entonces $Aut_{G}(Y)$ es isomorfo a un subgrupo normal de $Aut_{G}(X)$. 
\item[iii)] $Aut_{G}(Y)\cong Aut_{G}(X)_{Y}/Aut_{G}(X)_{(Y)}$.

\end{enumerate}

\end{lema}

\begin{proof}
El monoide $End_{G}(Y)$ se embebe en $End_{G}(X)$ vía el homomorfismo inyectivo $\Phi:End_{G}(Y) \rightarrow End_{G}(X)$ dado por
$$\Phi(\tau):= \left\{ \begin{array}{cc} \tau(x) & \mbox{si } y\in Y \\ x & \mbox{otro caso.}  \end{array} \right.$$
Es fácil verificar que de hecho $\Phi(\tau)(x)\in End_{G}(X)$ usando el hecho que $g\cdot y\in Y$ si y solo si $x\in Y$. Al restringir $\Phi$ a $Aut_{G}(Y)$, se muestra que $Aut_{G}(Y)$ esta embebido en $Aut_{G}(X)$.\\

Ahora, si fijamos $\sigma \in Aut_{G}(X)$ y $\tau \in Aut_{G}(Y)$ y asumiendo que $Y$ es $Aut_{G}(X)$-invariante, como $\sigma(y)\in Y$, para todo $y\in Y$, se tiene que $\sigma|_{Y}\in Aut_{G}(Y)$. Entonces $\sigma \Phi(\tau) \sigma^{-1}=\Phi(\sigma|_{Y} \tau \sigma^{-1}|_{Y})\in \Phi(Aut_{G}(Y))$, lo cual muestra que $Aut_{G}(Y)$ es normal en $Aut_{G}(X)$.\\

Finalmente, consideremos el homomorfismo de restricción $\Psi: Aut_{G}(X) \rightarrow Aut_{G}(Y)$ dado por $\Psi(\tau)=\tau|_{Y}$, para todo $\tau\in Aut_{G}(X)$. No cuesta trabajo verificar que $\Psi(Aut_{G}(X)_{Y})=Aut_{G}(Y)$ y $ker(\Psi)=Aut_{G}(X)_{(Y)}$, entonces la tercera parte del resultado se cumple por el primer Teorema de isomorfía. 
\end{proof}

\newpage

\section{El monoide de transformaciones $G$-equivariantes}

En esta sección presentamos los resultados más sobresalientes de todo el trabajo realizado durante el doctorado. En primera instancia debemos introducir algunos conceptos y notación los cuales son más especializados y son de gran importancia para nuestro trabajo. 

\begin{lema}\label{lema1}
Sean $x,y\in X$ tales que $Gx\neq Gy$. Entonces existe un endomorfismo no invertible $\tau\in End_{G}(X)$ tal que $\tau(x)=y$ si y solo si $G_{x} \leq G_{y}$.  
\end{lema}

\begin{proof}
Supongamos que existe $\tau \in End_{G}(X)$ tal que $\tau(x)=y$. Dado $g\in G_{x}$, tenemos que $g\cdot x=x$, entonces 
$$y=\tau(x)=\tau(g\cdot x)= g\cdot \tau(x)= g\cdot y.$$
En consecuencia $g\in G_{y}$ y por tanto $G_{x} \leq G_{y}$.\\
Ahora supongamos que $G_{x}\leq G_{y}$. Definimos $\tau$ como:
$$\tau(z)= \left\{ \begin{array}{cl} g\cdot y & if\ z=g\cdot x  \\ z & \mbox{ en otro caso. }  \end{array}  \right.$$
Debemos probar que $\tau:X\rightarrow X$ está bien definida. Supongamos que $z=g\cdot x=h\cdot x $. Esto significa que $h^{-1}g\in G_{x}$. Como $G_{x} \leq G_{y}$, tenemos que:
$$\tau(h^{-1}g\cdot x)=h^{-1}g \cdot y=y \Rightarrow g\cdot y= h \cdot y \Rightarrow \tau(g\cdot x)= \tau(h\cdot x).$$
El caso donde $z\neq g\cdot x$ no requiere un proceso adicional, esto es por que $z$ no puede tener otra representación.\\
Debemos verificar que $\tau$ es también $G$-equivariante. Sea $z=g\cdot x$, si $h\in G$ actúa sobre $\tau(z)$ obtenemos que:
$$h\cdot \tau(z)=h\cdot \tau(g\cdot x)= h\cdot(g\cdot y)=hg\cdot y=\tau(hg\cdot x),$$
donde $h\cdot z=hg\cdot x$.\\
Si $z$ no pertenece a $Gx$:
$$\tau(g\cdot z)=g\cdot z=g\cdot \tau(z).$$

\end{proof}

El siguiente lema es una generalización de \cite[Lemma 3]{finite groups}.

\begin{lema}\label{lema2}
Sean $x,y\in X$ entonces existe un automorfismo $\tau\in Aut_{G}(X)$ tal que $\tau(x)=y$ si y solo si $G_{x}=G_{y}$.
\end{lema}

\begin{proof}
Supongamos que existe $\tau \in Aut_{G}(X)$ tal que $\tau(x)=y$. Dado $g\in G_{x}$, tenemos que $g\cdot x=x$, entonces: 
$$y=\tau(x)=\tau(g\cdot x)= g\cdot \tau(x)= g\cdot y,$$ en consecuencia $g\in G_{y}$. \\
Por el otro lado, dado $h\in G_{y}$, como $\tau$ es invertible, existe $\tau^{-1}\in Aut_{G}(X)$ tal que $\tau^{-1}(y)=x$. Y utilizando el mismo proceso $h\in G_{x}$.\\

Ahora supongamos que $G_{x}=G_{y}$. 
Identificamos dos casos: cuando $Gx=Gy$ y cuando $Gx\neq Gy$. \\

Sean $x,y$ tales que $Gx\neq Gy$, definimos $\tau$ como:
$$\tau(z)= \left\{ \begin{array}{cl} g\cdot y & if\ z=g\cdot x \\ g\cdot x & if\ z=g\cdot y \\ z & \mbox{ en otro caso. }  \end{array}  \right.$$
En primer lugar debemos mostrar que $\tau:X\rightarrow X$ está bien definida. Supongamos que $z=g\cdot x=h\cdot x $, esto significa que $h^{-1}g\in G_{x}$. Entonces tenemos las siguientes implicaciones:
$$\tau(h^{-1}g\cdot x)=h^{-1}g\cdot y =y  \Rightarrow g\cdot y= h \cdot y \Rightarrow \tau(g\cdot x)= \tau(h\cdot x).$$
El proceso es análogo para $z=g\cdot y= h\cdot y$ y el tercer caso no requiere de ningún proceso adicional para demostrarse.\\
Debemos verificar que $\tau$ es además $G$-equivariante. Sea $z=g\cdot x$, si $h\in G$ actúa sobre $\tau(z)$, entonces tenemos que:
$$h\cdot \tau(z)=h\cdot \tau(g\cdot x)= h\cdot(g\cdot y)=hg\cdot y=\tau(hg\cdot x),$$
donde $h\cdot z=hg\cdot x$. 
Una vez mas, el proceso para $z=g\cdot y$ es análogo.\\
Si $z\notin Gx \cup Gy$, entonces $g\cdot z \notin Gx\cup Gy$. Entonces $$\tau(g\cdot z)=g\cdot z=g\cdot \tau(z).$$
Debemos probar que $\tau$ es además biyectiva. Dado $\tau(z)=g\cdot y$. Elegimos $z=g\cdot x $, entonces $\tau(g\cdot x)=\tau(z)$. Un fenómeno análogo sucede para $\tau(z)=g\cdot x$, eligiendo $z=g\cdot y$ y para el caso donde $z\notin Gx \cup Gy$ se satisface que $\tau(z)=z$.  Entonces $\tau$ es sobreyectiva y en consecuencia, como $A^{G}$ es finito, es también inyectiva.

Sean $x,y$ tales que $Gx= Gy$, entonces existe $k\in G$ tal que $y=k\cdot x$ y definimos $\tau_{k}$ como:
$$\tau_{x,k}(z)= \left\{ \begin{array}{cl} gk\cdot x & if\ z=g\cdot x, \\ z & \mbox{ en otro caso. }  \end{array}  \right.$$

Notemos que las funciones $G$-equivariantes biyectivas $\tau_{x,k}$ están definidas si y solo si se satisface que $G_{x}=G_{k\cdot x}=kG_{x}k^{-1}$, lo cual se cumple si y solo si $k\in N_{G}(G-{x})$. Más aún, son $G$-equivariantes, por un argumento similar al caso anterior y son invertibles pues $\tau_{x,k^{-1}}$ es su función inversa.

\end{proof}

\begin{corolario}\label{qw}
Sea $x\in \mathcal{B}_{H}$, denotemos por $Aut_{G}(x)$ a la $Aut_{G}(X)$-órbita de $x$, entonces $$\mathcal{B}_{H}=Aut_{G}(x).$$
\end{corolario}

\begin{ejemplo}
Para cualquier grupo $G$ y un conjunto $A$ tal que $|A|\geq 2$, consideremos la acción shift de $G$ sobre $A^{G}$. Sin pérdida de generalidad podemos asumir que $\{0,1\}\subset A$. Para cada subgrupo $H \leq G$ definimos la función de indicación $\chi_{H}:G \rightarrow A$ como
$$\chi_{H}(g):= \left\{ \begin{array}{cc} 1 & \mbox{si }g\in H \\ 0 & \mbox{otro caso.}  \end{array} \right.$$
Es fácil ver que el estabilizador de $\chi_{H}$ satisface que $G_{\chi_{H}}=H$ y por consiguiente se tiene que $Stab_{G}(A^{G})=\{H:\ H \leq G\}$. Es decir que todo subgrupo de $G$ es el estabilizador de alguna configuración $x$ en $A^{G}$, en este caso las configuraciones en cuestión son las funciones de indicación.  
\end{ejemplo}

Como $G_{g\cdot x}=gG_{x}g^{-1}$, para todo $g\in G$ y para todo $x\in X$, el conjunto $\mathcal{B}_{[H]}$ es $G$-invariante. Más aún, por el Lema \ref{lema1}, es también $Aut_{G}(X)$-invariante y por tanto se tiene que $Aut_{G}(\mathcal{B}_{[H]})$ es isomorfo a un subgrupo de $Aut_{G}(X)$, $$Aut_{G}(\mathcal{B}_{[H]}) \leq Aut_{G}(X).$$

De hecho es posible calcular el número de $Aut_{G}(X)$-órbitas dentro de $\mathcal{B}_{[H]}$.

\begin{teorema}\label{qwe}
Sea $G$ un grupo que actúa sobre un conjunto $X$ y $H$ un subgrupo dado, entonces el número de $Aut_{G}(X)$-órbitas dentro de $\mathcal{B}_{[H]}$ es $[G:N_{G}(H)]$.
\end{teorema}

\begin{proof}

Denotemos por $N$ a $N_{G}(H)$. Notemos que por su definición tenemos que $\mathcal{B}_{K}=\mathcal{B
}_{J}$ si y solo si $K=J$, para dos subgrupos $K,J\leq G$. Y recordemos que $Aut_{G}(x)=\mathcal{B}_{G_{x}}$.\\ Además todos los subgrupos que definen las $Aut_{G}$-órbitas en $\mathcal{B}_{[H]}$ deben ser conjugados de $H$. 
Entonces basta probar que la aplicación  $\mathcal{B}_{gHg^{-1}} \longmapsto gN$ es una biyección, lo cual se sigue de las siguientes implicaciones:
$$\mathcal{B}_{gHg^{-1}}= \mathcal{B}_{kHk^{-1}} \iff gHg^{-1}=kHk^{-1} \iff$$
$$k^{-1}gHg^{-1}k=(k^{-1}g)H(k^{-1}g)^{-1}=H \iff k^{-1}g\in N \iff gN=kN.$$
\end{proof}

Por el Teorema de Órbita-estabilizador, la cardinalidad de cada $G$-órbita dentro de $\mathcal{B}_{[H]}$ es el índice de $H$ en $G$, $[G:H]$.\\

\subsection{La estructura del monoide; caso finito}

Encontrar isomorfismos entre espacios nuevos con respecto a espacios conocidos ha sido siempre una tarea clásica de cualquier algebrista. En esta sección buscamos equivalencias entre los objetos de nuestro interés y algunos otros a partir de construcciones como el producto corona.

\subsubsection{Estructura dentro de las $G$-órbitas}

Abordamos nuestro trabajo desglozando con base en sus particiones al conjunto $X$. Comenzamos trabajando con el caso más pequeño: las $G$-órbitas. \\

Denotemos por $\mathcal{O}_{[H]}$ el conjunto de $G$-órbitas contenidas en $\mathcal{B}_{[H]}$:
$$\mathcal{O}_{[H]}:= \{Gx\subseteq X:\ [G_{x}]=[H]\},$$
y sea $\alpha_{[H]}:=|\mathcal{O}_{[H]}|$.\\

Notemos que para cada clase $[H]\in Conj_{G}(X)$ y dada cualquier función $f:\mathcal{O}_{[H]} \rightarrow \mathcal{O}_{[H]}$, por el axioma de elección , para cada $O\in \mathcal{O}_{[H]}$ es posible elegir un elemento $x_{O}\in O$ tal que $G_{x_{O}}=H$. Entonces es posible definir una función $\tau:\mathcal{B}_{[H]} \rightarrow \mathcal{B}_{[H]}$ como $\tau(g\cdot x_{O})=g\cdot x_{f(O)}$. Es fácil verificar que $\tau\in End_{G}(\mathcal{B}_{[H]})$ y que $\tau(O)=f(O)$ para todo $O\in \mathcal{O}_{[H]}$. Más aún, si $f\in Sym(\mathcal{O}_{[H]})$, entonces $\tau \in Aut_{G}(\mathcal{B}_{[H]})$.\\

Cuando $H$ es un subgrupo de índice finito de un grupo finitamente generado $G$ podemos usar el látice de subgrupos de índice finito $L(G)$ de $G$, el cual es localmente finito y su correspondiente función de Möbius $\mu:L(G)\times L(G) \rightarrow \mathbb{Z}$ para calcular el valor de $\alpha_{[H]}$, como se marca en el siguiente lema (ver \cite{undim} para más detalles).

\begin{lema}
Supongamos que $H\in Stab_{G}(X)$ es un subgrupo de índice finito  de un grupo finitamente generado $G$ y que $X$ es un conjunto finito. Entonces,
$$\alpha_{[H]}=\frac{[G:N_{G}(H)]}{[G:H]}\sum_{H\leq K \leq G}{\mu(H,K)}|Fix(K)|,$$
donde $Fix(K):=\{x\in X:\ k\cdot x=x, \ \forall k\in K\}$.
\end{lema}

Notemos que la suma del lema anterior es finita, ya que el número de subgrupos infinitos entre $G$ y $H$ es finito, ya que $H$ es de índice finito y $G$ es finitamente generado (y además tiene una cantidad finita de subgrupos de un índice finito dado). Más aún, cualquier grupo que contiene a $H$ debe de tener índice finito, y, en particular, $N_{G}(H)$ es de índice finito. Cuando $X=A^{G}$ y la acción de $G$ sobre $A^{G}$ es la acción shift, entonces $|Fix(K)|=|A^{G/K}|$, para todo $H\leq G$. Entonces, el resultado anterior se satisface para la acción shift siempre y cuando $A$ sea un conjunto finito. \\

Si la acción de $G$ sobre $X$ es no-transitiva, el Lema \ref{lema2} muestra que $Aut_{G}(X)$ está propiamente contenido en $End_{G}(X)$. Cuando la acción es transitiva, el siguiente resultado caracteriza cuando $Aut_{G}(X)$ resulta ser igual que $End_{G}(X)$.

\begin{teorema}
Supongamos que $G$ actúa transitivamente sobre $X$ y sea $x\in X$. El subgrupo $G_{x}$ no está propiamente contenido en ninguno de sus conjugados si y solo si $$End_{G}(X)=Aut_{G}(X).$$ 
\end{teorema}

\begin{proof}
Supongamos que $G_{x}$ no está propiamente contenido en ninguno de sus conjugados. Tomamos una función $\tau\in End_{G}(X)$. Como la acción es transitiva, tenemos que $X=Gx$ para todo $x\in X$. Por $G$-equivarianza tenemos que:
$$\tau(X)=\tau(Gx)=G\tau(x)=X,$$
lo que muestra que $\tau$ es sobreyectiva. Supongamos que $\tau(x)=\tau(y)$, para algunos $x,y\in X$. Por transitividad, existe un $g\in G$ tal que $y=g\cdot x$. Entonces $\tau(x)=\tau(y)=\tau(g\cdot x)=g\cdot \tau(x)$, implica que $g\in G_{\tau(x)}$. Por el Lema \ref{lema2} se tiene que $G_{x}\leq G_{\tau(x)}$ y de nuevo, por transitividad existe un $h\in G$ tal que $G_{\tau(x)}=hG_{x}h^{-1}$. Entonces, por la hipótesis que $G_{x}$ no está contenido propiamente en ninguno de sus conjugados se tiene que $G_{x}=G_{\tau(x)}$, entonces $g\in G_{x}$. Así $y=g\cdot x= x$, lo que muestra que $\tau$ es inyectiva. y por tanto $\tau\in Aut_{G}(X)$.\\

Probamos el converso por contrapositiva. Supongamos que existen $x\in X$ y $h\in G$ tales que $G_{x} < hG_{x}h^{-1}$.  Definamos una función $\tau$ como sigue:
$$\tau(g\cdot x)= gh\cdot x,\ \forall g\in G.$$
Es fácil ver que $\tau\in End_{G}(X)$, pero mostraremos que no es inyectiva, lo cual será una contradicción. Consideremos $k\in hG_{x}h^{-1} \setminus G_{x}$. Entonces $k\cdot x \neq x$. De cualquier forma, $k\in hG_{x}h^{-1}$ implica que 
$$k\cdot(h\cdot x)=h\cdot x \iff \tau(k\cdot x)=\tau(x).$$
Entonces $\tau$ no es una función biyectiva y por consiguiente $Aut_{G}(X)$ debería estar contenido propiamente en $End_{G}(X)$. 
\end{proof}

Una observación importante es notar que si $H\leq G$ es un grupo finito, o de índice finito, o normal, entonces $H$ no está contenido propiamente en ninguno de sus conjugados y por consiguiente $End_{G}(Gx)=Aut_{G}(Gx)$ en estos casos.

\begin{ejemplo}
Sea $G$ un grupo infinito tal que existe un subgrupo $K \leq G$ y $h\in G$ tal que $K < hKh^{-1}$ (ver \cite{mikko} para algunos ejemplos específicos). En esta situación, consideremos la acción shift de $G$ sobre $A^{G}$, donde $A$ es un conjunto con al menos dos elementos. Sea $x:=\chi_{K}\in A^{G}$ la función indicador de $K$. Entonces $G_{x}=K$ y entonces, por el teorema anterior, $Aut_{G}(Gx)$ esta propiamente contenida en $End_{G}(Gx)$.
\end{ejemplo}

\begin{lema}\label{equi1}
Si la acción de $G$ sobre $X$ es transitiva, entonces:
$$Aut_{G}(X) \cong N_{G}(G_{x})/G_{x},\ \forall x\in X.$$
\end{lema}

\begin{proof}
Sea $x\in X$. Definimos un automorfismo de $X$.
$$\tau_{x,k}(z)= \left\{ \begin{array}{cc}
gk\cdot x & z=g\cdot x\\
z & \mbox{otro caso.}
\end{array}  \right.$$
Es fácil verificar que $\tau_{x,k}$ es un automorfismo si y solo si $k\in N_{G}(G_{x})$. 
$$k\in N_{G}(G_{x}) \iff G_{x}=k^{-1}G_{x}k .$$

Entonces se tiene que 
$$g\cdot x = h\cdot x  \iff h^{-1}g\cdot x= x \iff   k^{-1}h^{-1}gk \cdot x = x $$

$$\iff gk\cdot x = hk\cdot x \iff \tau_{x,k}(g\cdot x)= \tau_{x,k}(h\cdot x).$$
Entonces el mapeo $k \mapsto \tau_{x,k}$ de $N_{G}(G_{x})$ en $Aut_{G}(X)$ es una función sobreyectiva con kernel $G_{x}$. Lo que satisface el teorema por el Primer Teorema de Isomorfía. 
\end{proof}

Es interesante mencionar que el resultado anterior es valido aún si $G$ no es un grupo finito.

\subsubsection{Estructura dentro de las cajas}

Nuestro siguiente objetivo es describir la estructura de $Aut_{G}(X)$ para acciones no-transitivas, pero es importante ver cómo es la estructura si nos restringimos a cada caja y más importante aún es intentar describir la estructura completa del monoide a partir de la estructura de las cajas.

\begin{lema}\label{equi2}
Sea $G$ un grupo actuando sobre un conjunto $X$ y sea $H\in Stab_{G}(X)$. Entonces
$$Aut_{G}(\mathcal{B}_{[H]})\cong (N_{G}(H)/H)\wr Sym(\mathcal{O}_{[H]}).$$
\end{lema}

\begin{proof}
El conjunto $\mathcal{O}_{[H]}$ de $G$-órbitas dentro de $\mathcal{B}_{[H]}$ forman una partición $G$-invariante de $\mathcal{B}_{[H]}$, con $|\mathcal{O}_{[H]}|=\alpha_{[H]}$. Notemos que la acción inducida de $Aut_{G}(\mathcal{B}_{[H]})$ sobre $\mathcal{O}_{[H]}$ es isomorfa al grupo simétrico completo $Sym(\mathcal{O}_{[H]})$. Por el Teorema del embebimiento corona imprimitivo  (ver \cite[Teorema 5.5]{Csaba}) y por el Lema \ref{lema0}, $Aut_{G}(\mathcal{B}_{[H]})$ es permutacionalmente isomorfo a un subgrupo R de $Aut_{G}(Gx)\wr Sym(\mathcal{O}_{[H]})$, con $Gx\in \mathcal{O}_{[H]}$. Observemos que el kernel de la proyección de $R$ sobre $Sym(\mathcal{O}_{[H]})$ es isomorfa al producto directo $Aut(Gx)^{\mathcal{O}_{[H]}}$. Entonces, $$R=Aut_{G}(Gx)\wr Sym(\mathcal{O}_{[H]}).$$ El resultado se sigue del Lema \ref{equi1}. 
\end{proof}

La acción imprimitiva del producto corona tal como se establece en la sección 2.4.2 requiere que la partición $\mathcal{O}_{[H]}$ tenga una cantidad finita de bloques; de cualquier forma, esta hipótesis puede ser descartada y la demostración del Lema \ref{equi2} también funciona usando el axioma de elección. \\

El resultado anterior permite un avance muy grande en la estructura de todo el grupo de automorfimos, donde el resultado principal se presenta a continuación.

\begin{teorema}
Sea $G$ un grupo actuando sobre un conjunto $X$. Entonces:
$$Aut_{G}(X) \cong \prod_{[H]\in Conj_{G}(X)}{(N_{G}(H)/H)\wr Sym(\mathcal{O}_{[H]})}.$$
\end{teorema}

\begin{proof}
Para cada $\tau\in Aut_{G}(X)$, definimos una función
$$F:Aut_{G}(X) \rightarrow \prod_{[H]\in Conj_{G}(X)}{Aut_{G}(\mathcal{B}_{[H]})},$$ dada por $F(\tau)_{[H]}:=\tau|_{\mathcal{B}_{[H]}}$ para cada $\tau\in Aut_{G}(X)$. La función $F$ es inyectiva ya que $\{\mathcal{B}_{[H]}: [H]\in Conj_{G}(X)\}$ es una partición de $X$. Para mostrar que $F$ es sobreyectiva, observemos que para cada $\tau_{[H]}$, $[H]\in Conj_{G}(X)$, en el producto podemos definir $\tau \in Aut_{G}(X)$ como $\tau(x)= \tau_{[H]}(x)$ si y solo si $x\in \mathcal{B}_{[H]}$. Se sigue entonces que $\tau$ es de hecho $G$-equivariante ya que, para cada $[H]\in Conj_{G}(X)$, $\tau_{[H]}$ es $G$-equivariante y es $\mathcal{B}_{[H]}$ es $G$-invariante. Finalmente, $F$ es un homomorfismo ya que $\mathcal{B}_{[H]}$ es $Aut_{G}(X)$-invariante, entonces $(\tau \circ \sigma )|\mathcal{B}_{[H]}=\tau|_{\mathcal{B}_{[H]}}\circ \sigma|_{\mathcal{B}_{[H]}}$, para todo $\sigma, \tau \in Aut_{G}(X)$. El resultado se sigue por el Lema \ref{equi2}.  
\end{proof}

\begin{teorema}
Sea $G$ un grupo actuando sobre un conjunto $X$ y sea $H\in Stab_{G}(X)$. Supongamos que $H$ no está propiamente contenido en ninguno de sus conjugados. Entonces:
$$End_{G}(\mathcal{B}_{[H]})\cong (N_{G}(H)/H) \wr Trans( \mathcal{O}_{[H]}).$$
\end{teorema}

\begin{proof}
Como $End_{G}(\mathcal{B}_{[H]})$ preserva la partición uniforme de $\mathcal{B}_{[H]}$ en $G$-órbitas, se sigue que es isomorfo a un submonoide $M$ de $End_{G}(Gx)\wr Trans(\mathcal{O}_{[H]})$ para cada $x\in \mathcal{B}_{[H]}$ (Ver \cite[Lema 2.1]{Csaba}). Notemos también que la acción inducida de $End_{G}(\mathcal{B}_{[H]})$ sobre $\mathcal{O}_{[H]}$ genera a todo $Trans(\mathcal{O}_{[H]})$. Por los lemas anteriores se tiene que $End_{G}(Gx)\cong Aut_{G}(Gx)\cong N_{G}(H)/H$. Como en demostraciones anteriores, el kernel de la proyección de $M$ sobre $Sym(\mathcal{O}_{[H]})$ es isomorfo al producto directo $Aut_{G}(Gx)^{\mathcal{O}_{[H]}}$ y el resultado se sigue.  
\end{proof}

\subsection{El monoide finito de transformaciones $G$-equivariantes}

Existen dos parámetros muy importantes para nuestro trabajo: el cardinal del grupo que actúa $G$ y el cardinal del conjunto sobre el que se actúa $X$. En esta sección abordamos el caso donde ambos, tanto $G$ como $X$, son finitos. Esto tiene por consecuencia directa que el monoide de endomorfismos de $X$, $End_{G}(X)$, sea finito. \\

Dado un grupo finito $G$, entonces podemos enlistar las clases de conjugación de subgrupos:
$$[H_{1}]\leq [H_{2}] \leq ... \leq [H_{r}].$$
Entonces para cualquier endomorfismo $\tau\in End_{G}(X)$ podemos definir otros $r$ endomorfismos dados por:

$$\tau_{k}(z)=  \left\{\begin{array}{cc} \tau(z) & \mbox{ si }\ z\in \mathcal{B}_{[H_{k}]} \\ z& \mbox{ otro caso. }   \end{array}  \right.$$

\begin{proposicion}\label{descomposicion}
Sea $G$ un grupo finito y $X$ un conjunto finito. Dadas sus clases de conjugación $[ H_{1}],[H_{2}],\cdots,[H_{r}]$, ordenadas por $|H_{1}|\leq |H_{2}|\leq\cdots\leq|H_{r}|$. Sea $\tau\in End_{G}(X)$ un endomorfismo no invertible, entonces:
$$\tau=\tau_{1} \tau_{2} \cdots  \tau_{r}.$$

\end{proposicion}

\begin{proof}
Supongamos que $z\in\mathcal{B}_{[H_{i}]}$, entonces  $\tau_{k}(z)=z$ para todo $k>i$. Por definición   $\tau_{i}(z)=\tau(z)\in\mathcal{B}_{[H_{j}]}$, donde $i \leq j$, por el Lema \ref{lema1}.  Entonces $\tau_{k}(\tau_{i}(z))=\tau_{i}(z)=\tau(z)$ para todo $k<i$ y por consecuencia $\tau_{1}  \tau_{2}  \cdots  \tau_{r}(z)=\tau(z)$. 
\end{proof}

Esta construcción nos genera una herramienta importante para cumplir nuestro objetivo principal: determinar el rank relativo de $End_{G}(X)$.

\subsubsection{Conjuntos generadores}

Para este apartado seguimos considerando que $G$ es un grupo finito actuando sobre un conjunto finito $X$. También consideraremos que $[H_{1}],[H_{2}],...,[H_{r}]$ es la lista de clases de conjugación de subgrupos de $G$ en $Conj_{G}(X)$. Denotaremos $\mathcal{B}_{i}:=\mathcal{B}_{[H_{i}]}$, $\alpha_{i}:=\alpha_{[H_{i}]}$ y para cada $n\in \mathbb{N}$ denotamos $[n]:=\{1,2,...,n\}$.\\

Introducimos un poco de terminología y algunas definiciones de funciones necesarias para nuestra construcción. 
$$[x\mapsto y](z):= \left\{\begin{array}{cc} g\cdot y & \mbox{si }z=g\cdot x \\ z & \mbox{otro caso}\end{array} \right.$$
y
$$[x\leftrightarrow  y](z):= \left\{\begin{array}{cc} g\cdot y & \mbox{si }z=g\cdot x \\
g\cdot x & \mbox{si }z=g\cdot y \\
z & \mbox{otro caso.}\end{array} \right.$$

Es fácil verificar que $[x \mapsto y]$ es un endomorfismo y que $[x \leftrightarrow y]$ es un automorfismo de $X$.\\

Es importante mencionar que 
$$End_{G}(\mathcal{B}_{H}) \leq End_{G}(X), \forall H \leq G,$$
vía el homomorfismo de inclusión ($End_{G}(\mathcal{B}_{[H]})\hookrightarrow End_{G}(X)$).

\begin{lema}\label{generador base}
El conjunto
$$W:=\{[x \mapsto y]:\ x,y\in X, Gx\neq Gy,\ G_{x} \leq G_{y}\}$$
genera a $End_{G}(X)$ módulo $Aut_{G}(X)$. 
\end{lema}

\begin{proof}
Dado un endomorfismo $\tau\in End_{G}(X)$ y su descomposición
$$\tau=\tau_{1}\tau_{2}...\tau_{r},$$ dada en la Proposición \ref{descomposicion}, por el Lema \ref{lema2} tenemos que $G_{x} \leq G_{\tau(x)}$, para todo $x\in X$. Entonces $ \tau(\mathcal{B}_{i})\subseteq \bigcup_{i\leq j}{\mathcal{B}_{j}}$ y definimos para cada $i\in [r]$ los conjuntos
$$\mathcal{B}_{i}^{0}:=\{x\in \mathcal{B}_{i}: \tau(x)\in \mathcal{B}_{i}\}$$
$$\mathcal{B}_{i}^{1}:=\{x\in \mathcal{B}_{i}: \tau(x)\in \mathcal{B}_{j}, i < j\}.$$

Para $\epsilon\in \{0,1\}$ es fácil ver que $\mathcal{B}_{i}^{\epsilon}$ es un subconjunto $G$-invariante de $X$, entonces podemos definir endomorfismos $\tau_{i}^{\epsilon}\in End_{G}(X)$ como
$$\tau_{i}^{\epsilon}:= \left\{\begin{array}{cc} \tau_{i}(x) & \mbox{si } x\in \mathcal{B}_{i}^{\epsilon}\\
x& \mbox{otro caso.}\end{array} \right.$$

Se sigue que $\tau_{i}=\tau_{i}^{0} \tau_{i}^{1}$. Terminamos la demostración probando que $\tau_{i}^{\epsilon}\in \langle W \cup Aut_{G}(X) \rangle$, para $\epsilon\in \{0,1\}$.

\begin{enumerate}
\item[] 
Caso $\epsilon=0$.\\
Mostraremos que $\tau_{i}^{0}\in \langle W \cup Aut_{G}(X)\rangle$. En este caso, $\tau_{i}^{0}$ está contenido en un submonoide de $End_{G}(X)$ isomorfo a $End_{G}(\mathcal{B}_{i})$. Si $\alpha_{i}=1$, entonces $End_{G}(\mathcal{B}_{i})=End_{G}(Gx)=Aut_{G}(Gx)$, para algún $x\in \mathcal{B}_{i}$, así que asumimos que $\alpha_{i} \geq 2$. Recordemos que por el Lema \ref{equi2}, $End_{G}(\mathcal{B}_{i})\cong Aut_{G}(Gx)\wr Trans(\mathcal{O}_{[H_{i}]})$, para cualquier $x\in  \mathcal{B}_{i}$. El monoide $Trans(\mathcal{O}_{[H_{i}]})$ es generado por $Sym(\mathcal{O}_{[H_{i}]})$ junto con cualquier mapeo con imagen de tamaño $\alpha_{i}-1$ (ver \cite[Prop. 1.2]{undim}). Ahora, $Sym(\mathcal{O}_{[H_{i}]})$ es generado por las permutaciones en $\mathcal{O}_{[H_{i}]}$ inducidas por el conjunto 
$$\{[a\leftrightarrow b]|_{\mathcal{B}_{i}}:\ a,b\in \mathcal{B}_{i}, Ga\neq Gb,\ G_{a}=G_{b}\}\leq Aut_{G}(\mathcal{B}_{i}),$$

además, para cada $y\in \mathcal{B}_{i}$ tal que $Gx\neq Gy$, la función $[x\mapsto y]$ induce sobre $\mathcal{O}_{[H_{i}]}$ una función con imagen de tamaño $\alpha_{i}-1$. Se sigue que $End_{G}(\mathcal{B}_{i})$ está generado por $Aut_{G}(\mathcal{B}_{i})\cup \{[x\mapsto y]_{\mathcal{B}_{i}} \} $. Entonces, $\tau_{i}^{0}\in \langle W \cup Aut_{G}(X)\rangle$.

\item[] Caso $\epsilon=1$.\\
Mostramos que $\tau_{i}^{1}\in \langle W \cup Aut_{G}(X) \rangle$. Supongamos que $\mathcal{B}_{i}^{1}$ es la unión de las siguientes $G$-órbitas
$$\mathcal{B}_{i}^{1}=Gx_{1}\cup Gx_{2}\cup ... \cup Gx_{s},$$
para algunos $x_{1},x_{2},...,x_{s}\in X$. Denotamos por $y_{k}=\tau_{i}^{1}(x_{k})$, para $k\in \{1,2,...,s\}$. Como $y_{k}\notin \mathcal{B}_{i}$, entonces $Gx_{k}\neq Gy_{k}$. Entonces, tenemos la siguiente factorización para $\tau_{i}^{1}$:
$$\tau_{i}^{1}=[x_{1}\mapsto y_{1}][x_{2}\mapsto y_{2}]...[x_{s}\mapsto y_{s}]\in \langle W \cup Aut_{G}(X)\rangle.$$

\end{enumerate}

\end{proof}

Para cualesquiera $N,H\leq G$, definimos la $N$-clase de conjugación de $H$ como $[H]_{N}:=\{gHg^{-1}: g\in N\}$. Y para economizar notación expresamos $[x\ y]=[x \leftrightarrow y]$.

\begin{observacion}
Notemos que si $G_{x} \leq G_{y}$ y dado un $n\in N_{G}(G_{x})$, entonces $G_{x}\leq n^{-1}G_{y}n$. 
\end{observacion}

\begin{teorema} \label{chido3}
Sea $G$ un grupo finito actuando sobre un conjunto finito $X$. Para cada $i\in [r]$, $N_{i}:=N_{G}(H_{i})$ y sean $[K_{i,1}]_{N_{i}},[K_{i,2}]_{N_{i}},...,[K_{i,r_{i}}]_{N_{i}}$ todas las $N_{i}$-clases de conjugación de subgrupos $K_{i,j}\in Stab_{G}(X)$ tales que $H_{i}\leq K_{i,j}$. Fijamos algunos elementos en $X$ como sigue:

\begin{enumerate}
\item[$\bullet$] Para cada $i\in [r]$, fijamos $x_{i}\in \mathcal{B}_{i}$ tal que $G_{x_{i}}=H_{i}$.
\item[$\bullet$] Para cada $i\in [r]$ y cada $j\in [r_{i}]$, fijamos $y_{i,j}\in X$ tal que $G_{y_{i,j}}=K_{i,j}$.
\item[$\bullet$] Para cada $i\in [r]$ tal que $\alpha_{i}\geq 2$, fijamos $x_{i}'\in \mathcal{B}_{i}$ tal que $G_{x_{i}'}=H_{i}$ y $Gn\cdot x_{j}\neq Gx_{i}$.

\end{enumerate}
Entonces el conjunto 
$$V=\{[x_{i}\mapsto y_{i,j}]: i\in [r],j\in [r_{i}]\}\cup \{[x_{i}\mapsto x_{i}']: i\in [r], \alpha_{i}\geq 2\}$$
genera $End_{G}(X)$ módulo $Aut_{G}(X)$.

\end{teorema}

\begin{proof}
Por el Lema \ref{generador base}, basta con mostrar que cada $[x\mapsto y]$, con $x,y\in X$, $G_{x}\leq G_{y}$, puede ser expresado como producto de elementos de $V$ y de $Aut_{G}(X)$. Notemos que si $Gx=Gy$ se tiene que $[x\mapsto y]\in Aut_{G}(X)$. Podemos suponer, sin pérdida  de generalidad, que $[G_{x}]=[H_{i}]$. Entonces existe un $g\in G $ tal que $G_{g\cdot x}= g G_{x} g^{-1}=H_{i}$; como se tiene que $[x\mapsto y] = [(g\cdot x)\mapsto (g\cdot y)]$, podemos asumir, también sin pérdida  de generalidad, que $G_{x}=H_{i}$. Hay entonces dos casos a considerar.

\begin{enumerate}
\item[] Caso a) $G_{x}=G_{y}$.\\
Notemos que $G_{x}=H_{i}=G_{x_{i}}=G_{x_{i}'}=G_{y}$. Tenemos entonces 4 subcasos:

\begin{enumerate}
\item[Caso 1.-] $Gx\neq Gx_{i}$ y $Gy\neq Gx'_{i}$.\\

Aquí tenemos la factorización:
$$[x\mapsto y]= [y\ x_{i}'][x \ x_{i}][x_{i}\mapsto x_{i}'][x\ x_{i}][y\ x_{i}']\in \langle V \cup Aut_{G}(X)\rangle.$$

\item[Caso 2.-]  $Gx= Gx_{i}$ y $Gy\neq Gx'_{i}$.\\ 
Entonces existe $h\in G$ tal que $x_{i}=h\cdot x$. Como $G_{x}=G_{x_{i}}$, tenemos la siguente descomposición:
$$[x\mapsto y]= [x_{i}\mapsto y][(g\cdot x)\mapsto(gh\cdot x)],$$
como $Gy\neq Gx_{i}'$ entonces
$$[x_{i}\mapsto y]=[y\ x_{i}'][x_{i}\mapsto x_{i}'][y\ x_{i}'].$$
Así se cumple que
$$[x\mapsto y]=[y\ x_{i}'][x_{i}\mapsto x_{i}'][y\ x_{i}'][(g\cdot x)\mapsto(gh\cdot x)]\in \langle V \cup Aut_{G}(X) \rangle.$$

\item[Caso 3.-] $Gx\neq Gx_{i}$ y $Gy= Gx'_{i}$.\\ 
Notemos que existe un $k\in G$ tal que $y=k\cdot x_{i}'$, por lo que se satisface que
$$[x\mapsto y]= [(g\cdot x_{i}')\mapsto(gk\cdot x_{i}')][x\mapsto x_{i}'][(g\cdot x_{i}') \mapsto(gk^{-1}\cdot x_{i}')].$$
También se satisface que 
$$[x\mapsto x_{i}']=[x\ x_{i}][x_{i}\mapsto x_{i}'][x \ x_{i}],$$

entonces 
\begin{scriptsize}
$$[x\mapsto y]=[(g\cdot x_{i}')\mapsto(gk\cdot x_{i}')][x\ x_{i}][x_{i}\mapsto x_{i}'][x \ x_{i}][(g\cdot x_{i}')\mapsto(gk^{-1}\cdot x_{i}')]\in \langle V  \cup Aut_{G}(X)\rangle.$$
\end{scriptsize}

\item[Caso 4.-] $Gx= Gx_{i}$ y $Gy= Gx'_{i}$.\\ 
Con las supocisiones de los casos anteriores se verifica que
\begin{scriptsize}
$$[x\mapsto y]=[(g\cdot x_{i}')\mapsto(gk\cdot x_{i}')][x_{i}\mapsto x_{i}'][(g\cdot x)\mapsto(gh\cdot x)][(g\cdot x_{i}')\mapsto(gk^{-1}\cdot x_{i}')]\in \langle V\cup Aut_{G}(X)\rangle.$$  \end{scriptsize}

\end{enumerate}

\item[] Caso b) $G_{x}< G_{y}$.\\
Notemos primero que existe $j>i$ tal que $y\in \mathcal{B}_{j}$, $G_{y}=K_{i,j}$ y además existe $n\in N_{i}$ tal que   $$nG_{x_{j}}n^{-1}=G_{y}.$$

De igual manera que en el caso anterior se presentan los mismos subcasos. Las construcciones son análogas, sustituyendo el elemento $x_{i}'$ por el elemento $n\cdot x_{j}$.\\

\begin{enumerate}
\item[Caso 1.-] $Gx\neq Gx_{i}$ y $Gy\neq Gn\cdot x_{j}$.

$$[x\mapsto y]= [y\ (n\cdot x_{j})][x \ x_{i}][x_{i}\mapsto (n\cdot x_{j})][x\ x_{i}][y\ (n\cdot x_{j})].$$

\item[Caso 2.-] $Gx= Gx_{i}$ y $Gy\neq Gn\cdot x_{j}$.\\
Entonces existe $h\in G$ tal que $x_{i}=h\cdot x$. Se sigue que
$$[x\mapsto y]=[y\ (n\cdot x_{j})][x_{i}\mapsto (n\cdot x_{j})][y\ (n\cdot x_{j})][(g\cdot x)\mapsto(gh\cdot x)].$$

\item[Caso 3.-] $Gx\neq Gx_{i}$ y $Gy= Gn\cdot x_{j}$.\\
Notemos que existe un $k\in G$ tal que $y=k (n\cdot x_{j})$. 
\begin{scriptsize}
$$[x\mapsto y]=[(g\cdot (n\cdot x_{j}))\mapsto(gk\cdot (n\cdot x_{j}))][x\ x_{i}][x_{i}\mapsto (n\cdot x_{j})][x \ x_{i}][(g\cdot (n\cdot x_{j}))\mapsto(gk^{-1}\cdot (n\cdot x_{j}))].$$
\end{scriptsize}

\item[Caso 4.-] $Gx= Gx_{i}$ y $Gy= Gn\cdot x_{j}$.
\begin{scriptsize}
$$[x\mapsto y]=[(g\cdot (n\cdot x_{j}))\mapsto(gk\cdot (n\cdot x_{j}))][x_{i}\mapsto (n\cdot x_{j})][(g\cdot x)\mapsto(gh\cdot x)][(g\cdot (n\cdot x_{j}))\mapsto(gk^{-1}\cdot x_{i}')].$$  \end{scriptsize}

\end{enumerate}

\end{enumerate}

Entonces, habiendo recorrido todos los casos posibles, podemos garantizar que $V$ genera a $End_{G}(X)$ módulo $Aut_{G}(X)$.

\end{proof}

Definimos un poco de notación para expresar nuestros resultados de manera compacta. Dados $H,N$ subgrupos de $G$, definimos la $N$-clase de conjugación de $H$ como:

$$[H]_{N}:=\{g^{-1}Hg\ |\ g\in N\}.$$
También definimos el conjunto de $N$-clases que contienen a un subgrupo $H$:
$$U(H_{i}):=\{[K]_{N_{i}}:\ K\in Stab_{G}(X), H_{i}\leq K\}.$$
Y por último definimos el conjunto de índices de cajas unitarias, las cajas con una sola $G$-órbita:
$$\kappa_{G}(X):=\{i:\ \alpha_{i}=1\}.$$

\begin{corolario}
Con la notación previa.

$$rank(End_{G}(X):Aut_{G}(X))\leq |V|=\sum_{i=1}^{r}{|U(H_{i})|}-|\kappa_{G}(X)|. $$
\end{corolario}

\begin{ejemplo}
Si $A$ es un conjunto finito con al menos dos elementos y $G$ un grupo finito se tiene que:

$$\kappa_{G}(X)= \left\{\begin{array}{ll} \{i: |G/H_{i}|=2\} & \mbox{si }|A|=2 \\ 0 & \mbox{otro caso.}  \end{array}  \right.$$
Ver \cite[Lemma 5]{undim}, para más detalles.

\end{ejemplo}

\subsubsection{Colapsos elementales de transformaciones $G$-equivariantes}

Tenemos construido un conjunto $V$ el cual se demostró que genera a $End_{G}(X)$ módulo $Aut_{G}(X)$, un paso importante para nuestro objetivo principal. Sin embargo esto solo nos genera una cota para el rank relativo. En esta sección definimos un tipo especial de endomorfismos los cuales son pieza clave para completar nuestro trabajo. \\

Para que sea posible determinar el rank relativo exacto de $End_{G}(X)$ módulo $Aut_{G}(X)$ es necesario introducir un concepto. Para $\tau \in End_{G}(X)$, definimos:
$$ker(\tau):=\{(a,b)\in X \times X:\ \tau(a)=\tau(b)\}.$$
Esta es una relación de equivalencia sobre $X$ y está relacionada con la relación de Green izquierda $\mathcal{L}$ en un semigrupo de transformaciones (ver  \cite[Teorema 4.5.1]{Howie}).\\

Las siguientes propiedades del kernel de una transformación son bien conocidas. 

\begin{proposicion}
Sea $X$ un conjunto entonces se satisfacen las siguientes afirmaciones:
\begin{enumerate}
\item[i)] $ker(\tau)=\{(a,a): a\in X\}$ si y solo si $\tau$ es una biyección.
\item[ii)] Si $\tau,\sigma \in End_{G}(X)$, entonces $ker(\tau)\subseteq ker(\sigma\tau)$.
\item[iii)] $(a,b)\in ker(\sigma \tau)$ si y solo si $(\tau(a),\tau(b))\in ker(\sigma)$. 

\item[iv)] Si $\tau\in Aut_{G}(X)$, entonces $|ker(\sigma)|=|ker(\sigma \tau)|$.

\end{enumerate}

\end{proposicion}

Introducimos una caracterización de algunos tipos de endomorfismos la cual nos permite descartar endomorfismos dispensables dentro de cualquier conjunto generador.

\begin{definicion}{(Colapso elemental)}

Un colapso elemental es un endomorfismo $\tau \in End_{G}(X)$ tal que existen $x,y\in X$ que satisfacen que:
$$ker(\tau)=\{(g\cdot x,g\cdot y),(g\cdot y, g\cdot x): g\in G\}\cup\{(a,a): a\in X\}.$$
\end{definicion}

\begin{definicion}{(Tipo de colapso)}

Sea $K\in Stab_{G}(X)$ tal que $H_{i}\leq K \leq G$. Decimos que $\tau\in End_{G}(X)$ es un colapso elemental de tipo $(i,[K]_{N_{i}})$ si existen $x,y\in X$, con $Gx\neq Gy$, tales que:

 $$G_{x}=H_{i} \mbox{ y }[G_{\tau(x)}]_{N{i}}=[G_{y}]_{N_{i}}=[K]_{N_{i}}.$$

\end{definicion}

\begin{ejemplo}
Sean $x,y\in X$ tales que $G_{x}=H_{i}$ y además $[G_{y}]_{N_{i}}=[K]_{N_{i}}$, entonces la transformación $[x\mapsto y]$ es un colapso elemental de tipo $(i,[K]_{N_{i}})$.
\end{ejemplo}

\begin{lema}
Supongamos que $\tau \in End_{G}(X)$ es un colapso elemental de tipo $(i,[K]_{N_{i}})$ y además es un colapso elemental de tipo $(i',[K']_{N_{i}})$. Entonces $i=i'$ y $[K]_{N_{i}}=[K']_{N_{i}}$. 

\end{lema}

\begin{proof}

Como $\tau$ es un colapso elemental de dichos tipos, existen entonces $x,x',y,y'\in X$ tales que $Gx \neq Gy$ y $Gx'\neq Gy'$, tales que:
$$\{(g\cdot x, g\cdot y),(g\cdot y, g\cdot x):\ g\in G\}=\{g\cdot x', g\cdot y'),(g\cdot y', g\cdot x'):\ g\in G\}.$$
Esto nos dice que $Gx=Gx'$ y $Gy=Gy'$. Lo cual también implica que $G_{x}$ y $G_{x'}$ son conjugados, i.e., $[H_{i}]=[G_{x}]=[G_{x'}]=[H_{i'}]$, lo cual implica directamente que $i=i'$.  Más aún, existe $g\in G$ tal que $x=g\cdot x'$, entonces:
$$H_{i}=G_{x}=G_{g\cdot x'}=gG_{x'}g^{-1}=gH_{i}g^{-1} \iff g\in N_{i}=N_{G}(H_{i}).$$
Además, 
$$G_{\tau(x)}=G_{\tau(g\cdot x')}=G_{g\cdot \tau(x')}=gG_{\tau(x')}g^{-1},$$
lo cual implica que $[G_{\tau(x)}]_{N_{i}}=[G_{\tau(x')}]_{N_{i}}$ y entonces $[K]_{N_{i}}=[K']_{N_{i}}$.
\end{proof}

\begin{lema}\label{cantidad}
Sea $G$ un grupo finito que actúa sobre un conjunto finito $X$, entonces el número de posibles tipos de colapsos elementales es $$\sum_{i=i}^{r}{|U(H_{i})|}-|\kappa_{G}(X)|.$$

\end{lema}

\begin{proof}
No podemos definir más tipos de colapsos elementales a los que ya se definieron, aún si se utilizaran distintos puntos fijos a partir de conjugaciones se podría probar que son equivalentes y además todos están incluidos en el conjunto $V$, entonces el resultado se sigue de que un colapso elemental de tipo $(i,[K]_{N_{i}})$, el cual es el único posible para una sola caja, existe si y solo si $\alpha_{[H_{i}]}\neq 1$. 
\end{proof}

\begin{teorema}\label{chido2}
Sea $G$ un grupo finito actuando sobre un conjunto finito $X$. Sean $[H_{1}],[H_{2}],...,[H_{r}]$ la lista de clases de conjugación de subgrupos de $G$ que están en $Stab_{G}(X)$, entonces:
$$rank(End_{G}(X):Aut_{G}(X))=\sum_{i=1}^{r}{|U(H_{i})|}-|\kappa_{G}(X)|.$$
\end{teorema}

\begin{proof}
Sea $W\subseteq End_{G}(X)$ tal que $\langle Aut_{G}(X) \cup W\rangle= End_{G}(X)$. Basta con probar que $W$ debe contener por lo menos un colapso elemental de cada tipo posible. \\

Sea $K\in Stab_{G}(X)$ tal que $H_{i}\leq K \leq G$. Y sean $x,y\in X$ tales que $G_{x}=H_{i}$ y $[G_{y}]_{N_{i}}=[K]_{N_{i}}$. Notemos que $[x\mapsto y]$ es un colapso elemental de tipo $(i,[K]_{N_{i}})$ en $End_{G}(X)$ y por tanto existen $\tau_{1},\tau_{2},...,\tau_{s}\in Aut_{G}(X)\cup W$ tales que:
$$[x\mapsto y]=\tau_{1}\tau_{2}...\tau_{s}.$$
Si fijamos $t\in [s]$ y $m\in [r]\setminus {i}$, afirmamos que:

\begin{enumerate}
\item[1.-] $\tau_{t}(\mathcal{B}_{m})= \mathcal{B}_{m}$ y $G_{z}=G_{\tau_{t}(z)}$ para todo $z\in \mathcal{B}_{m}$.
\item[2.-] Si $y\in \mathcal{B}_{i}$, entonces $\tau_{t}(\mathcal{B}_{i})\subseteq \mathcal{B}_{i}$ y $G_{z}=G_{\tau_{t}(z)}$ para todo $z\in \mathcal{B}_{i}$.
\end{enumerate}

Ambas afirmaciones se siguen de las definiciones de los $\tau_{t}$, ya que $\tau_{t}$ se comporta como la identidad para cualquier $m \neq t$ y el Lema \ref{lema2} nos dice que solo podemos mandar elementos en cajas a elementos en cajas que se encuentren más arriba, es decir que $t\leq m$. Y precisamente la igualdad entre cajas solo es posible cuando $t=m$. La igualdad de los estabilizadores es consecuencia de que $G$ es finito; solo podemos pasar entre $G$-órbitas dentro de una misma caja si se cumple la igualdad entre estabilizadores. \\

Sea $k$ el índice más grande para el cual $\tau_{k}$ es no-invertible. Debemos probar que $\tau_{k}\in W$ y para tal efecto basta probar que $\tau_{k}$ es un colapso elemental de tipo $(i,[K]_{N_{i}})$, para algún $i$, esto verificando las 2 propiedades de su definición. 

\begin{enumerate}
\item[i)] Por $G$-equivarianza $\tau(h\cdot a)= \tau(h\cdot b)$ para todo $h \in G$. Por las propiedades del kernel tenemos que 
$$ker(\tau_{1}\tau_{2}...\tau_{k}) \subseteq ker(\tau_{1}\tau_{2}...\tau_{s})=ker([x\mapsto y]).$$
Más aún, como $\tau_{1}\tau_{2}...\tau_{k-1}\in Aut_{G}(X)$ tenemos que
$$|ker(\tau_{k})|=|ker(\tau_{1}\tau_{2}...\tau_{k})|\leq |ker([x\mapsto y])|.$$
Entonces tenemos que 
$$ker(\tau_{k})=\{(h\cdot a,h\cdot b),(h\cdot b, h\cdot a):\ h\in G\}\cup \{(c,c):\ c\in X\},$$
lo que nos dice que $\tau_{k}$ es un colapso elemental.

\item[ii)] Como $\tau_{k}$ es no-invertible, entonces existen $a,b\in X$, $a\neq b$, tales que $\tau_{k}(a)=\tau_{k}(b)$. Sean $a'=\tau_{s}^{-1}\tau_{s-1}^{-1}...\tau_{k+1}^{-1}(a)$ y $b'=\tau_{s}^{-1}\tau_{s-1}^{-1}...\tau_{k+1}^{-1}(b)$. Entonces
$$\tau_{1}\tau_{2}...\tau_{s}(a')=\tau_{1}...\tau_{k-1}\tau_{k}(a)=\tau_{1}...\tau_{k-1}\tau_{k}(b)=\tau_{1}...\tau_{s}(b')$$
implica que $(a',b')\in ker([x\mapsto y])$. Luego, podemos suponer sin pérdida  de generalidad que, existe un $g\in G$ tal que:
$$x=g\cdot a'=g\cdot \tau_{s}^{-1}\tau_{s-1}^{-1}...\tau_{k+1}^{-1}(a)=\tau_{s}^{-1}\tau_{s-1}^{-1}...\tau_{k+1}^{-1}(g\cdot a)$$
$$y=g\cdot b'=g\cdot \tau_{s}^{-1}\tau_{s-1}^{-1}...\tau_{k+1}^{-1}(b)=\tau_{s}^{-1}\tau_{s-1}^{-1}...\tau_{k+1}^{-1}(g\cdot b).$$

Entonces, por el Lema \ref{lema1} se tiene que $G_{x}=G_{g\cdot a}$ y $G_{y}=G_{g\cdot  b}$. Como $a$ debe pertenecer a alguna caja, podemos asumir que $a\in \mathcal{B}_{i}$ y entonces existe un elemento en $G$ tal que al conjugar $G_{a}$ por este elemento nos da de resultado el representante de clase $H_{i}$, podemos suponer sin pérdida  de generalidad que este elemento es $g$, entonces se satisface que $H_{i}=gG_{a}g^{-1}=G_{g\cdot a}=G_{x}$. De las afirmaciones inferimos que $G_{\tau(g\cdot b)}=G_{g\cdot b}=G_{y}$, entonces $G_{\tau(g\cdot a)}=G_{y}$. Entonces
$$[G_{\tau(g\cdot a)}]_{N_{i}}=[G_{b}]_{N_{i}}=[G_{y}]_{N_{i}}=[K]_{N_{i}},$$
lo que nos dice que $\tau$ es de tipo $(H_{i},[K]_{N_{i}})$.

\end{enumerate}
Lo que nos prueba que $W$ debe de contener un colapso elemental de cada tipo posible. Nuestro resultado se sigue a partir del Lema \ref{cantidad} y el Teorema \ref{chido3}.

\end{proof}

Este resultado nos permite encontrar una condición necesaria para que el rank relativo del monoide de endomorfismos sea finito, aún cuando se trabaje con grupos y conjuntos infinitos.

\begin{proposicion}\label{usefull}
Sea $G$ un grupo que actúa sobre un conjunto $X$,  si  $|Conj_{G}(X)|= \infty$ entonces se tiene que $rank(End_{G}(X): Aut_{G}(X)) = \infty$.

\end{proposicion}

\begin{proof}
Notemos que 
$$U(H_{1})= \{K\in Stab_{G}(X): H_{1} \leq K\}= Stab_{G}(X),$$
además tenemos que 
$$|Conj_{G}(X)|= \infty \Rightarrow |Stab_{G}(X)|= \infty  \Rightarrow |U(H_{1})|=\infty.$$
Es decir que, como marca el Teorema \ref{chido2}, un conjunto generador módulo $Aut_{G}(X)$ debe contener un colapso elemental de cada tipo, los cuales están en correspondencia con las $N$-clases de conjugación y esta cantidad  es infinita.
\end{proof}

\begin{lema}
Sea $X$ un conjunto finito, entonces $|Conj_{G}(X)|<\infty$.
\end{lema}

\begin{proof}
La aplicación $x \mapsto G_{x}$ es sobreyectiva de $X$ en $Stab_{G}(X)$.

\end{proof}

\subsection{Endomorfismos de $G$-conjuntos infinitos}
En esta sección presentamos unos pocos avances obtenidos para el caso en el que el conjunto $X$ es infinito. Utilizamos herramientas análogas a las encontradas en el Apéndice A y las encontradas en \cite{Ara} para poder determinar el rank relativo de $End_{G}(X)$ módulo su grupo de unidades. También se presentan ideas las cuales encaminan nuestros esfuerzos a determinar completamente los resultados para el caso cuando tanto $G$ como $X$ son infinitos. \\

La demostración del Teorema \ref{chido2}, así como en la Proposición \ref{usefull}, nos dan herramientas a partir de los endomorfismos de tipo $(i,[K]_{N_{i}})$. Notamos que si $G$ es un grupo infinito, puede no ser tan sencillo indexar los subgrupos de $G$, si hay una cantidad no numerable de ellos, entonces adaptamos la definición a ser un endomorfismo de tipo $(H,[K]_{N})$, donde $N=N_{G}(X)$ y todo esto nos da una herramienta poderosa para acotar el rank relativo de $End_{G}(X)$: 

$$rank(End_{G}(X):Aut_{G}(X))\geq \sum_{i=1}^{r}{|U(H_{i})|}-|\kappa_{G}(X)|.$$
Sin embargo el conjunto generador propuesto en el Teorema \ref{chido3} puede no generar todo $End_{G}(X)$ módulo $Aut_{G}(X)$ en el caso en el que $X$ sea un conjunto infinito. Es aquí donde proponemos un nuevo conjunto generador a partir de las siguientes construcciones.

\subsubsection{Endomorfimos en $\mathcal{B}_{[H]}$}
Dado un $G$-conjunto $X$, para una transformación $\tau \in End_{G}(X)$ definimos las siguientes cantidades:\\

\textbf{Rango $G$-equivariante}. $$rank_{G}(\tau):=|\{Gx\in X/G:\ x\in \tau(X)\}|.$$
\textbf{Defecto $G$-equivariante}.
$$def_{G}(\tau):=|\{Gx\in X/ G:\ x\in X\setminus \tau(X)\}|.$$ 
Para una $G$-órbita $Gx\in X/G$, definimos el conjunto:  $$B_{Gx}(\tau):=\{Gy\in X/G:\ \tau(Gy)=Gx\}.$$
\textbf{Índice de contracción infinito $G$-equivariante}.

$$k_{G}(\tau):=|\{Gx\in X/G:\ |B_{Gx}(\tau)|=|X|\}|.$$

\begin{proposicion}
Sea $G$ un grupo finito y $X$ un conjunto. Dado un subgrupo $H$ de $G$. Se tiene que para una función $\tau\in End_{G}(\mathcal{B}_{[H]})$ se satisface que:
\begin{enumerate}
\item[i)] $$rank_{G}(\tau)=\frac{rank(\tau)}{|G|}.$$
\item[ii)] $$def_{G}(\tau)=\frac{def(\tau)}{|G|}.$$
\item[iii)] $$k_{G}(\tau)=\frac{k(\tau)}{|G|}.$$

\end{enumerate}
\end{proposicion}

\begin{proof}
Sea $\tilde{X}$ el conjunto de representantes de órbita de $\tau(X)$, i.e.,
$$\tilde{X}=\{x,y\in \tau(X): Gx\neq Gy,\ G_{x}=G_{y}\},$$
entonces la función $\psi:G\times \tilde{X} \rightarrow \tau(X)$ dada por $\psi(g,x_{\lambda})=g\cdot x_{\lambda}$ es una biyección. Esto prueba la afirmación i). Construcciones análogas con los conjuntos 
$$X'=\{x,y\in X\setminus \tau(X), \ Gx\neq Gy,\ G_{x}=G_{y}\} $$
y 
$$ \hat{X}=\{x,y\in X: |B_{Gx}(\tau)|=|B_{Gy}(\tau)|=|X|,\ Gx\neq Gy,\ G_{x}=G_{y}\}$$
prueban ii) y iii).
\end{proof}

\begin{corolario}
Sea $X_{i}$ un cojunto tal que $|X|=\infty$ y dada una función $f:X\rightarrow X$, tal que $rank(f)=|X|$. Entonces se cumple que $rank_{G}(\tilde{f})=|X|$.

\end{corolario}

Definimos ahora las extensiones de las funciones, sea $H$ un subgrupo de $G$ y sea  $\Lambda$ un conjunto de índices de cardinal igual a $|\mathcal{B}_{[H]}/G|$, entonces podemos fijar elementos $x_{\lambda}\in \mathcal{B}_{[H]}$, $\lambda \in \Lambda$, sin pérdida  de genereralidad podemos tomarlos tales que $Gx_{\lambda}\neq Gx_{\kappa}$ y $G_{x_{\lambda}}=G_{x_{\kappa}}$,  $\lambda,\kappa\in \Lambda$.\\

\begin{center} 

\tikzset{every picture/.style={line width=0.75pt}} 

\begin{tikzpicture}[x=0.65pt,y=0.65pt,yscale=-1,xscale=1]

\draw    (86,687) -- (589.4,687.32) ;
\draw    (87,868) -- (588.4,868.32) ;
\draw    (86,687) -- (87,868) ;
\draw   (114,767.32) .. controls (114,726.45) and (123.04,693.32) .. (134.2,693.32) .. controls (145.36,693.32) and (154.4,726.45) .. (154.4,767.32) .. controls (154.4,808.19) and (145.36,841.32) .. (134.2,841.32) .. controls (123.04,841.32) and (114,808.19) .. (114,767.32) -- cycle ;
\draw   (164,767.32) .. controls (164,726.45) and (173.04,693.32) .. (184.2,693.32) .. controls (195.36,693.32) and (204.4,726.45) .. (204.4,767.32) .. controls (204.4,808.19) and (195.36,841.32) .. (184.2,841.32) .. controls (173.04,841.32) and (164,808.19) .. (164,767.32) -- cycle ;
\draw   (214,768.32) .. controls (214,727.45) and (223.04,694.32) .. (234.2,694.32) .. controls (245.36,694.32) and (254.4,727.45) .. (254.4,768.32) .. controls (254.4,809.19) and (245.36,842.32) .. (234.2,842.32) .. controls (223.04,842.32) and (214,809.19) .. (214,768.32) -- cycle ;
\draw   (264,768.32) .. controls (264,727.45) and (273.04,694.32) .. (284.2,694.32) .. controls (295.36,694.32) and (304.4,727.45) .. (304.4,768.32) .. controls (304.4,809.19) and (295.36,842.32) .. (284.2,842.32) .. controls (273.04,842.32) and (264,809.19) .. (264,768.32) -- cycle ;
\draw   (315,767.32) .. controls (315,726.45) and (324.04,693.32) .. (335.2,693.32) .. controls (346.36,693.32) and (355.4,726.45) .. (355.4,767.32) .. controls (355.4,808.19) and (346.36,841.32) .. (335.2,841.32) .. controls (324.04,841.32) and (315,808.19) .. (315,767.32) -- cycle ;
\draw   (365,767.32) .. controls (365,726.45) and (374.04,693.32) .. (385.2,693.32) .. controls (396.36,693.32) and (405.4,726.45) .. (405.4,767.32) .. controls (405.4,808.19) and (396.36,841.32) .. (385.2,841.32) .. controls (374.04,841.32) and (365,808.19) .. (365,767.32) -- cycle ;
\draw   (415,768.32) .. controls (415,727.45) and (424.04,694.32) .. (435.2,694.32) .. controls (446.36,694.32) and (455.4,727.45) .. (455.4,768.32) .. controls (455.4,809.19) and (446.36,842.32) .. (435.2,842.32) .. controls (424.04,842.32) and (415,809.19) .. (415,768.32) -- cycle ;
\draw   (465,768.32) .. controls (465,727.45) and (474.04,694.32) .. (485.2,694.32) .. controls (496.36,694.32) and (505.4,727.45) .. (505.4,768.32) .. controls (505.4,809.19) and (496.36,842.32) .. (485.2,842.32) .. controls (474.04,842.32) and (465,809.19) .. (465,768.32) -- cycle ;
\draw    (86,723) -- (589.4,723.32) ;
\draw  [color={rgb, 255:red, 0; green, 0; blue, 0 }  ,draw opacity=1 ][fill={rgb, 255:red, 0; green, 0; blue, 0 }  ,fill opacity=1 ] (130,710.02) .. controls (130,707.98) and (131.66,706.32) .. (133.7,706.32) .. controls (135.74,706.32) and (137.4,707.98) .. (137.4,710.02) .. controls (137.4,712.06) and (135.74,713.72) .. (133.7,713.72) .. controls (131.66,713.72) and (130,712.06) .. (130,710.02) -- cycle ;
\draw  [color={rgb, 255:red, 0; green, 0; blue, 0 }  ,draw opacity=1 ][fill={rgb, 255:red, 0; green, 0; blue, 0 }  ,fill opacity=1 ] (181,710.02) .. controls (181,707.98) and (182.66,706.32) .. (184.7,706.32) .. controls (186.74,706.32) and (188.4,707.98) .. (188.4,710.02) .. controls (188.4,712.06) and (186.74,713.72) .. (184.7,713.72) .. controls (182.66,713.72) and (181,712.06) .. (181,710.02) -- cycle ;
\draw  [color={rgb, 255:red, 0; green, 0; blue, 0 }  ,draw opacity=1 ][fill={rgb, 255:red, 0; green, 0; blue, 0 }  ,fill opacity=1 ] (230,710.02) .. controls (230,707.98) and (231.66,706.32) .. (233.7,706.32) .. controls (235.74,706.32) and (237.4,707.98) .. (237.4,710.02) .. controls (237.4,712.06) and (235.74,713.72) .. (233.7,713.72) .. controls (231.66,713.72) and (230,712.06) .. (230,710.02) -- cycle ;
\draw  [color={rgb, 255:red, 0; green, 0; blue, 0 }  ,draw opacity=1 ][fill={rgb, 255:red, 0; green, 0; blue, 0 }  ,fill opacity=1 ] (280,711.02) .. controls (280,708.98) and (281.66,707.32) .. (283.7,707.32) .. controls (285.74,707.32) and (287.4,708.98) .. (287.4,711.02) .. controls (287.4,713.06) and (285.74,714.72) .. (283.7,714.72) .. controls (281.66,714.72) and (280,713.06) .. (280,711.02) -- cycle ;
\draw  [color={rgb, 255:red, 0; green, 0; blue, 0 }  ,draw opacity=1 ][fill={rgb, 255:red, 0; green, 0; blue, 0 }  ,fill opacity=1 ] (331,711.02) .. controls (331,708.98) and (332.66,707.32) .. (334.7,707.32) .. controls (336.74,707.32) and (338.4,708.98) .. (338.4,711.02) .. controls (338.4,713.06) and (336.74,714.72) .. (334.7,714.72) .. controls (332.66,714.72) and (331,713.06) .. (331,711.02) -- cycle ;
\draw  [color={rgb, 255:red, 0; green, 0; blue, 0 }  ,draw opacity=1 ][fill={rgb, 255:red, 0; green, 0; blue, 0 }  ,fill opacity=1 ] (382,711.02) .. controls (382,708.98) and (383.66,707.32) .. (385.7,707.32) .. controls (387.74,707.32) and (389.4,708.98) .. (389.4,711.02) .. controls (389.4,713.06) and (387.74,714.72) .. (385.7,714.72) .. controls (383.66,714.72) and (382,713.06) .. (382,711.02) -- cycle ;
\draw  [color={rgb, 255:red, 0; green, 0; blue, 0 }  ,draw opacity=1 ][fill={rgb, 255:red, 0; green, 0; blue, 0 }  ,fill opacity=1 ] (432,711.02) .. controls (432,708.98) and (433.66,707.32) .. (435.7,707.32) .. controls (437.74,707.32) and (439.4,708.98) .. (439.4,711.02) .. controls (439.4,713.06) and (437.74,714.72) .. (435.7,714.72) .. controls (433.66,714.72) and (432,713.06) .. (432,711.02) -- cycle ;
\draw  [color={rgb, 255:red, 0; green, 0; blue, 0 }  ,draw opacity=1 ][fill={rgb, 255:red, 0; green, 0; blue, 0 }  ,fill opacity=1 ] (482,711.02) .. controls (482,708.98) and (483.66,707.32) .. (485.7,707.32) .. controls (487.74,707.32) and (489.4,708.98) .. (489.4,711.02) .. controls (489.4,713.06) and (487.74,714.72) .. (485.7,714.72) .. controls (483.66,714.72) and (482,713.06) .. (482,711.02) -- cycle ;

\draw (520,750.2) node [anchor=north west][inner sep=0.75pt]    {$\dotsc $};
\draw (551,658.32) node [anchor=north west][inner sep=0.75pt]    {$\mathcal{B}_{[H]}$};
\draw (52,692.32) node [anchor=north west][inner sep=0.75pt]    {$G_{x}{}_{_{\lambda }}$};
\draw (114,840.32) node [anchor=north west][inner sep=0.75pt]    {$Gx_{\lambda }$};
\draw (165,841.32) node [anchor=north west][inner sep=0.75pt]    {$Gx_{\kappa}$};

\end{tikzpicture}

\end{center}
Ahora construiremos una extensión de funciones en $\mathcal{B}_{[H]}$ a partir de transformaciones de un conjunto de representantes de cada órbita. Además extenderemos unas definiciones también al monoide de funciones $G$-equivariantes.\\

Sea $X$ el conjunto de representantes de $G$-órbitas. Primero observemos que cualquier elemento $z\in \mathcal{B}_{[H]}$ satisface que existe un $g\in G$ y un $\lambda\in \Lambda$ tal que $z=g\cdot x_{\lambda}$.\\ Además 
notemos que para cualquier función $f:X \rightarrow X$, podemos extenderla a todo $\mathcal{B}_{[H]}$ como:
/$$\tilde{f}(g\cdot x_{\lambda})=g\cdot f(x_{\lambda}).$$ 

Notemos que $\tilde{f}$ está bien definida, pues si tenemos $z\in \mathcal{B}_{[H]}$ tal que $z=g\cdot x=h\cdot x$, entonces debe existir algún $\lambda\in \Lambda$ y algún $g'\in G$ tal que $x=g'\cdot x_{\lambda}$. Entonces se tiene que $gg'\cdot x_{\lambda}=hg'\cdot x_{\lambda}$. Podemos reescribir esta expresión, sin pérdida  de generalidad, como $g\cdot x_{\lambda}=h\cdot x_{\lambda}$. Luego se tiene que $h^{-1}g\in G_{x_{\lambda}}$. Ahora como
$$\tilde{f}(h^{-1}g\cdot x_{\lambda})=h^{-1}g\cdot f(x_{\lambda})$$
y se tiene que $G_{x_{\lambda}}=G_{f(x_{\lambda})}$. Entonces tenemos que 
$$h^{-1}g\cdot f(x_{\lambda})= f(x_{\lambda})\Rightarrow g\cdot f(x_{\lambda})=h\cdot f(x_{\lambda}) \Rightarrow \tilde{f}(g\cdot x_{\lambda})=\tilde{f}(h\cdot x_{\lambda}).$$
También debemos verificar que $\tilde{f}$ sea $G$-equivariante. Dado $z\in \mathcal{B}_{[H]}$, entonces existe un $g\in G$, un $\lambda \in\Lambda$ y un  $x_{\lambda}\in X_{i}$, tal que $z= g\cdot x_{\lambda}$. Luego tenemos que
$$\tilde{f}(h\cdot z)=\tilde{f}(hg\cdot x_{\lambda})= hg\cdot f(x_{\lambda})=h\cdot \tilde{f}(g\cdot x_{\lambda})=h\cdot \tilde{f}(z).$$
\begin{proposicion}\label{esta}
Dada $f:\tilde{X}\rightarrow \tilde{X}$ una transformación del conjunto de representantes de $G$-órbitas de un $G$-conjunto $X$, entonces se satisfacen las siguientes:  
\begin{enumerate}
\item[i)] Si $f$ es inyectiva, entonces $\tilde{f}$ es inyectiva.
\item[ii)] Si $f$ es sobreyectiva, entonces $\tilde{f}$ es sobreyectiva. 
\end{enumerate}
\end{proposicion}
\begin{proof}
Supongamos que $f$ es inyectiva. Si suponemos que $\tilde{f}(z_{1})=\tilde{f}(z_{2})$, entonces existen 2 índices no por fuerza iguales $\lambda$ y $\kappa$ y dos elementos $g,h$ en el grupo tales que $z_{1}=g\cdot x_{\lambda}$ y $z_{2}=h\cdot x_\kappa$. Entonces tenemos que 
$$\tilde{f}(z_{1})=\tilde{f}(z_{2}) \Rightarrow g\cdot f(x_{\lambda})=h\cdot f(x_{\kappa}).$$
Como el elemento $f(x)$ es único en cada órbita, por que así está definida la $f$. Entonces la igualdad anterior mete los elementos en una sola $G$-órbita y $f(x_{\lambda})=f(x_{\kappa})$. Y por la inyectividad de $f$ se tiene que  $x_{\lambda}=x_{\kappa}$. Luego tenemos que como $g\cdot f(x_{\lambda})=h\cdot f(x_{\kappa})$ implica que $h^{-1}g\in G_{f(x_{\kappa})}$, por construcción los estabilizadores de todos los $x_{\lambda}$ son iguales y por consecuencia también de todas sus imágenes $f(x_{\lambda})$, entonces   $h^{-1}g\cdot x_{\lambda}=x_{\lambda}=x_{\kappa}$ implica que $g\cdot x_{\lambda}=h\cdot x_{\kappa}$. Por lo tanto $\tilde{f}$ es inyectiva.\\

Supongamos ahora que $f$ es sobreyectiva. Por construcción, para cualquier elemento de la forma $g\cdot x\in \mathcal{B}_{[H]}$ existe un $g'\in G$ tal que $x=g'\cdot x_{\lambda}$, con $x_{\lambda}\in X$, para algún $\lambda$. Como $f$ es sobreyectiva, existe un elemento $\overline{x}\in X$ tal que $f(\overline{x})=x_{\lambda}$. Es fácil verificar que $\tilde{f}(gg'\cdot \overline{x})=g\cdot x$.
\end{proof}

\begin{corolario}
Si $f\in Sym(X)$, entonces $\tilde{f}\in Aut_{G}(\mathcal{B}_{[H]})$.
\end{corolario}

Entonces podemos tomar una función $\mu:X \rightarrow X $ inyectiva con defecto $|X |$ y una función $\nu :X \rightarrow X $, que sea sobreyectiva con índice de contracción infinito $|X |$ y extenderlas a toda la caja $\mathcal{B}_{[H ]}$. Tenemos pues funciones $\tilde{\mu }, \tilde{\nu }:\mathcal{B}_{[H ]}\rightarrow \mathcal{B}_{[H ]}$ definidas como:

$$\tilde{\mu }(g\cdot x_{\lambda})=g\cdot \mu (x_{\lambda})$$
$$\tilde{\nu }(g\cdot x_{\lambda})=g\cdot \nu (x_{\lambda}),$$
para cualesquiera $g\cdot x_{\lambda}\in \mathcal{B}_{[H ]}$, $x_{\lambda}\in X$ y $g\in G$.

\begin{proposicion}
Dada una función $\mu:X \rightarrow X $ inyectiva con defecto $|X |$ y una función $\nu:X \rightarrow X $, que sea sobreyectiva con índice de contracción infinito $|X |$. Entonces se tiene que:
\begin{enumerate}
\item[i)] $\tilde{\mu}$ es inyectiva.
\item[ii)] $\tilde{\nu}$ es sobreyectiva.
\item[iii)] $def_{G}(\tilde{\mu})=|X |$.
\item[iv)] $k_{G}(\tilde{\nu})=|X |$.
\end{enumerate}
\end{proposicion}

\begin{proof}

El punto i) y ii) son consecuencias de la Proposicion \ref{esta}. \\

Para los puntos iii) y iv) basta con notar que la asignación $x_{\lambda}\mapsto Gx_{\lambda}$ es una biyección entre el conjunto $X $ y el conjunto de $G$-órbitas dentro de la caja $\mathcal{B}_{[H ]}$ y que restringir este mapeo a los conjuntos $$\{x\in X : x\notin \mu(X )\} \mbox{ y }\{Gx\in \mathcal{B}_{[H ]}/G: Gx \nsubseteq \tilde{\mu}(\mathcal{B}_{[H ]})\},$$ sigue siendo una biyección muestra la afirmación iii).  De igual manera, para la afirmación iv) basta ver que el mismo mapeo $x_{\lambda} \mapsto Gx_{\lambda}$ es una biyección entre los conjuntos 
$$\{x_{\lambda}\in X : |\nu^{-1}(x_{\lambda})|=|X |\}\mbox{ y }\{Gx\in \mathcal{B}_{[H ]}/G:\ |B_{Gx}(\tilde{\nu})|=|X| \}.$$

\end{proof}

\subsubsection{Conjuntos infinitos}

En esta sección se generan las construcciones que nos permiten determinar un conjunto generador para el monoide de endomorfismos para el caso en que $X$ es finito. Empezamos trabajando dentro de las cajas, particularmente en las cajas infinitas, pues no existen condiciones que garanticen que toda caja debe de ser infinita si el conjunto $X$ lo es. Se utilizan construcciones análogas a las encontradas en \cite{cite3} con las funciones definidas para extender una transformación del conjunto de representantes de órbita a toda la caja.\\

\begin{proposicion}
Sea $G$ un grupo, un subgrupo $H_{i}$ de $G$ y $A$ un conjunto arbitrario infinito. Y sea $\tau\in End_{G}(\mathcal{B}_{[H_{i}]})$ un transformación $G$-equivariante arbitraria, entonces $End_{G}(\mathcal{B}_{[H_{i}]})\neq\langle Aut_{G}(\mathcal{B}_{[H_{i}]}),\tau \rangle$.
\end{proposicion}

\begin{proof}
Supongamos que $End_{G}(\mathcal{B}_{[H_{i}]})=\langle Aut_{G}(\mathcal{B}_{[H_{i}]}),\tau \rangle$. Tomemos una transformación $\kappa\in End_{G}(\mathcal{B}_{[H_{i}]})\setminus Aut_{G}(\mathcal{B}_{[H_{i}]})$, la cual sea inyectiva. Entonces $\kappa$ puede escribirse como producto de elementos de $Aut_{G}(\mathcal{B}_{[H_{i}]})\cup \{\tau\}$. Digamos
$$\kappa=\beta_{1}\beta_{2}\dots \beta_{k}.$$
Como $\kappa$ no es biyectiva, entonces algún $\beta_{i}$ debe no ser biyectiva y por tanto debe ser igual a $\tau$, digamos $\beta_{j}$. Entonces sucede que
$$\beta_{j-1}^{-1}\cdots \beta_{1}^{-1} \kappa = \tau\beta_{j+1}\cdots \beta_{k}. $$
Notemos que del lado izquierdo tenemos la composición de funciones inyectivas y por consecuencia, como la composición de funciones inyectivas sigue siendo una función inyectiva, $\tau$ debe ser inyectiva también. Entonces $\langle Aut_{G}(\mathcal{B}_{[H_{i}]}),\tau\rangle$ consiste enteramente de funciones inyectivas y por consecuencia no puede ser todo $End_{G}(\mathcal{B}_{[H_{i}]})$.
\end{proof}

\begin{proposicion}
Sea $G$ un grupo finito, un subgrupo $H_{i}$ de $G$ y $X$ un conjunto arbitrario infinito. Sea $\mathcal{J}_{i}$ el conjunto de todas las transformaciones $G$-equivariantes  de $\mathcal{B}_{[H_{i}]}$ en sí mismo cuyo rango $G$-equivariante sea igual a $|X|$. \\
i.e., $$\mathcal{J}_{i}:=\{\tau\in End_{G}(\mathcal{B}_{[H_{i}]}):\ rank_{G}(\tau)=|X|\}.$$  
Entonces $$End_{G}(\mathcal{B}_{[H_{i}]})=\langle \mathcal{J}_{i}\rangle.$$
\end{proposicion}

\begin{proof}
Tenemos una función $\tilde{\tau}\in End_{G}(\mathcal{B}_{[H_{i}]})$.\\
Primero notemos que $X_{i}$ se puede dividir como $X_{i}=Y_{i}\sqcup Z_{i}$, tal que $|X_{i}|=|Y_{i}|=|Z_{i}|$. Tomamos una biyección entre $X_{i}$ y $Y_{i}$, $\beta:X_{i}\rightarrow Y_{i}$. Se tiene que como $rank(\beta)=|Y_{i}|=|X_{i}|$, entonces $rank_{G}(\tilde{\beta})=|X_{i}|$, i.e., $\tilde{\beta}\in \mathcal{J}_{i}$.\\

Definimos una funcion $\tilde{\gamma}:X_{i}\rightarrow X_{i}$ como:

$$\tilde{\gamma}(z)=\left\{  \begin{array}{cc}
\tilde{\tau} \tilde{\beta^{-1}}(z)& z\in\tilde{\beta}(\mathcal{B}_{[H_{i}]})\\
z& \mbox{otro caso}\end{array}\right. .$$

Notemos que $\forall z\in Z_{i}$ se tiene que $z\notin \beta(X_{i})$ y por tanto  $z\notin \tilde{\beta}(X_{i})$, entonces $Z_{i}\subset \tilde{\gamma}(X_{i})$ y por $G$-equivarianza se tiene que $\{Gz\in \mathcal{B}_{[H_{i}]}: z\in Z_{i}\} \subset \tilde{\gamma}(X_{i})$, entonces $rank_{G}(\tilde{\gamma})=|X_{i}|$ y así $\tilde{\gamma}\in \mathcal{J}_{i}$. 
 Es sencillo ver que $\tilde{\tau}=\tilde{\gamma}\tilde{\beta}$. Entonces $\tilde{\tau}\in \langle \mathcal{J}_{i}\rangle.$
\end{proof}

\begin{teorema} \label{cajas}
Sea $X$ un conjunto arbitrario infinito, $G$ un grupo que actúa sobre $X$, tales que $|Conj_{G}(X)|<\infty$ y sea  $H_{i}\in Stab_{G}(X)$ tal que $|\mathcal{B}_{[H_{i}]}|=\infty$. Entonces existen  $\tilde{\mu}_{i},\tilde{\nu}_{i}\in End_{G}(\mathcal{B}_{[H_{i}]})$ tales que $\tilde{\mu}_{i}$ es inyectiva de rango $G$-equivariante $|X|$ y $\tilde{\nu}_{i}$ es sobreyectiva de índice de contracción infinito $G$-equivariante igual a $|X|$ y se cumple que:  $$End_{G}(\mathcal{B}_{[H_{i}]})=\langle Aut_{G}(\mathcal{B}_{[H_{i}]}),\tilde{\mu}_{i},\tilde{\nu}_{i}\rangle.$$
\end{teorema}

\begin{proof}
Sean $\mu_{i}:X_{i}\rightarrow X_{i}$ una función inyectiva de defecto $|X_{i}|$ y $\nu_{i}:X_{i}\rightarrow X_{i}$ una función sobreyectiva de índice de contracción infinito $|X_{i}|$, entonces $\tilde{\mu}_{i}$ y $\tilde{\nu}_{i}$ satisfacen con las condiciones deseadas. Basta con mostrar que $\mathcal{J}_{i}\subset \langle Aut_{G}(\mathcal{B}_{[H_{i}]}),\tilde{\mu}_{i},\tilde{\nu}_{i}\rangle$, para que se cumpla la afirmación.\\

Sea $\tau\in \mathcal{J}_{i}$, i.e., $rank_{G}(\tau)=|X|=|X_{i}|$. Se tiene entonces la familia 
$$\Sigma=\{\tilde{C}\subseteq \mathcal{B}_{[H_{i}]}: G\tau(x)=G\tau(y),\ \forall x,y\in \tilde{C}\}$$
o equivalentemente
$$\{\tilde{C}=\tau^{-1}(Gy):\ y\in \mathcal{B}_{[H_{i}]}\}.$$
Notemos que si $x\in \tilde{C}$ se tiene que $Gx \subseteq \tilde{C}$. Sea $\Lambda$ un conjunto de índices tal que $|\Lambda|=|X_{i}|$. Podemos dividir $\Lambda=\Lambda_{1}\sqcup \Lambda_{2}$, tal que $|\Lambda|=|\Lambda_{1}|=|\Lambda_{2}|$. Entonces podemos indexar a todos los elementos de $\Sigma$ por elementos en $\Lambda_{1}$. Es decir que podemos visualizar a cada $\tilde{C}_{\lambda}$, $\lambda\in \Lambda_{1}$, como el conjunto de $G$-órbitas en $\mathcal{B}_{[H_{i}]}$ las cuales sus imágenes recaen en la misma $G$-órbita bajo $\tau$. \\

Ahora, como $k_{G}(\tilde{\nu}_{i})=|X_{i}|$. Tenemos entonces la familia
$$\Delta=\{\tilde{B}\subseteq \mathcal{B}_{[H_{i}]}:\ G\tilde{\nu}_{i}(x)=G\tilde{\nu}_{i}(y),\ \forall x,y\in \tilde{B},\ |\tilde{B}/G|=|X_{i} \},$$
y podemos indexar esta familia con el conjunto $\Lambda$. Es decir que $\tilde{B}_{\lambda}$, $\lambda \in \Lambda$, es el conjunto de $G$-órbitas cuyas imágenes caen en la misma $G$-órbita bajo $\tilde{\nu}_{i}$ y además la cantidad de $G$-órbitas en $\tilde{B}_{\lambda}$ es $|X_{i}|$. \\

Sean $B_{\lambda}=\tilde{B}_{\lambda}\cap X_{i}$ y $C_{\lambda}=\tilde{C}_{\lambda}\cap X_{i}$, los conjuntos de representantes de las órbitas dentro de $\tilde{B}_{\lambda}$ y $\tilde{C}_{\lambda}$ respectivamente.\\

Se tiene que $|B_{\lambda}|=|X_{i}|$, entonces definimos una permutación $\pi$ como sigue. Para cada $\lambda \in \Lambda_{1}$ existe una función inyectiva 
$$\pi_{\lambda}:\mu(C_{\lambda})\rightarrow B_{\lambda}.$$
Sea $$\overline{\pi}=\bigcup_{\lambda\in \Lambda_{1}}{\pi_{\lambda}},$$ entonces es claro que $\overline{\pi}$ es una inyección parcial y tenemos que 
$$|X_{i}\setminus dom(\overline{\pi})|=|X_{i}\setminus im(\mu)|=def(\mu)= |X_{i}|.$$
Notemos también que $$im(\overline{\pi})\subseteq \bigcup_{\lambda\in \Lambda_{1}}{B_{\lambda}},$$
entonces tenemos que 
$$|X_{i}|=|\bigcup_{\lambda\in \Lambda_{2}}{B_{\lambda}}|\leq |X_{i}\setminus im(\overline{\pi})|.$$
Entonces, $\overline{\pi}$ puede ser extendida a una biyección $\pi\in Sym(X_{i})$ y tomamos $\tilde{\pi}\in Aut_{G}(\mathcal{B}_{[H_{i}]})$.\\

Notemos pues que para cada $\lambda \in \Lambda$, se tiene que la imagen de cada bloque $B_{\lambda}$ bajo $\nu_{i}$, $\nu_{i}(B_{\lambda})$, es un conjunto unipuntual y podemos identificarlos con el elemento que contiene. Entoces ahora utilizamos el axioma de elección para seleccionar, para cada $\lambda\in \Lambda_{1}$, un elemento $z_{\lambda}$ en $\nu_{i}^{-1}(\tau(C_{\lambda}))$. Notemos que $\nu_{i}^{-1}(\tau(C_{\lambda}))$ es no vacío, debido a que $\nu_{i}$ es sobreyectiva. Entonces definimos una función $\overline{\sigma}:\{\nu_{i}(B_{\lambda}): \lambda\in \Lambda_{1}\}\rightarrow X$ definida por:
$$\overline{\sigma}(\nu_{i}(B_{\lambda}))=z_{\lambda}.$$
Se tiene que $$\nu_{i}(z_{\lambda})=\tau(C_{\lambda}),\ \lambda\in \Lambda_{1}.$$
Recordemos que $\Lambda=\Lambda_{1}\sqcup \Lambda_{2}$, es decir que si $\lambda\in \Lambda_{2}$, no está en $\Lambda_{1}$ y por tanto $B_{\lambda}$ no pertenece a $dom(\overline{\sigma})$ y por tanto:
$$|X_{i}\setminus dom(\overline{\sigma})|=|X_{i}|.$$
Notemos que $im(\overline{\sigma})$ contiene a lo mucho un elemento de cada uno de los bloques $B_{\lambda}$. Como $|B_{\lambda}|=|X_{i}|$, se sigue que $$|X_{i}\setminus im(\overline{\sigma})|=|X_{i}|,$$
y por tanto $\overline{\sigma}$ puede ser extendida a una biyección $\sigma\in Sym(X_{i})$. Tomamos pues $\tilde{\sigma}\in Aut_{G}(\mathcal{B}_{[H_{i}]})$.\\
Afirmamos que:
$$\tau=\tilde{\nu}_{i}\tilde{\sigma}\tilde{\nu}_{i}\tilde{\pi}\tilde{\mu}_{i}.$$
Sea $z\in \mathcal{B}_{[H_{i}]}$ un elemento arbitrario, entonces existe $g\in G$ y $x\in X_{i}$ tal que $z=g\cdot x$. Como $x\in X_{i}$, existe un único índice $\lambda\in \Lambda$ tal que $x\in C_{\lambda}$, en particular podemos suponer sin pérdida  de generalidad que $\lambda\in \Lambda_{1}$. En consecuencia se tiene que $z=g\cdot x \in \tilde{C}_{\lambda}$ y $\tilde{\mu}_{i}(z)\in \tilde{\mu}_{i}(\tilde{C}_{\lambda})$. Por construcción como $\mu_{i}(x)\in \mu_{i}(C_{\lambda})$, se tiene que $\pi\mu_{i}(x)\in B_{\lambda}$. Entonces $\tilde{\pi}\tilde{\mu}_{i}(g\cdot x)=g\cdot \pi \mu_{i}(x)\in \tilde{B}_{\lambda}$. Se tiene pues que $\tilde{\nu}_{i}\tilde{\pi}\tilde{\mu}_{i}(g\cdot x)=g\cdot \nu_{i}\pi\mu_{i}(x)=g\cdot \nu_{i}(B_{\lambda})$ es un solo punto. Como $\sigma(\nu_{i}(B_{\lambda}))=z_{\lambda}$, tenemos que $\tilde{\sigma}\tilde{\nu}_{i}\tilde{\pi}\tilde{\mu}_{i}(g\cdot x)=g\cdot z_{\lambda}$. Luego $\tilde{\nu}_{i}(g\cdot z_{\lambda})= g\cdot \nu_{i}(z_{\lambda})=g\cdot \tau(C_{\lambda})$. Pero como $\tau(C_{\lambda})=\tau(x)$, entonces se tiene que $$\tilde{\nu}_{i}\tilde{\sigma}\tilde{\nu}_{i}\tilde{\pi}\tilde{\mu}_{i}(z)=\tilde{\nu}_{i}\tilde{\sigma}\tilde{\nu}_{i}\tilde{\pi}\tilde{\mu}_{i}(g\cdot x)=g\cdot \tau(x)= \tau(g\cdot x)=\tau(z).$$

\end{proof}

\begin{corolario}
Sea $G$ un grupo finito, $H$ un subgrupo de $G$ y $X$ un conjunto sobre el cual $G$ actúa, de tal manera que $|\mathcal{B}_{[H]}|=\infty$  . Entonces se tiene que $$rank(End_{G}(\mathcal{B}_{[H]}):Aut_{G}(\mathcal{B}_{[H]}))=2.$$
\end{corolario}

\subsubsection{Generadores de $End_{G}(X)$}

Ahora buscamos acotar el rank relativo del monoide de endomorfismos completo. Para esto debemos extender estas funciones a todo $X$.\\
Debemos recordar que $Conj_{G}(X)$ debe ser finito para que el rank relativo sea finito. Así que nos restringiremos a este caso. \\

Sean $\hat{\mu_{i}}:X\rightarrow X$ y $\hat{\nu_{i}}:X \rightarrow X$, definidas como:

$$\hat{\mu_{i}}(x):= \left\{  \begin{array}{cc}  \tilde{\mu_{i}}(x) & x\in \mathcal{B}_{[H_{i}]}\\
x & \mbox{otro caso,}\end{array} \right.$$
$$\hat{\nu_{i}}(x):= \left\{  \begin{array}{cc}  \tilde{\nu_{i}}(x) & x\in \mathcal{B}_{[H_{i}]}\\
x & \mbox{otro caso.}\end{array} \right.$$

\begin{teorema}
Dado un grupo $G$ y un conjunto $X$, tal que $|Conj_{G}(X)|< \infty$. Sea el conjunto $V:=\{[x_{i}\mapsto y_{i,j}^{n}],[x_{i}\mapsto x'_{i}], \hat{\mu_{i}}, \hat{\nu_{i}},\  0< i,j \leq r\}$, entonces se satisface que:
$$End_{G}(X)=\langle Aut_{G}(X) \cup V \rangle. $$
\end{teorema}

\begin{proof}
Basta probar que para cualquier endomorfismo $\tau$ se cumple que $\tau_{r}\in \langle Aut_{G}(X) \cup V \rangle $, para todo $r$. Del mismo modo que en el Teorema \ref{chido3}. Dado un elemento $z\in X$ y un endomorfismo $\tau\in End_{G}(X)$, denotamos a la imagen de $z$ por $\tau(z)=z^{j}_{m}$, donde $j$ denota el índice de la clase de conjugación a la cual $\tau(z)$ pertenece, $\mathcal{B}_{[H_{j}]}$ y $m$ su $G$-órbita.\\

Para cada $\tau_{i}$ descomponemos  $\mathcal{B}_{[H_{i}]}=T_{i}' \cup T''_{i}$ donde:
$$T'_{i}=\bigcup \left\{ Gx:\ Gx \subseteq \mathcal{B}_{[H_{i}]},\ \tau_{i}(Gx)\in \mathcal{B}_{[H_{j}]},\ i< j \right \},$$

$$T''_{i}=\bigcup \left\{ Gx:\ Gx \subseteq \mathcal{B}_{[H_{i}]},\ \tau_{i}(Gx)\in \mathcal{B}_{[H_{i}]}\right \}.$$
Definimos un par de endomorfismos  $\tau'_{i},\tau''_{i}\in End_{G}(X)$ como:

$$\tau'_{i}(z)= \left\{ \begin{array}{cc}\tau_{i}(z) & \ \mbox{si }\ z\in T'_{i}\\
z & \mbox{ otro caso} \end{array} \right.$$

$$\tau''_{i}(z)= \left\{ \begin{array}{cc}\tau_{i}(z) & \ \mbox{si }\ z\in T''_{i}\\
z & \mbox{ otro caso. } \end{array} \right.$$
Entonces se satisface que $\tau_{i}=\tau''_{i}  \tau'_{i}$.\\

Sea $\{n_{k}\}$ la lista de índices tales que $Gw^{i}_{n_{k}}\subseteq T'_{i}$. A causa de la $G$-equivarianza podemos asumir sin pérdida  de generalidad que $z=w_{n}^{i}$, para cada $n\in \{n_{k}\}$, entonces podemos verificar que:
$$\tau'_{i}= \prod_{n\in\{n_{k}\}}{  (x^{j}_{t}\ z^{j}_{m})(i\mapsto j_{t})(1\ n)^{i}(x^{j}_{t}\ z^{j}_{m})} $$
para algún $t=1,2,..., \beta_{ij}$. Debemos mencionar que $t$ y $j$ dependen de $n$. Entonces $\tau'_{i}\in \langle V \cup Aut_{G}(X) \rangle$.\\

Para probar que  $\tau''_{i}$ pertenece a $\langle V \cup Aut_{G}(X) \rangle$, usamos un procedimiento análogo a la Proposición \ref{cajas}. Si construimos $\pi$ y $\sigma$ en $Sym(X_{i})$ de la misma manera que en la Proposición \ref{cajas}, pero para $\tau''_{i}$, entonces se satisface que:
$$\tau''_{i}=\hat{\nu}_{i}\hat{\sigma}\hat{\nu}_{i}\hat{\pi}\hat{\mu}_{i},$$
y por lo tanto $\tau''_{i}\in \langle V \cup Aut_{G}(X) \rangle$.
\end{proof}

En este punto solo hemos podido encontrar cotas para el rank relativo del monoide de endomorfismos:

\begin{small}
$$\sum_{i=1}^{r}{|U(H_{i})|}-|\kappa_{G}(X)| \leq rank(End_{G}(X):Aut_{G}(X))\leq \sum_{i=1}^{r}{|U(H_{i})|}-|\kappa_{G}(X)|+ |\eta_{G}(X)|.$$
\end{small}
Donde $\eta_{G}(X)=\{i:\ \alpha_{i}= \infty\}$. Los objetivos a corto plazo son probar la igualdad. La conjetura dice que la igualdad se cumple del lado derecho. 

\newpage

\section{Aplicaciones en el monoide de autómatas celulares}

En esta sección definiremos el espacio de autómatas celulares, espacio sobre el cual basamos nuestro trabajo. Analizaremos los aspectos básicos de la estructura como $G$-conjunto y aplicaremos los resultados obtenidos a estos conceptos.

\subsection{El espacio de configuraciones $A^{G}$}

Sea $G$ un grupo y $A$ un conjunto al que nos referiremos como alfabeto. Es de particular interés el espacio de configuraciones, esto es, el conjunto de las funciones $x: G \rightarrow A$. Entonces definimos el espacio de configuraciones como $$A^{G}:=\{\alpha : G \rightarrow A\}.$$
Toda función en $A^{G}$ puede representarse como una túpla de elementos de $A$ como se muestra a continuación,

$$x \in A^{G}, \hspace{0.2in} x=(...x(g_{0}),x(g_{1}),x(g_{2}),...), \hspace{0.2in} g_{i}\in G,$$
notando que $x(g_{i})$ son elementos del alfabeto.\\

En particular podemos ver un elemento de la manera $x=(...a_{0},a_{1},a_{2},...)$ donde es fácil ver que $x(g_{i})=a_{i}$. Esta representación funciona cuando $G$ es un conjunto contable, no lo generalizaremos, solo utilizamos esta notación para generar un mejor entendimiento del concepto de autómata celular.

\subsection{La acción natural de $G$ en $A^{G}$}

La llamada acción natural de $G$ en $A^{G}$ está definida de la siguiente manera:
$$\begin{array}{l}
\alpha \in A^{G},\ g\in G \\
g \cdot \alpha \in A^{G},\  g \cdot \alpha: G \rightarrow A \\
(g \cdot \alpha )(h) = \alpha (g^{-1}h),\ \forall h\in G.
\end{array} $$
Es necesario verificar que la acción natural es precisamente una acción de $G$ en $A^{G}$. Para esto sea $e\in G$ la identidad:\\
\begin{eqnarray*} 
(e\cdot \alpha)(h) &=& \alpha (e^{-1} h) \\
&=&\alpha (eh) \\
&=& \alpha (h) \hspace{0.1in} \forall h \in G,
\end{eqnarray*}
esto implica que $e\cdot \alpha = \alpha$. \\
Luego para $g_{1}, g_{2} \in G$
\begin{eqnarray*} 
(g_{1}g_{2})\cdot \alpha \rightarrow ((g_{1}g_{2})\cdot \alpha)(h) &=& \alpha ((g_{1}g_{2})^{-1} h) \\
&=&\alpha ((g_{2}^{-1}g_{1}^{-1})h) \\
&=& \alpha (g_{2}^{-1}(g_{1}^{-1}h)) \\
&=& (g_{2}\cdot \alpha)(g_{1}^{-1}h) \\
&=& (g_{1}\cdot ( g_{2} \cdot\alpha)) (h) \hspace{0.1in} \forall h \in G,
\end{eqnarray*}
lo cual implica que $(g_{1}g_{2})\cdot \alpha = g_{1}\cdot(g_{2}\cdot \alpha)$.

\subsection{Autómatas celulares}

\begin{definicion}{(Autómata Celular)}\label{automata}

Dado un grupo $G$ y un alfabeto $A$, un autómata celular es una función $\tau : A^{G} \rightarrow A^{G}$ si existe un subconjunto finito $S\subseteq G$ y una función $\mu : A^{S} \rightarrow A$ tal que $$(\tau(\alpha))(g)=\mu \left(\left(\alpha \circ R_{g}\right)|_{S}\right).$$
Donde $R_{g}:G \rightarrow G$ está dada por $R_{g}(h)=gh$.\\
\end{definicion}

La función $\mu$ es llamada regla local y el conjunto $S$ es llamado conjunto memoria.\\

\begin{lema}\label{intersecciondememorias}
Sea $\tau:A^{G} \rightarrow A^{G}$ un autómata celular. Sean $S_{1}$ y $S_{2}$ dos conjuntos memoria para $\tau$. Entonces $S_{1}\cup S_{2}$ es también un conjunto memoria para $\tau$.
\end{lema}

\begin{proposicion}\label{memoriaminimal}
Sea $\tau:A^{G} \rightarrow A^{G}$ un autómata celular. Entonces existe un único conjunto memoria $S_{0}\subset G$ para $\tau$ de cardinalidad mínima. De hecho, si $S$ es un subconjunto finito de $G$, entonces $S$ es un conjunto memoria para $\tau$ si y solo si $S_{0} \subset S$.
\end{proposicion}

Dicho conjunto $S_{0}$ es llamado \emph{conjunto memoria minimal}.\\
Las demostraciones del Lema \ref{intersecciondememorias} y la Proposición \ref{memoriaminimal} pueden ser encontradas en \cite[Lema \ 1.5.1, Prop. 1.5.2]{Cecc} .\\

Definimos entonces el conjunto de autómatas celulares de $G$ en $A$:

$$\text{CA}(G;A):=\{ \tau: A^{G}\rightarrow A^{G}|\  \tau\mbox{ es autómata celular}\}.$$

\subsubsection{Propiedades de los autómatas celulares}
Mencionamos en este apartado las propiedades básicas que los autómatas celulares tienen.

\begin{lema}\label{invariante}
 Todo autómata celular es $G$-equivariante.
\end{lema}

\begin{proof}
Sea $S$ el conjunto memoria de $\tau$ y sea $\mu:A^{S}\rightarrow S$ su regla local. Para cada $g,h \in G$ y para cada $x\in A^{G}$ se tiene que : \begin{eqnarray*}
[g\cdot \tau(\alpha)](h)&=&\tau(\alpha)(g^{-1}h)\\
&=&\mu((\alpha \circ R_{g^{-1}h})|_{S})\\
&=&\mu(((g^{-1}h)^{-1}\cdot \alpha))|_{S})\\
&=&\mu((h^{-1}g)\cdot \alpha)|_{S})\\
&=&\mu((h^{-1}\cdot(g\cdot \alpha))|_{S})\\
&=&\mu(((g \cdot \alpha)\circ R_{h})|_{S})\\
&=&\tau(g\cdot \alpha)(h), \hspace{0.1in} \forall h\in G.
\end{eqnarray*}
Entonces concluimos que $g\cdot \tau(\alpha)=\tau(g\cdot \alpha)$, es decir que es $G$-equivariante.\\
\end{proof}

Esta proposición nos genera una conclusión importante: 
$$\text{CA}(G;A) \leq End_{G}(A^{G}). $$
Es decir que todos los autómatas celulares son endomorfismos del $G$-conjunto $A^{G}$.

\begin{lema}\label{continuo}
Sea $\tau$ un autómata celular de $A$ en $G$. Entonces $\tau$ es una función continua.\\
\end{lema}

\begin{proof}
Sea $x\in A^{G}$ y dado un conjunto $\Omega \subset G$, definimos el conjunto $$V(x,\Omega):=\{y\in A^{G} : x|_{\Omega}=y|_{\Omega}\}.$$ Se puede mostrar que $V(x,\Omega)$ es un entorno para $x$. Sea $\tau$ un autómata celular, $S$ su conjunto memoria y sea $W$ un entorno de $x$. Entonces podemos encontrar un subconjunto finito $\Omega \subset G$ tal que $$V(\tau(x),\Omega)\subset W.$$
Consideremos el conjunto finito $\Omega S:=\{gs: g\in \Omega, s\in S\}$. Si $y\in A^{G}$ coincide con $x$ en $\Omega S$, entonces $\tau(y)$ debe coincidir con $\tau(x)$ en $\Omega$. Así podemos decir que $$\tau(V(x,\Omega S))\subset V(\tau(x),\Omega),$$
esto prueba que la función $\tau$ es continua.
\end{proof}

Debemos notar que $\text{CA}(G;A)$ también actúa naturalmente sobre $A^{G}$, con la acción de evaluación, i.e., 
$$\tau\cdot x= \tau(x),\ \tau \in \text{CA}(G;A),\ x\in A^{G}.$$

\subsection{Autómatas celulares de alfabeto finito}
En esta sección trabajaremos algunas propiedades importantes de los autómatas celulares que fueron punto de inicio para nuestro trabajo; los alfabetos finitos.\\

\subsubsection{Teorema de Curtis-Hedlund}

\begin{teorema}{(Teorema de Curtis-Hedlund)}\label{CurtisHed}

Sea $A$ un conjunto finito y dado $G$ un grupo las siguientes afirmaciones son equivalentes:\\

a) $\tau$ es un autómata celular.\\

b) $\tau$ es $G$-equivariante. $$\tau(g \cdot \alpha )=g \cdot  \tau(\alpha), \hspace{0.2in} \forall \alpha \in A^{G}, \forall g\in G,$$
y además es continuo en la topología prodiscreta de $A^{G}$.\\
\end{teorema}

\begin{proof}
Ambos Lemas \ref{invariante} y \ref{continuo} prueban la primer implicación del teorema.\\

Conversamente, supongamos que $\tau$ es continua y $G$-equivariante. La función $\phi:A^{G}\rightarrow A$ definida por $\phi(x)=\tau(x)(1_{G})$ es continua. Ahora, para cada configuración $x\in A^{G}$ podemos encontrar un subconjunto finito $\Omega_{x} \subset G$ tal que $y\in \Omega_{x}$ si coincide con $x$ en $\Omega_{x}$, esto significa que $y\in V(x,\Omega_{x})$, entonces $\tau(x)(1_{G})=\tau(y)(1_{G})$. Recordemos que los conjuntos $V(x,\Omega_{x})$ son conjuntos abiertos y forman una cubierta abierta para $A^{G}$. Además, al ser $G$ finito, tenemos que $A^{G}$ es compacto. Esto significa que podemos cubrir $A^{G}$ con un número finito de conjuntos $V(x,\Omega_{x})$. Sea $F\subset G$ el conjunto de los $x$ cuyos $V(x,\Omega_{x})$ cubren $A^{G}$. Es importante mencionar que $F$ es un conjunto finito. Sea $S=\bigcup_{x\in F}{V(x,\Omega_{x})}$ y supongamos que $y,z \in A^{G}$ son dos configuraciones que coinciden en $S$. Seleccionemos $x_{0}$ de tal manera que $y\in V(x_{0},\Omega_{x_{0}})$. Esto significa que $y|_{\Omega_{x_{0}}}=x_{0}|_{\Omega_{x_{0}}}$. Como $S\supset \Omega_{x_{0}}$, se tiene entonces que $y|_{\Omega_{x_{0}}}=z|_{\Omega_{x_{0}}}$ y por consiguiente tenemos que $\tau(y)(1_{G})=\tau(x_{0})(1_{G})=\tau(z)(1_{G})$. Entonces existe una función $\mu:A^{S}\rightarrow A$ tal que $\tau(x)(1_{G})=\mu(x|_{S})$ $\forall x\in A^{G}$. Como $\tau$ es $G$-equivariante, por una propiedad básica de los autómatas celulares, vease \cite[Prop. 1.4.6]{Cec}, se tiene que $\tau$ es un autómata celular con conjunto memoria $S$ y regla local $\mu=\phi$.
\end{proof}

\begin{ejemplo}
Sea $G$ un grupo arbitrario infinito y tomemos $A=G$ como el alfabeto. Consideremos la función $\tau:A^{G} \rightarrow A^{G}$ definida por $$\tau(x)(g)=x(g\cdot x(g)) \ \forall x\in A^{G}, \  g\in G.$$ Mostraremos que $\tau$ no es un autómata celular.\\

Para evitar confusiones, recordemos que la acción entre un elemento de $G$ y una configuración $x$ la denotamos $g \cdot x$, mientras que el producto de dos elementos $g,h \in G$  lo denotamos directamente $gh$.

Dados $x\in A^{G}$ y $g,h\in G$ tenemos que $$\begin{array}{cl}
g\cdot( \tau(x))(h)&= \tau(x)(g^{-1}h)\\
&=x(g^{-1}hx(g^{-1}h))\\
&=x(g^{-1}h(g\cdot x)(h))\\
&=(g\cdot x)(h(g\cdot x)(h))\\
&=\tau(g\cdot x)(h).
\end{array}$$

Esto muestra que $g\cdot( \tau(x))=\tau(g\cdot x)$ $\forall x\in A^{G}$ y $\forall g\in G$. Y así $\tau$ es $G$-equivariante. Ahora dada una configuración $x\in A^{G}$ y sea $K\subset G$ un subconjunto finito. Sea $F=K\cup\{kx(k): k\in K\}$, entonces si $y\in V(x,F)$ tenemos que $\forall k\in K$ se cumple que: $$\tau(x)(k)=x(kx(k))=y(kx(k))=y(ky(k))=\tau(y)(k).$$ Esto significa que $\tau(y)\in V(\tau(x),K)$, o dicho de otra forma $\tau(V(x,F))\subset V(\tau(x),K)$. Y por lo tanto la función $\tau$ es continua.\\

Ahora, sea $g_{0}\in G$ un elemento distinto de la identidad, $g\neq 1_{G}$. Definimos para cualquier elemento $g\in G$ las configuraciones $x_{g},y_{g}\in A^{G}$ de la manera siguiente:
$$\begin{array}{ccc}
x_{g}(h)=\left\{\begin{array}{ll}
g, & h=1_{G}\\
g_{0}, & h=g\\
1_{G}, & \mbox{de otro modo}
\end{array}\right.  & \mbox{ y } & y_{g}(h)=\left\{\begin{array}{ll}
g, & h=1_{G}\\
1_{G}, & \mbox{de otro modo.}
\end{array}\right.
\end{array}$$

Notemos que $x_{g}|_{G-\{g\}}=y_{g}|_{G-\{g\}}$. Entonces para cualquier conjunto finito $F \subset G$ podemos elegir un elemento $g\in G-F$, por el hecho de que $G$ es infinito. Así tenemos que $x_{g}|_{F}=y_{g}|_{F}$.
Pero también tenemos las siguientes igualdades $$\tau(x_{g})(1_{G})=x_{g}(x_{g}(1_{G}))=x_{g}(g)=g_{0} $$ $$ \tau(y_{g})(1_{G})=y_{g}(y_{g}(1_{G}))=y_{g}(g)=1_{G} $$
y por consiguiente tenemos que $\tau(x_{g})(1_{G}) \neq \tau(y_{g})(1_{G})$.\\
De la Definición \ref{automata}, como el mapeo restricción $A^{G} \rightarrow A^{S}$ dado por $x \mapsto x|_{S}$ es sobreyectivo, tenemos que si $S$ es un conjunto memoria para el autómata celular $\tau$, entonces habrá una única función $\mu: A^{S} \rightarrow A$ que cumpla que $\tau(x)(g)=\mu ((x \circ R_{g})|_{S})$. Entonces decimos que esta función $\mu$ está completamente definida por el conjunto $S$. Pero de la manera en la que construimos las configuraciones $x_{g}$ y $y_{g}$, no existe un conjunto finito $F\subset G$ tal que para cualquier configuración $x\in A^{G}$ el valor de$\tau(x)(1_{G})$ dependa solo de los valores de $x|_{F}$, es decir del conjunto $F$. O dicho de otra manera, $\mu$ y por consecuencia $\tau$, no están bien definidos. \qed 
\end{ejemplo}

\subsubsection{Autómatas celulares invertibles}

Las unidades de cualquier monoide son un grupo de particular interés en el estudio de dicho monoide, es por eso que en esta sección presentamos un poco del grupo de unidades del monoide de autómatas celulares, el grupo de autómatas celulares invertibles.

\begin{definicion}{(Autómata celular invertible)}

Decimos que un autómata celular $\tau$ es un autómata celular invertible, si es una función biyectiva y su función inversa $\tau^{-1}:A^{G}\rightarrow A^{G}$ es también un autómata celular.\\
\end{definicion}

\begin{teorema}{(Autómatas celulares invertibles)}\label{invertibles}

Sea $G$ un grupo y $A$ un conjunto finito. Entonces todo autómata celular de $A$ en $G$ biyectivo, es un autómata celular invertible.\\
\end{teorema}

\begin{proof}
Sea $\tau$ un autómata celular biyectivo. Entonces existe $\tau^{-1}:A^{G}\rightarrow A^{G}$ tal que $\tau^{-1} \circ \tau=Id$.\\

$$\begin{array}{rl}
Id(g\cdot \alpha)&=g \cdot Id(\alpha)\\
(\tau^{-1}\circ \tau)(g\cdot \alpha)&=g\cdot (\tau^{-1}\circ \tau)(\alpha)\\
\tau^{-1}(\tau(g\cdot \alpha))&=g\cdot (\tau^{-1}(\tau(\alpha)))\\
\tau^{-1}(g\cdot\tau(\alpha))&=g\cdot \tau^{-1}(\beta)\\
\tau^{-1}(g\cdot \beta)&=g\cdot \tau^{-1}(\beta).
\end{array}$$
Por lo tanto $\tau^{-1}$ es $G$-equivariante. \\
Además al ser $\tau$ una función continua con la topología prodiscreta en $A^{G}$, por las propiedades de $A^{G}$ como espacio topológico, tales como ser Haussdorff, su inversa también es una función continua. Entonces, por el Teorema de Curtis-Hedlund, $\tau^{-1}$ es un autómata celular. Por lo tanto $\tau$ es una autómata celular invertible.\\

\end{proof}

Definimos pues el conjunto de autómatas celulares invertibles como $Aut_{G}(X)$. Además si $A$ es un conjunto finito, por el Teorema \ref{invertibles}, se tiene la igualdad:
$$\text{ICA}(G;A):=\text{CA}(G;A) \cap Sym(A^{G}).$$

\subsection{Particiones de $A^{G}$}

Ejemplificamos en esta sección las particiones que se presentan en el espacio de configuraciones a partir de la acción del grupo $G$. 

\begin{ejemplo}
Sea $G=\mathbb{Z}_{4}$ y $A=\{0,1\}$. Dada la configuración $x=(0,0,0,1)\in A^{G}$ tenemos que:
$$Gx=\left\{ \begin{array}{c} (0,0,0,1)\\ (0,0,1,0)\\ (0,1,0,0)\\ (1,0,0,0) \end{array}\right\}.$$

Recordemos que visualizamos una configuración $x:G\rightarrow A$ como
$$x: \left( \begin{array}{cccc} g_{0}&g_{1}&\dots&g_{n}\\
\downarrow & \downarrow & \dots & \downarrow \\ a_{1}&a_{2}&\dots&a_{n}\end{array} \right),$$
en específico tenemos:
$$x: \left( \begin{array}{cccc} 0&1&2&3 \\
\downarrow & \downarrow & \downarrow & \downarrow \\ 0&0&0&1\end{array} \right)$$
y al accionar con un elemento en $G$
$$h\cdot x \longrightarrow h \cdot x(g)= x(h^{-1}g). $$
Por ejemplo $3\in G$, tenemos que 
$$3\cdot x \longrightarrow 3 \cdot x(g)= x(3^{-1}g).$$
Sabemos que en $\mathbb{Z}_{4}$ la operación en la suma y que $3^{-1}=1$ en esta notación, entonces:
$$3\cdot x \longrightarrow 3 \cdot x(g)= x(1+g). $$
Así tenemos que:

$$3\cdot x: \left( \begin{array}{cccc} 0&1&2&3 \\
\downarrow & \downarrow & \downarrow & \downarrow \\ 
x(1+0)&x(1+1)&x(1+2)&x(1+3) \\
\downarrow & \downarrow & \downarrow & \downarrow \\ 
x(1)&x(2)&x(3)&x(0)\\
\downarrow & \downarrow & \downarrow & \downarrow \\ 
0&0&1&0\end{array} \right),$$
entonces visualizamos 
$$3\cdot x: \left( \begin{array}{cccc} 0&1&2&3 \\
\downarrow & \downarrow & \downarrow & \downarrow \\ 0&0&1&0\end{array} \right).$$
Siguiendo estos procesos para accionar con los demás elementos de $G$ y compactando notación, tenemos que:
$$\begin{array}{rl}
x&=(0,0,0,1)\\
3\cdot x&=(0,0,1,0)\\
2\cdot x&=(0,1,0,0)\\
1\cdot x&=(1,0,0,0).
\end{array}
$$

\end{ejemplo}

Ahora, dado un subgrupo $H$ de $G$,  ejemplificamos las particiones en cajas, los conjuntos:
$$\mathcal{B}_{H}:=\{x\in X:\ H=G_{x}\}.$$

\begin{ejemplo}
Sea $G=\mathbb{Z}_{4}$ y $A=\{0,1\}$. Identificamos que los únicos subgrupos de $\mathbb{Z}_{4}$ son:
$$\begin{array}{ccccc}

 \{0\}&\subseteq &\{0,2\} &\subseteq &\mathbb{Z}_{4}  \\
 H_{0}&\subseteq &H_{1} &\subseteq & H_{2}.

\end{array}$$
Notemos que para la configuración $x=(1,0,0,0)\in A^{G}$ el único elemento $g\in \mathbb{Z}_{4}$ que cumple que $g\cdot x=x$ es la identidad, i.e., $g=0$. O dicho de otra manera $G_{x}=\{0\}$. Esto significa que $x\in \mathcal{B}_{H_{0}}$. Por otro lado, para la configuración $y=(1,0,1,0)\in A^{G}$ se tiene que no solo $g=0$ cumple con $g\cdot y=y$, sino también $g=2$. Entonces se tiene que $G_{y}=\{0,2\}=H_{1}$. Es decir, $y\in \mathcal{B}_{H_{1}}$.\\

Se tiene que para las configuraciones constantes, por ejemplo $z=(1,1,1,1)\in A^{G}$, se cumple que cualquier elemento $g\in \mathbb{Z}_{4}$ cumple que $g\cdot z=z$, por lo que $z\in \mathcal{B}_{H_{2}}$.
Entonces tenemos el siguiente diagrama:\\

$$\begin{array}{|l|}

\hline

\begin{array}{lccc}
\mathcal{B}_{H_{2}} &
\begin{array}{|c}
(1,0,0,0)\\ (0,1,0,0)\\ (0,0,1,0)\\ (0,0,0,1)
\end{array} &
\begin{array}{c}
(1,1,0,0)\\ (0,1,1,0)\\ (0,0,1,1)\\ (1,0,0,1)
\end{array} &
\begin{array}{c|}
(1,1,1,0)\\ (1,1,0,1)\\(1,0,1,1)\\ (0,1,1,1)
\end{array} 
\end{array}   \\
\hline

\begin{array}{lcccccccccc}
\mathcal{B}_{H_{1}} & & & & & 
\begin{array}{|c|}
(1,0,1,0)\\ (0,1,0,1)
\end{array}   & & & & & \end{array} \\

\hline

\begin{array}{lccc|ccc|cc}
\mathcal{B}_{H_{0}} & & & & (1,1,1,1)& &(0,0,0,0) & &
\end{array} \\
\hline
\end{array}.$$

\end{ejemplo}

Uno de nuestros objetivos es visualizar otras particiones de $\mathcal{B}_{H}$ a partir también de sus $ICA$-órbitas. Entonces generaremos un par de ejemplos adicionales como el anterior.\\

Para los siguientes ejemplos utilizaremos grupos de 6 elementos, $\mathbb{Z}_{6}$ y $S_{3}$, utilizando como alfabeto nuevamente $A=\{0,1\}$.\\

A continuación presentamos las tablas de Cayley de los grupos $\mathbb{Z}_{6}$ y $S_{3}$.

$$
\begin{array}{cc}

\begin{array}{c||c|c|c|c|c|c|}   
\mathbb{Z}_{6} & 0&1&2&3&4&5 \\
\hline
\hline
0& 0&1&2&3&4&5 \\
\hline
1& 1&2&3&4&5&0 \\
\hline
2& 2&3&4&5&0&1 \\
\hline
3& 3&4&5&0&1&2 \\
\hline
4& 4&5&0&1&2&3 \\
\hline
5& 5&0&1&2&3&4 \\
\hline

\end{array}
&

\begin{array}{c||c|c|c|c|c|c|}   
S_{3} & e&a&b&c&f&g \\
\hline
\hline
e& e&a&b&c&f&g \\
\hline
a& a&e&f&g&b&c \\
\hline
b& b&g&e&f&c&a\\
\hline
c& c&f&g&e&a&b \\
\hline
f& f&c&a&b&g&e \\
\hline
g& g&b&c&a&e&f \\
\hline

\end{array}

\end{array}
$$
Debemos notar la relación entre inversos, pues es importante para nuestra acción:
$$\begin{array}{c|c} 
\mathbb{Z}_{6} & S_{3}\\
\hline
\hline
0^{-1}=0&e^{-1}=e\\
5^{-1}=1&a^{-1}=a \\
4^{-1}=2&b^{-1}=b\\
3^{-1}=3&c^{-1}=c\\
2^{-1}=4&g^{-1}=f\\
1^{-1}=5&f^{-1}=g

\end{array}$$
Entonces para la configuración $x=(0,0,0,1,1,0)$, podemos visualizarla como

$$x=  \left( \begin{array}{cccccc} 
0&1&2&3&4&5\\
\downarrow &\downarrow &\downarrow &\downarrow &\downarrow &\downarrow   \\
0&0&0&1&1&0
\end{array}\right)=
 \left( \begin{array}{cccccc} 
e&a&b&c&f&g\\
\downarrow &\downarrow &\downarrow &\downarrow &\downarrow &\downarrow   \\
0&0&0&1&1&0
\end{array}\right),$$
según el grupo con el que se trabaje.\\

Entonces, al accionar por el primer elemento de cada grupo, $1$ y $a$ respectivamente, tenemos que:

$$1\cdot x= \left( \begin{array}{cccccc} 
0&1&2&3&4&5\\
\downarrow &\downarrow &\downarrow &\downarrow &\downarrow &\downarrow   \\
x(1^{-1}0)&x(1^{-1}1)&x(1^{-1}2)&x(1^{-1}3)&x(1^{-1}4)&x(1^{-1}5)\\
\downarrow &\downarrow &\downarrow &\downarrow &\downarrow &\downarrow   \\
x(5+0)&x(5+1)&x(5+2)&x(5+3)&x(5+4)&x(5+5)\\
\downarrow &\downarrow &\downarrow &\downarrow &\downarrow &\downarrow   \\
x(5)&x(0)&x(1)&x(2)&x(3)&x(4)\\
\downarrow &\downarrow &\downarrow &\downarrow &\downarrow &\downarrow   \\
0&0&0&0&1&1
\end{array}\right)$$

$$a\cdot x= \left( \begin{array}{cccccc} 
e&a&b&c&f&g\\
\downarrow &\downarrow &\downarrow &\downarrow &\downarrow &\downarrow   \\
x(a^{-1}e)&x(a^{-1}a)&x(a^{-1}b)&x(a^{-1}c)&x(a^{-1}f)&x(a^{-1}g)\\
\downarrow &\downarrow &\downarrow &\downarrow &\downarrow &\downarrow   \\
x(ae)&x(aa)&x(ab)&x(ac)&x(af)&x(ag)\\
\downarrow &\downarrow &\downarrow &\downarrow &\downarrow &\downarrow   \\
x(a)&x(e)&x(f)&x(g)&x(b)&x(c)\\
\downarrow &\downarrow &\downarrow &\downarrow &\downarrow &\downarrow   \\
0&0&1&0&0&1
\end{array}\right).$$
Entonces se satisface que
$$\begin{array}{c} 
x=(0,0,0,1,1,0)\\

\begin{array}{cc}
1\cdot x= (0,0,0,0,1,1)
&
a\cdot x=(0,0,1,0,0,1)
\end{array}

\end{array}$$
y tenemos que 

$$\begin{array}{cc}
Gx= \left\{ \begin{array}{c}

0\cdot x=(0,0,0,1,1,0)\\
1\cdot x=(0,0,0,0,1,1)\\
2\cdot x=(1,0,0,0,0,1)\\
3\cdot x=(1,1,0,0,0,0)\\
4\cdot x=(0,1,1,0,0,0)\\
5\cdot x=(0,0,1,1,0,0)\\

\end{array}
\right\}
&
Gx= \left\{ \begin{array}{c}

e\cdot x=(0,0,0,1,1,0)\\
a\cdot x=(0,0,1,0,0,1)\\
b\cdot x=(0,0,0,1,1,0)\\
c\cdot x=(1,1,0,0,0,0)\\
f\cdot x=(0,0,1,0,0,1)\\
g\cdot x=(1,1,0,0,0,0)\\

\end{array}
\right\}.
\end{array}
$$
Entonces concluimos que las órbitas de una configuración $x$ dependen de la acción, o del grupo $G$ con el cual se actúa. Algo que debemos notar es que los conjuntos $\mathcal{B}_{H}$ también se ven modificados a partir del grupo que actúe.\\

Para compactar el espacio, utilizaremos una notación especial, solo en esta parte, para referirnos a las configuraciones. Dada una configuración, por ejemplo $x=(0,1,1,0,1,0)$, visualizaremos la configuración como el número decimal el cual la combinación $011010$ representa en código binario, es decir:
$$x=(0,1,1,0,1,0)=11010=26.$$
En general tenemos 
$$\begin{array}{c}  
(0,0,0,0,0,0)=0\\
(0,0,0,0,0,1)=1\\
(0,0,0,0,1,0)=2\\
\vdots \\
(1,1,1,1,1,0)=62\\
(1,1,1,1,1,1)=63.

\end{array}$$

\begin{ejemplo}
Sea $G=\mathbb{Z}_{6}$ y $G'=S_{3}$, notaremos que los $\mathcal{B}_{H}$ no son iguales.\\
Es importante ver que los subgrupos de $\mathbb{Z}_{6}$ únicamente son: 
$$\begin{array}{c} 
\{0,3\}\cong \mathbb{Z}_{2}\\
\{0,2,4\}\cong \mathbb{Z}_{3} \end{array}$$
a los cuales denotaremos $H_{2}$ y $H_{3}$ respectivamente. Mientras que en $S_{3}$ tenemos que sus subgrupos son:
$$\begin{array}{c}   
\{e,a\} \cong \{e,b\} \cong \{e,c\} \cong \mathbb{Z}_{2}\\
\{e,f,g\} \cong \mathbb{Z}_{3},
\end{array}$$  
a los cuales denotaremos $H_{a}$, $H_{b}$, $H_{c}$ y $H_{3}$ respectivamente. Para identificar entre los dos conjuntos $\mathcal{B}_{H_{3}}$, primaremos a los conjuntos correspondientes cuando actúa $S_{3}$. Entonces tenemos que:

\begin{scriptsize}
$$
\begin{array}{ll}
\begin{array}{|lc|}

\hline

\mathcal{B}_{\mathbb{Z}_{6}} 
&
\begin{array}{|cc|}
\hline
0 & 63 \\
\hline

\end{array} \\

\hline

\mathcal{B}_{H_{3}} 
&
\begin{array}{|c|}
\hline
21  \\
42 \\ 
\hline

\end{array} \\

\hline

\mathcal{B}_{H_{2}} 
&

\begin{array}{|cc|}
\hline
9 & 27 \\
18 & 45 \\
36 & 54 \\
\hline

\end{array} \\
\hline

\mathcal{B}_{H_{0}} 
&
\begin{array}{|ccccccccc|}
\hline
1&3&5&7&11&13&15&23&31\\
2&6&10&14&22&26&30&46&62\\
4&12&20&28&44&52&60&29&61\\
8&24&40&56&25&41&57&58&59\\
16&48&17&49&50&19&51&53&55\\
32&33&34&35&37&38&39&43&47\\
\hline

\end{array} \\

\hline
\end{array}

&

\begin{array}{|lc|}

\hline

\mathcal{B}'_{S_{3}} 
&
\begin{array}{|c|c|}
\hline
0 & 63 \\
\hline

\end{array} \\
 
\hline

\mathcal{B}'_{H_{3}} 
&
\begin{array}{|c|}
\hline
28\\
35\\

\hline

\end{array} \\

\hline

\mathcal{B}'_{H_{a}} 
&
\begin{array}{|cccccc|}
\hline
5&48&10&15&58&53\\
\hline

\end{array} \\

\hline
\mathcal{B}'_{H_{b}} 
&
\begin{array}{|cccccc|}
\hline
40&6&17&57&23&46\\
\hline

\end{array} \\

\hline
\mathcal{B}'_{H_{c}} 
&
\begin{array}{|cccccc|}
\hline
18&9&36&54&45&27\\
\hline

\end{array} \\

\hline

\mathcal{B}'_{H_{0}} 
&
\begin{array}{|ccccccc|}
\hline
1&3&7&11&21&29&31\\
4&12&13&14&37&39&47\\
16&20&22&15&19&51&55\\
8&24&56&25&42&43&59\\
2&34&50&38&26&30&62\\
32&33&41&49&44&60&61\\

\hline

\end{array} \\

\hline
\end{array}

\end{array}
$$
\end{scriptsize}
\end{ejemplo}

A partir de la relación de equivalencia que definimos sobre los subgrupos de un grupo $G$ se definieron también los siguientes conjuntos a los cuales denominamos caja
s:

$$\mathcal{B}_{[H]}:=\{x\in A^{G}:\ [G_{x}]=[H]\}.$$

\begin{lema}
Sea $G$ un grupo y $H$ un subgrupo de $G$. Las $G$-órbitas de elementos en $\mathcal{B}_{[H]}$ forman una partición (uniforme) para $\mathcal{B}_{[H]}$. 
\end{lema}

Esto es consecuencia de la acción y la $G$-equivarianza.

\begin{ejemplo}
Sea nuevamente $A=\{0,1\}$ y tomamos nuevamente los grupos $\mathbb{Z}_{6}$ y $S_{3}$. Recordemos que en $\mathbb{Z}_{6}$ tenemos que las clases de conjugación de subgrupos son:
$$[H_{2}]=\left\{\{0,3\}\right\}  $$
$$[H_{3}]=\left\{\{0,2,4\}\right\}  $$
y para $S_{3}$ tenemos:
$$[H_{2}]=\left\{\{e,a\},\{e,b\},\{e,c\}\right\}  $$
$$[H_{3}]=\left\{\{e,f,g\}\right\},  $$
por lo que obtenemos el siguiente diagrama:
\begin{scriptsize}
$$
\begin{array}{ll}
\begin{array}{|lc|}

\hline

\mathcal{B}_{[\mathbb{Z}_{6}]} 
&
\begin{array}{|c|c|}
\hline
0 & 63 \\
\hline

\end{array} \\
 
\hline

\mathcal{B}_{[H_{3}]} 
&
\begin{array}{|c|}
\hline
21  \\
42 \\ 
\hline

\end{array} \\

\hline

\mathcal{B}_{[H_{2}]} 
&

\begin{array}{|c|c|}
\hline
9 & 27 \\
18 & 45 \\
36 & 54 \\
\hline

\end{array} \\

\hline

\mathcal{B}_{[H_{0}]} 
&
\begin{array}{|c|c|c|c|c|c|c|c|c|}
\hline
1&3&5&7&11&13&15&23&31\\
2&6&10&14&22&26&30&46&62\\
4&12&20&28&44&52&60&29&61\\
8&24&40&56&25&41&57&58&59\\
16&48&17&49&50&19&51&53&55\\
32&33&34&35&37&38&39&43&47\\
\hline

\end{array} \\

\hline
\end{array}

&

\begin{array}{|lc|}

\hline

\mathcal{B}'_{[S_{3}]} 
&
\begin{array}{|c|c|}
\hline
0 & 63 \\
\hline

\end{array} \\
 
\hline

\mathcal{B}'_{[H_{3}]} 
&
\begin{array}{|c|}
\hline
28\\
35\\

\hline

\end{array} \\

\hline

\mathcal{B}'_{[H_{2}]} 
&
\begin{array}{|c|c|c|c|c|c|}
\hline
5&6&10&15&23&27\\
18&9&17&54&45&46\\
40&48&36&57&58&53\\
\hline

\end{array} \\

\hline

\mathcal{B}'_{[H_{0}]} 
&
\begin{array}{|c|c|c|c|c|c|c|}
\hline
1&3&7&11&21&29&31\\
4&12&13&14&37&39&47\\
16&20&22&15&19&51&55\\
8&24&56&25&42&43&59\\
2&34&50&38&26&30&62\\
32&33&41&49&44&60&61\\

\hline

\end{array} \\

\hline
\end{array}
\end{array}
$$
\end{scriptsize}
Notemos que cada recuadro dentro de $\mathcal{B}_{[H]}$ es una $G$-órbita. 
\end{ejemplo}

Dada la acción de evaluación de $\text{ICA}(G:A)$ sobre $A^{G}$, tenemos pues también las órbitas generadas por esta acción. 
$$ICA(x):=\{y\in A^{G}:\ \exists \sigma\in \text{ICA}(G;A), \ y=\sigma \cdot x= \sigma(x)\}.$$

Sea $G$ un grupo y $A$ un conjunto, dado un elemento $x\in A^{G}$, como consecuencia del Corolario \ref{qw} y del Teorema \ref{qwe} , se tiene que $\mathcal{B}_{G_{x}}=ICA(x)$ y además el número de $Aut_{G}(X)$-órbitas dentro de una $G$-órbita está dado por $[G:N_{G}(G_{x})]$.

Para un subgrupo $H$ de un grupo finito $G$ tenemos las siguientes cantidades:
$$\alpha_{[H]}(G;A)=|\{Gx,\ x\in [H]\}|.$$
Dicho con palabras, es el número de $G$-órbitas dentro de $[H]$. Si tenemos un subgrupo indexado $H_{i}$ simplemente denotamos $\alpha_{i}$.\\

Tenemos un resultado particular de la acción de $G$ sobre $A^{G}$ sobresaliente:

\begin{teorema}
Sea $G$ un grupo finito y $A$ un conjunto finito de tamaño $q\geq 2$. Sea $H\leq G$ tal que $[G:H]=2$. Entonces, $\alpha_{[H]}(G;A)=1$ si y solo si $q=2$. 
\end{teorema}

\begin{proof}
Como $H\leq G$ tiene índice 2, es normal. Fijamos $s\in G\setminus H$. Definimos $x\in A^{G}$ como
$$x(g)=\left\{\begin{array}{cl} 
0 & g\in H\\ 1 & g\in sH=Hs.
\end{array}   \right. .$$
Claramente $G_{x}=H$ y $x\in \mathcal{B}_{[H]}$. Supongamos que $A=\{0,1\}$. Sea $y\in \mathcal{B}_{[H]}$. Como $H$ es normal, $[H]=\{H\}$, entonces $G_{y}=H$. Para cualquier $h\in H$, 
$$y(h)=h^{-1}\cdot y(e)\ \mbox{ y }\ y(hs)=h^{-1}\cdot y(s)=y(s),$$
entonces $y$ es una configuración constante en $H$ y $sH=Hs$. Entonces tenemos que $y=x$ o $$y(g)=\left\{ \begin{array}{cc}1& g\in H \\ 0 & g\in sH=Hs.   \end{array} \right. $$

En este último caso, $s\cdot y=x$ y $y\in Gx$. Esto muestra que existe una única $G$-órbita contenida en $\mathcal{B}_{[H]}$, entonces $\alpha_{[H]}(G;A)=1$.\\
Si $|A|\geq 3$, podemos usar un argumento similar como en la construcción anterior, excepto que ahora $y\in \mathcal{B}_{[H]}$ debe satisfacer que $y(g)\in A \setminus \{0,1\}$ para todo $g\in H$, entonces $y \notin Gx$ y $\alpha_{[H]}(G;A)\geq 2$.
\end{proof}

Dadas todas sus características tenemos que para todo grupo de Dedekind, aquel en el que todos sus subgrupos son normales, si el alfabeto tiene al menos dos elementos, entonces las cajas $\mathcal{B}_{[H]}$, para todo subgrupo no trivial $H$ de $G$,  tienen únicamente una $G$-órbita.

\subsection{Resultados aplicados}

En esta sección aplicamos el resto de los resultados al monoide de autómatas celulares. \\

Un caso muy específico dentro de nuestros trabajo son los grupos de Dedekind, aquellos grupos en los que todos sus subgrupos propios son normales. Nuestros resultados fueron demostrados para estos grupos con un alfabeto finito en \cite{finite monoids}, en esta sección mencionamos estos mismos resultados, justificados por los resultados principales obtenidos. \\

\begin{ejemplo}
Sea $G = \mathbb{Z}_2 \times \mathbb{Z}_2$ el 4-grupo de Klein y $A= \{ 0, 1\}$. Como $G$ es abeliano, $[H] = \{ H\}$, para todo $H \leq G$. Los subgrupos de $G$ son 
\[G, \ H_1 = \langle (0,1) \rangle, \ H_2 = \langle (1,0) \rangle, \ H_3 = \langle (1,1) \rangle, \ \text{y} \ H_0 = \langle (0,0) \rangle, \]
donde $\langle (a,b) \rangle$ denota el subgrupo generado por $(a,b) \in G$. Tenemos entonces el grafo $(C_{G},\varepsilon_{G})$

$$\xymatrix{
 &[G]\ar@(ul,ur)[]\ar[ld]\ar[d]\ar[rd]\ar@/_{6mm}/[dd]&  \\
 [H_{1}]\ar@(ul,dl)[]\ar[rd]& [H_{2}]\ar@(ur,dr)[]\ar[d]& [H_{3}]\ar@(ur,dr)[]\ar[ld]\\
&[H_{0}]\ar@(dr,dl)[]
}$$

\hspace{0.3in}

Debemos notar que cualquier configuración $x : G \to A$ puede ser escrita como una matriz de $2 \times 2$  $(x_{i,j})$ donde $x_{i,j} := x(i-1,j-1)$, $i,j \in \{1,2 \}$. Entonces tenemos que las $G$-órbitas en $A^G$ son
\begin{align*}
& O_1  := \left\{  \left( \begin{tabular}{cc}
$0$ \ & \ $0$ \\
 $0$ \  & \ $0$
 \end{tabular} \right) \right\}, \  \ O_2 := \left\{ \left( \begin{tabular}{cc}
$1$ \ & \ $1$ \\
 $1$ \  & \ $1$
 \end{tabular} \right) \right\}, \\ & \ O_3  :=  \left\{  \left( \begin{tabular}{cc}
$1$ \ & \ $0$ \\
 $1$ \  & \ $0$
 \end{tabular} \right), \left( \begin{tabular}{cc}
$0$ \ & \ $1$ \\
 $0$ \  & \ $1$
 \end{tabular} \right) \right\}, \\[.5em]
 & O_4  :=  \left\{ \left( \begin{tabular}{cc}
$1$ \ & \ $1$ \\
 $0$ \  & \ $0$
 \end{tabular} \right), \left( \begin{tabular}{cc}
$0$ \ & \ $0$ \\
 $1$ \  & \ $1$
 \end{tabular} \right)  \right\}, \\ &  \ O_5  :=  \left\{ \left( \begin{tabular}{cc}
$1$ \ & \ $0$ \\
 $0$ \  & \ $1$
 \end{tabular} \right), \left( \begin{tabular}{cc}
$0$ \ & \ $1$ \\
 $1$ \  & \ $0$
 \end{tabular} \right)   \right\}, \\[.5em]
& O_6  :=  \left\{ \left( \begin{tabular}{cc}
$1$ \ & \ $0$ \\
 $0$ \  & \ $0$
 \end{tabular} \right), \left( \begin{tabular}{cc}
$0$ \ & \ $1$ \\
 $0$ \  & \ $0$
 \end{tabular} \right), \left( \begin{tabular}{cc}
$0$ \ & \ $0$ \\
 $0$ \  & \ $1$
 \end{tabular} \right), \left( \begin{tabular}{cc}
$0$ \ & \ $0$ \\
 $1$ \  & \ $0$
 \end{tabular} \right) \right\},  \\[.5em] 
& O_7  :=  \left\{ \left( \begin{tabular}{cc}
$0$ \ & \ $1$ \\
 $1$ \  & \ $1$
 \end{tabular} \right), \left( \begin{tabular}{cc}
$1$ \ & \ $0$ \\
 $1$ \  & \ $1$
 \end{tabular} \right), \left( \begin{tabular}{cc}
$1$ \ & \ $1$ \\
 $1$ \  & \ $0$
 \end{tabular} \right), \left( \begin{tabular}{cc}
$1$ \ & \ $1$ \\
 $0$ \  & \ $1$
 \end{tabular} \right) \right\}.
\end{align*}
Entonces,
\begin{align*}
& B_{[G]}:=O_1 \cup O_2, \ B_{[H_1]}:=O_4, \ B_{[H_2]}:=O_3, \ B_{[H_3]}:=O_5, \ B_{[H_0]}:=O_6 \cup O_7; \\
& \alpha_{[H_i]}(G;A) = 2, \text { para } i \in \{1,5 \}, \text { y } \alpha_{[H_i]}(G;A) = 1, \text{ para } i \in \{ 2,3,4 \}.
\end{align*}
Entonces tenemos el siguiente el diagrama para todas las  $\mathcal{B}_{[H_{i}]}$.
\begin{center}

\tikzset{every picture/.style={line width=0.75pt}} 

\begin{tikzpicture}[x=0.75pt,y=0.75pt,yscale=-1,xscale=1]

\draw   (396,33) -- (488.5,33) -- (488.5,83.2) -- (396,83.2) -- cycle ;
\draw   (351,110) -- (395.5,110) -- (395.5,149.2) -- (351,149.2) -- cycle ;
\draw   (395,176) -- (487.5,176) -- (487.5,226.2) -- (395,226.2) -- cycle ;
\draw   (425,110) -- (469.5,110) -- (469.5,149.2) -- (425,149.2) -- cycle ;
\draw   (488,109) -- (532.5,109) -- (532.5,148.2) -- (488,148.2) -- cycle ;
\draw    (396,83.2) -- (395.54,108) ;
\draw [shift={(395.5,110)}, rotate = 271.07] [color={rgb, 255:red, 0; green, 0; blue, 0 }  ][line width=0.75]    (10.93,-3.29) .. controls (6.95,-1.4) and (3.31,-0.3) .. (0,0) .. controls (3.31,0.3) and (6.95,1.4) .. (10.93,3.29)   ;
\draw    (488.5,83.2) -- (488.04,108) ;
\draw [shift={(488,110)}, rotate = 271.07] [color={rgb, 255:red, 0; green, 0; blue, 0 }  ][line width=0.75]    (10.93,-3.29) .. controls (6.95,-1.4) and (3.31,-0.3) .. (0,0) .. controls (3.31,0.3) and (6.95,1.4) .. (10.93,3.29)   ;
\draw    (454,83.2) -- (453.54,108) ;
\draw [shift={(453.5,110)}, rotate = 271.07] [color={rgb, 255:red, 0; green, 0; blue, 0 }  ][line width=0.75]    (10.93,-3.29) .. controls (6.95,-1.4) and (3.31,-0.3) .. (0,0) .. controls (3.31,0.3) and (6.95,1.4) .. (10.93,3.29)   ;
\draw    (395.5,149.2) -- (395.04,174) ;
\draw [shift={(395,176)}, rotate = 271.07] [color={rgb, 255:red, 0; green, 0; blue, 0 }  ][line width=0.75]    (10.93,-3.29) .. controls (6.95,-1.4) and (3.31,-0.3) .. (0,0) .. controls (3.31,0.3) and (6.95,1.4) .. (10.93,3.29)   ;
\draw    (488,148.2) -- (487.54,173) ;
\draw [shift={(487.5,175)}, rotate = 271.07] [color={rgb, 255:red, 0; green, 0; blue, 0 }  ][line width=0.75]    (10.93,-3.29) .. controls (6.95,-1.4) and (3.31,-0.3) .. (0,0) .. controls (3.31,0.3) and (6.95,1.4) .. (10.93,3.29)   ;
\draw    (456,149.2) -- (455.54,174) ;
\draw [shift={(455.5,176)}, rotate = 271.07] [color={rgb, 255:red, 0; green, 0; blue, 0 }  ][line width=0.75]    (10.93,-3.29) .. controls (6.95,-1.4) and (3.31,-0.3) .. (0,0) .. controls (3.31,0.3) and (6.95,1.4) .. (10.93,3.29)   ;
\draw    (426.5,84.2) .. controls (402.86,95.03) and (397.65,152.44) .. (433.82,174.23) ;
\draw [shift={(435.5,175.2)}, rotate = 208.93] [color={rgb, 255:red, 0; green, 0; blue, 0 }  ][line width=0.75]    (10.93,-3.29) .. controls (6.95,-1.4) and (3.31,-0.3) .. (0,0) .. controls (3.31,0.3) and (6.95,1.4) .. (10.93,3.29)   ;

\draw (410,44) node [anchor=north west][inner sep=0.75pt]    {$O_{1}$};
\draw (449,45) node [anchor=north west][inner sep=0.75pt]    {$O_{2}$};
\draw (410,189) node [anchor=north west][inner sep=0.75pt]    {$O_{6}$};
\draw (449,190) node [anchor=north west][inner sep=0.75pt]    {$O_{7}$};
\draw (361,116) node [anchor=north west][inner sep=0.75pt]    {$O_{4}$};
\draw (435,116) node [anchor=north west][inner sep=0.75pt]    {$O_{3}$};
\draw (496,115) node [anchor=north west][inner sep=0.75pt]    {$O_{5}$};

\end{tikzpicture}

\end{center}

\end{ejemplo}

Este ejemplo no tiene mayor problema que el contar el número de aristas en el grafo $(C_{G},\varepsilon_{G})$ y los lazos por cada caja que tenga mas de una $G$-órbita dentro de su caja $\mathcal{B}_{[H_{i}]}$. Esto no sucede siempre, pues debemos recordar que las $ICA$-órbitas generan una partición diferente para las cajas.

\begin{ejemplo}
Sean $G=S_{3}$ y $A=\{0,1\}$. Entonces, con base en sus clases de conjugación de subgrupos, tenemos el grafo dirigido:
\begin{center}
\tikzset{every picture/.style={line width=0.75pt}} 

\tikzset{every picture/.style={line width=0.75pt}} 

\begin{tikzpicture}[x=0.75pt,y=0.75pt,yscale=-1,xscale=1]

\draw   (370,29) -- (440,29) -- (440,69) -- (370,69) -- cycle ;
\draw   (320,106) -- (390,106) -- (390,146) -- (320,146) -- cycle ;
\draw   (422,106) -- (538.5,106) -- (538.5,146) -- (422,146) -- cycle ;
\draw   (374,182) -- (444,182) -- (444,222) -- (374,222) -- cycle ;
\draw    (369.5,105.2) -- (369.97,71) ;
\draw [shift={(370,69)}, rotate = 450.79] [color={rgb, 255:red, 0; green, 0; blue, 0 }  ][line width=0.75]    (10.93,-3.29) .. controls (6.95,-1.4) and (3.31,-0.3) .. (0,0) .. controls (3.31,0.3) and (6.95,1.4) .. (10.93,3.29)   ;
\draw    (440.5,105.2) -- (440.03,71) ;
\draw [shift={(440,69)}, rotate = 449.21] [color={rgb, 255:red, 0; green, 0; blue, 0 }  ][line width=0.75]    (10.93,-3.29) .. controls (6.95,-1.4) and (3.31,-0.3) .. (0,0) .. controls (3.31,0.3) and (6.95,1.4) .. (10.93,3.29)   ;
\draw    (406.5,182.2) -- (406.5,71.2) ;
\draw [shift={(406.5,69.2)}, rotate = 450] [color={rgb, 255:red, 0; green, 0; blue, 0 }  ][line width=0.75]    (10.93,-3.29) .. controls (6.95,-1.4) and (3.31,-0.3) .. (0,0) .. controls (3.31,0.3) and (6.95,1.4) .. (10.93,3.29)   ;
\draw    (390.5,181.2) -- (390.03,148) ;
\draw [shift={(390,146)}, rotate = 449.19] [color={rgb, 255:red, 0; green, 0; blue, 0 }  ][line width=0.75]    (10.93,-3.29) .. controls (6.95,-1.4) and (3.31,-0.3) .. (0,0) .. controls (3.31,0.3) and (6.95,1.4) .. (10.93,3.29)   ;
\draw    (422.5,181.2) -- (422.03,148) ;
\draw [shift={(422,146)}, rotate = 449.19] [color={rgb, 255:red, 0; green, 0; blue, 0 }  ][line width=0.75]    (10.93,-3.29) .. controls (6.95,-1.4) and (3.31,-0.3) .. (0,0) .. controls (3.31,0.3) and (6.95,1.4) .. (10.93,3.29)   ;

\draw (396,38) node [anchor=north west][inner sep=0.75pt]    {$S_{3}$};
\draw (342,116) node [anchor=north west][inner sep=0.75pt]    {$H_{f}$};
\draw (432,114) node [anchor=north west][inner sep=0.75pt]    {$H_{a}$};
\draw (395,189) node [anchor=north west][inner sep=0.75pt]    {$H_{e}$};
\draw (466,114) node [anchor=north west][inner sep=0.75pt]    {$H_{b}$};
\draw (500,114) node [anchor=north west][inner sep=0.75pt]    {$H_{c}$};
\draw (392,6) node [anchor=north west][inner sep=0.75pt]  [font=\footnotesize]  {$[ G]$};
\draw (317,86) node [anchor=north west][inner sep=0.75pt]  [font=\footnotesize]  {$[ H_{3}]$};
\draw (476,86) node [anchor=north west][inner sep=0.75pt]  [font=\footnotesize]  {$[ H_{2}]$};
\draw (374,222) node [anchor=north west][inner sep=0.75pt]  [font=\footnotesize]  {$[ H_{1}]$};

\end{tikzpicture}

\end{center}
Si tomamos al primer grupo en la clase como representante notamos que:

$$\begin{array}{l}
N_{1}=  G \\
N_{2}= H_{a}  \\
N_{3}=  G  \\ 
N_{4}=  G.  
\end{array}$$
Notamos que para $H_{e}$, todos los subgrupos de $S_{3}$ lo contienen, por lo tanto hay que calcular todas conjugaciones de todos ellos por los elementos de $S_{3}$ (su normalizador) y agruparlas en $N_{i}$-clases de conjugación:

$$\begin{array}{|ccc|}
\hline
e^{-1}H_{e}e=H_{e} & a^{-1}H_{e}a=H_{e} & b^{-1}H_{e}b=H_{e} \\ c^{-1}H_{e}c=H_{e} & f^{-1}H_{e}f=H_{e} & g^{-1}H_{e}g=H_{e} \\
\hline
e^{-1}H_{a}e=H_{a} & a^{-1}H_{a}a=H_{a} & b^{-1}H_{a}b=H_{c} \\ c^{-1}H_{a}c=H_{b} & f^{-1}H_{a}f=H_{c} & g^{-1}H_{a}g=H_{b} \\
e^{-1}H_{b}e=H_{b} & a^{-1}H_{b}a=H_{c} & b^{-1}H_{b}b=H_{b} \\ c^{-1}H_{b}c=H_{a} & f^{-1}H_{b}f=H_{a} & g^{-1}H_{b}g=H_{c} \\
e^{-1}H_{c}e=H_{c} & a^{-1}H_{c}a= H_{b} & b^{-1}H_{c}b=H_{a} \\ c^{-1}H_{c}c=H_{c} & f^{-1}H_{c}f=H_{b} & g^{-1}H_{c}g=H_{a} \\
\hline
e^{-1}H_{f}e=H_{f} & a^{-1}H_{f}a=H_{f} & b^{-1}H_{f}b=H_{f} \\ c^{-1}H_{f}c=H_{f} & f^{-1}H_{f}f=H_{f} & g^{-1}H_{f}g=H_{f} \\
\hline

\end{array}$$
Entonces el conjunto $U(H_{1})$ sería:

$$ U(H_{1})=\left\{ \{H_{e}\}, \{H_{a},H_{b},H_{c}\},\{H_{f}\},\{S_{3}\}\right\}. $$
Si repetimos estos con los otros subgrupos y sus respectivos normalizadores notamos que se tienen los siguientes conjuntos:
$$\begin{array}{l}
U(H_{1})=\{ \{H_{e}\}, \{H_{a},H_{b},H_{c}\},\{H_{f}\},\{S_{3}\}\}\\ 
U(H_{2})=\{\{H_{a}\}, \{S_{3}\}\}\\
U(H_{3})=\{\{H_{f}\},\{S_{3}\}\}\\
U(H_{4})=\{\{S_{3}\}\}.   \end{array}$$
Como $H_{f}$ es el único subgrupo de $S_{3}$ de índice 2, podemos calcular el rank relativo:
$$\begin{array}{rl} rank(End_{G}(A^{G}):Aut_{G}(A^{G}))&=\sum_{i=1}^{4}{U(H_{i})} - |\kappa_{S_{3}}(A^{G})|\\
&= (4+2+2+1)-1 \\
&=8. \end{array} $$
\end{ejemplo}

\newpage

\section{Conclusiones}
Esta tesis comienza tratando de responder una pregunta muy específica dentro de la teoría de monoides: ¿cuál es el rank relativo del monoide de autómatas celulares cuando trabajamos con un grupo arbitrario? Se comenzó a trabajar bajo una conjetura, un primer estimado de cual podría ser esta cantidad. En términos de objetivos es satisfactorio notar que se logró generar un resultado que es una generalización de la conjetura original.  Se logró visualizar que la conjetura original era errónea, pero se logró crear una nueva idea la cual fue completamente verificada para un caso en específico: los $G$-conjuntos finitos. Y más aún, generamos una conjetura y antecedentes para continuar esta investigación sobre el caso de los $G$-conjuntos infinitos. Recordemos que no solo generamos una conjetura general para el resultado, si no que se plantean también algunas condiciones necesarias para que el rank relativo sea finito en estos casos. \\

Este proyecto nos permitió conocer una pequeña gama de las líneas de investigación que incluyen a los $G$-conjuntos, tales como el anillo de Burnside, los grupos de transformaciones, las aplicaciones de las acciones de grupos en diversas ramas de la ciencia como la física, la biología, entre otras. En este punto de nuestro trabajo nos enfocamos en los endomorfismos de $G$-conjuntos, trabajándolos como un monoide, abordándolo desde un enfoque meramente algebraico.  \\

Notamos que un punto base en el cual continuar nuestro trabajo es el cardinal del $G$-conjunto sobre el que se trabajó. En este trabajo abordamos el caso en el que el $G$-conjunto es finito. Nuestros resultados de estructura y rank relativo han sido probados de manera concluyente en este proyecto. Si bien se hacen menciones y se toman consideraciones para el caso en que el $G$-conjunto sea infinito, en este trabajo no se logra determinar explícitamente el rank relativo para el monoide de endomorfismos de un $G$-conjunto infinito. La continuación directa de este trabajo es determinar con precisión el rank relativo para este caso, así como determinar las diferencias o similitudes estructurales del monoide de endomorfismos con otros objetos conocidos de la misma categoría. \\

Cabe mencionar nuevamente que este trabajo es la cúspide que marca la finalización de este programa de doctorado; el doctorado en ciencias físico-matemáticas con orientación en matemáticas. Entonces se considera necesario mencionar que a lo largo de todo este proceso de investigación se conocieron nuevas ramas, nuevas lineas de investigación, nuevos colaboradores, entre otras situaciones, que tal vez no empataron directamente con los resultados plasmados en el desarrollo de esta tesis, sin embargo lograron generar un crecimiento personal en su servidor como profesionista, como docente, como físico-matemático y como investigador. El compromiso principal es aprovechar todas estas herramientas y saberes obtenidos y desarrollados para continuar siempre trabajando a favor de esta ciencia tan bonita como son las matemáticas y por su puesto también en favor del bien y desarrollo de la humanidad.

\newpage

\appendix

\section*{Apéndices}
\addcontentsline{toc}{section}{Apéndices}

\section{Conjuntos generadores del monoide de transformaciones}

En esta sección conoceremos las construcciones que motivaron nuestro trabajo. Los conjuntos generadores del monoide de transformaciones fueron pieza clave para la construcción de los conjuntos generadores en el monoide de endomorfismos. 
\subsection{Conjuntos generadores}

Siendo $X$ un conjunto finito, en esta sección deduciremos las propiedades que cualquier conjunto generador de $Trans(X)$ y $Sym(X)$.\\

Dado un conjunto $T$ de un monoide $M$, definimos su submonoide generado, $\langle T \rangle$, como el submonoide de $M$ más pequeño que contiene a $T$.
También es bien conocida la igualdad 
$$\langle T \rangle = \{t_{1}t_{2}\cdots t_{k}\in M:\ t_{i}\in T,\ k\in \mathbb{N}\}.$$

Si se tiene que $\langle T \rangle =M$, decimos que el $T$ es un \emph{conjunto generador} de $M$.\\
Además decimos que un monoide $M$ es finitamente generado si existe un conjunto generador finito de $M$.

\begin{ejemplo}
Sea $M=\mathbb{N}$, notemos que el conjunto $\{1\}$ es un conjunto generador, pues se satisface que:
$$n=1+1+...+1(n-veces)= t_{1}t_{2}...t_{n}.$$
\end{ejemplo}
\subsubsection{Generadores de $Trans(X)$}
Primero abordaremos las propiedades de los conjuntos generadores para $Trans(X)$.\\

Dado un conjunto finito $X$, definimos el rango de una transformación $f:X\rightarrow X$ como $r(f)=|Im(f)|$. Es fácil ver que $r(f)\leq |X|$ para cualquier transformación $f\in Trans(X)$. También es sencillo ver que el rango se refiere al número de valores distintos que la función $f$ toma. Además se verifica que $r(f)=|X|$ si y solo si $f\in Sym(X)$.

\begin{proposicion}\label{p1}
Dadas dos transformaciones $f,g\in Trans(X)$ se tiene que $$r(f\circ g)\leq \min\{r(f),r(g)\}.$$
\end{proposicion}

\begin{proof}
Sean $|X|=k$, $r(f)=n$ y $r(g)=m$, se tiene que al aplicar $g$ a todos los elementos de $X$ obtenemos $m$ valores distintos de $g(x_{i})$, a los cuales denotaremos $g_{1},\dots,g_{m}$. 

\begin{center}
\xymatrix{  & x_{1}\ar[ddr] & \cdots & \cdots & \cdots & \cdots & x_{k}\ar[ddl] \\
g & & & & & & \\
& & g_{1} &\cdots & \cdots & g_{m}}
\end{center}
Al aplicar $f$ se tendrán dos casos, si $n \leq m$ tendremos entonces a lo mucho $n$ valores $f(g_{k})$. 
\begin{center}
\xymatrix{  & x_{1}\ar[ddr] &\dots & \cdots & \cdots & \cdots & \cdots & x_{k}\ar[ddl] \\
g & & & & &  & & \\
& & g_{1}\ar[ddr] &\cdots & \cdots  & \cdots & g_{m}\ar[ddl]\\
f & & & & & & \\
& &  & f_{1} & \cdots& f_{n} & 
}
\end{center}

En el caso donde $n \geq m$, a lo mucho tendremos $m$ valores de $f(g_{k})$.   

\begin{center}
\xymatrix{  & x_{1}\ar[ddr] &\dots & \cdots & \cdots & \cdots & \cdots & x_{k}\ar[ddl] \\
g & & & & &  & & \\
& & g_{1}\ar[dd] &\cdots & \cdots  & \cdots & g_{m}\ar[dd]\\
f & & & & & & \\
&f_{1} \ \cdots & f_{\alpha_{1}} & \cdots & \cdots& \cdots  &f_{\alpha_{m}} &\cdots \ f_{n}
}

\end{center}

 donde $1 \leq \alpha_{i} \leq n$. Entonces el número máximo de valores que podemos tomar esta limitado por el menor de $m$ y $n$, es decir $r(f\circ g)\leq \min\{r(f),r(g)\}$.

\end{proof}

Esta proposición tiene consecuencias bastante útiles, las cuales serán utilizadas en secciones posteriores. A continuación escribimos la primera de éstas.

\begin{corolario}\label{p2}
Una permutación solo puede ser escrita como el producto de permutaciones. 
\end{corolario}

\begin{proof}
Sabemos que si $|X|=n$,  se tiene que $r(h)=n$, $\forall h\in Sym(X)$. Si $h=f\circ g$, entonces se debe cumplir que $$n=r(h) \leq \min\{r(f),r(g)\},$$ en donde la única opción es que $r(f)=r(g)=n$. 
\end{proof}
Además se tiene que $Sym(X)$ es cerrado bajo composiciones, esto significa que si $r(f)=r(g)=n$, se tiene que $r(f\circ g)=n$.\\

\begin{proposicion}\label{p3}
Todo conjunto $W$ generador de $Trans(X)$ contiene un conjunto generador de $Sym(X)$.
\end{proposicion}
Esto también es consecuencia directa de la Proposición \ref{p1}, pues $Sym(X)\subseteq Trans(X)$ y estos solo pueden ser generados por elementos de rango $|X|$.\\

Un conjunto generador $W$ de un monoide $M$ se dice \emph{irreducible} si no existe un subconjunto propio de $W$ el cual también genere a $M$.

\begin{lema}\label{p4}
Sea $X$ un conjunto finito, cada conjunto generador de $Trans(X)$ contiene al menos un elemento de rango $n-1$.
\end{lema}

\begin{proof}
Esto se sigue de que existe al menos un elemento $h$ de rango $n-1$ en $Trans(X)$. Entonces si $h=f\circ g$ tiene rango $n-1$, se tiene que $$n-1\leq \min\{r(f),r(g)\},$$ entonces se tiene que $r(f)$ y $r(g)$ toman el valor de $n$ ó $n-1$. Por otro lado, si ambos toman el valor de $n$, por cerradura dentro de $Sym(X)$, $h$ tendría rango $n$. Entonces por lo menos uno de los dos debe de tener rango $n-1$.
\end{proof}

Este lema inspira el siguiente teorema. Necesitaremos una definición auxiliar para esto.\\

Llamamos \emph{defecto} de una transformación $f\in Trans(X)$, a la cantidad $d(f)=|X|-r(f)$, donde notamos que $0\leq d(f) < |X|$.

\begin{teorema}
Sea $X$ un conjunto finito de $n$ elementos y $W$ un conjunto generador de $Trans(X)$. Entonces $W$ es irreducible si y solo si $W=S_{1}\cup\{\alpha\}$, donde $S_{1}$ es un conjunto generador irreducible de $Sym(X)$ y $r(\alpha)=n-1$.
\end{teorema}

\begin{proof}
Dada la Proposición \ref{p1} y todas sus consecuencias directas: las Proposiciones \ref{p2}, \ref{p3} y \ref{p4}, basta probar que $W=S_{1}\cup\{\alpha\}$ donde $r(\alpha)=n-1$.\\
Como $\langle S_{1} \rangle = Sym(X)$, entonces $\sigma \in \langle W \rangle$, $\forall \sigma \in Sym(X)$. Como $r(\alpha)=n-1$, podemos visualizar $\alpha$ como:
$$\alpha = \left( \begin{array}{cccccccccccc}  1 & 2& \cdots& i-1 & i & i+1& \cdots &j-1&j&j+1&\cdots&n   \\
a_{1} & a_{2} & \dots & a_{i-1}& a & a_{i+1}&\cdots &a_{j-1}&a&a_{j+1}&\cdots & a_{n}   \end{array} \right). $$
Sea $f\in Trans(X)$, una transformación cualquiera de rango $n-1$, también podemos visualizarla como:
$$f = \left( \begin{array}{cccccccccccc}  1 & 2& \cdots& k-1 & k & k+1& \cdots &l-1&l&l+1&\cdots&n   \\
b_{1} & b_{2} & \dots & b_{k-1}& b_{k} & b_{k+1}&\cdots &b_{l-1}&b_{l}&b_{l+1}&\cdots & b_{n}   \end{array} \right). $$

Sea $\pi\in Sym(X)$, tal que $\pi(k)=i$ y $\pi(l)=j$. Entonces se tiene que 
$$\alpha \circ \pi  = \left( \begin{array}{cccccccccccc}  1 & 2& \cdots& k-1 & k & k+1& \cdots &l-1&l&l+1&\cdots&n   \\
c_{1} & c_{2} & \dots & c_{k-1}& a & c_{k+1}&\cdots &c_{l-1}&a&c_{l+1}&\cdots & c_{n}   \end{array} \right). $$

Sea $x\in X$ el elemento que faltan en la imagen de $\alpha$ y sea $y$ el elemento faltante en la imagen de $f$. Entonces podemos encontrar una permutación $\sigma \in Sym(X)$ tal que
$$\sigma = \left( \begin{array}{cccccccccccc} c_{1} & c_{2} &\cdots & c_{k-1}& c_{k+1}&\cdots c_{l-1}&c_{l+1}&\cdots & c_{n}&a&x\\
b_{1} & b_{2} &\cdots & b_{k-1}& b_{k+1}&\cdots b_{l-1}&b_{l+1}&\cdots & b_{n}&b&y
\end{array} \right), $$
donde podemos identificar $c_{r}=a_{\pi(r)}$, entonces es fácil ver que $f=\sigma \circ \alpha \circ \pi$. Esto muestra que cualquier transformación de rango $n-1$ está contenida en $\langle W\rangle$. O dicho de otra manera, cualquier transformación de defecto 1. Probaremos por inducción que cada transformación de defecto $k+1$ también está contenida en $\langle W \rangle$. La cerradura de las permutaciones prueba el caso $d(f)=0$ y lo anterior prueba el caso $d(f)=1$.  Supongamos pues que se cumple para cualquier valor $k$ y probemoslo para cualquier transformación $f\in Trans(X)$ de defecto $k+1$.\\

Sea $f\in Trans(X)$, tal que $d(f)=k+1$, con $k>0$ y suponemos que cualquier transformación de defecto menor o igual a $k$ está contenida en $\langle W \rangle$. Entonces, existe por lo menos un elemento $a\in X$, el cual tiene más de una preimagen, es decir, existen $b_{1},b_{2}\in X$ tales que $f(b_{1})=f(b_{2})=a$. Identifiquemos pues:
$$f = \left( \begin{array}{cccccccccccc}  1 & 2& \cdots& b_{1}-1 & b_{1} & b_{1}+1& \cdots &b_{2}-1&b_{2}&b_{2}+1&\cdots&n   \\
f_{1} & f_{2}& \cdots& f_{b_{1}-1} & a & f_{b_{1}+1}& \cdots &f_{b_{2}-1}&a&f_{b_{2}+1}&\cdots&f_{n}    \end{array} \right). $$

Ahora consideremos la transformación $\mu \in Trans(X)$.

$$\mu = \left( \begin{array}{cccccccccccc}  1 & 2& \cdots& b_{1}-1 & b_{1} & b_{1}+1& \cdots &b_{2}-1&b_{2}&b_{2}+1&\cdots&n   \\
1 & 2& \cdots& b_{1}-1 & b_{1} & b_{1}+1& \cdots &b_{2}-1&b_{1}&b_{2}+1&\cdots&n    \end{array} \right). $$
Notemos que $d(\mu)=1$, por lo tanto $\mu \in \langle W \rangle$. E identificaremos como $\tilde{a}\in X$ a un elemento que no este en la imagen de $f$, es decir, $\forall x\in X$ se tiene que $f(x)\neq \tilde{a}$. Entonces consideremos una transformación $\delta \in Trans(X)$ la cual cumple que $\delta(b_{2})=\tilde{a}$ y además $f(x)=\delta(x)$, para toda $x\neq b_{2}$. 
$$\delta = \left( \begin{array}{cccccccccccc}  1 & 2& \cdots& b_{1}-1 & b_{1} & b_{1}+1& \cdots &b_{2}-1&b_{2}&b_{2}+1&\cdots&n   \\
f_{1} & f_{2}& \cdots& f_{b_{1}-1} & a & f_{b_{1}+1}& \cdots &f_{b_{2}-1}&\tilde{a}&f_{b_{2}+1}&\cdots&f_{n}    \end{array} \right). $$
Entonces $d(\delta)=k$ y por lo tanto $\delta \in \langle W\rangle$. Es fácil verificar que $f=\delta \circ \mu$,

$$\left( \begin{array}{c}
\xymatrix{
&1\ar[dd] & 2\ar[dd] & \cdots & b_{1}\ar[dd] & \cdots & b_{2}\ar[ddll] & \cdots & n\ar[dd] \\
\mu & \\
&1\ar[dd] & 2\ar[dd] & \cdots & b_{1}\ar[dd] & \cdots & b_{2}\ar[dd] & \cdots & n\ar[dd] \\
\delta &   \\
&f_{1} & f_{2} & \cdots & a & \cdots & \tilde{a} & \cdots & f_{n}
} \end{array}\right)$$
y por tanto $f\in \langle W \rangle$.

\end{proof}

\subsubsection{Generadores de $Sym(X)$}

En esta sección presentaremos varias versiones de conjuntos generadores para $Sym(X)$. Deduciremos axiomáticamente las propiedades de estos.\\

Supongamos que $X$ es un conjunto finito con $n$ elementos. Entonces una permutación puede visualizarse como un arreglo que intercambia los elementos de $X$.

Llamaremos \emph{ciclo} al conjunto de imágenes de una permutación $\sigma$, según se cumpla que $\sigma^{k}(x)=x$. Una permutación puede ser descrita a partir de sus ciclos.

\begin{ejemplo}
Dada un conjunto $X$ con cinco elementos, denotados por $\{1,2,3,4,5\}$. Tomamos una permutación $\sigma \in Sym(X)$ descrita por el siguiente diagrama:

$$\sigma= \left\{ \begin{array}{c}\xymatrix{1\ar[dr] & 1\\
2\ar[ur]&2\\  3\ar[ddr]&3 \\ 4\ar[ur]&4 \\ 5\ar[ur]&5} \end{array} \right.$$
Dicho de otra manera tenemos que $\sigma(1)=2$, $\sigma(2)=1$, $\sigma(3)=5$, $\sigma(4)=3$ y $\sigma(5)=4$. Queremos visualizar que si tomamos $\sigma(1)=2$ y le aplicamos nuevamente $\sigma$, $\sigma(2)=\sigma \circ \sigma (1)=\sigma^{2}(1)=1$. Es decir, que $1$ se transforma en $2$ y luego $2$ se transforma en $1$, $1 \rightarrow 2 \rightarrow 1$. Este ciclo tiene únicamente 2 elementos de $X$, así que lo llamamos \emph{2-ciclo}. Si empezamos en el número $2$, los elementos que aparecen vuelven a ser los mismos. Por el otro lado, si empezamos en el número $3$ su primer imagen es el número 5, si aplicamos nuevamente $\sigma$ obtendremos el número y al aplicarlo por tercera vez obtenemos nuevamente 3, i.e., $3 \rightarrow 5 \rightarrow 4 \rightarrow 3$, este ciclo tiene 3 elementos, por lo que lo llamamos \emph{3-ciclo}. Notemos que nuestra función $\sigma$ esta compuesta por un 2-ciclo y un 3-ciclo.\\
Por otro lado, consideremos las siguientes funciones:

$$
\begin{array}{cc}
\alpha= \left\{ \begin{array}{c}\xymatrix{1\ar[dr] & 1\\
2\ar[ur]&2\\  3\ar[r]&3 \\ 4\ar[r]&4 \\ 5\ar[r]&5} \end{array} \right.

&

\beta= \left\{ \begin{array}{c}\xymatrix{1\ar[r] & 1\\
2\ar[r]&2\\  3\ar[ddr]&3 \\ 4\ar[ur]&4 \\ 5\ar[ur]&5} \end{array} \right.
\end{array}.
$$
Notemos que $\alpha$ tiene un 2-ciclo, $1\rightarrow 2 \rightarrow 1$ y está conformado por tres 1-ciclos, los cuales notaremos no serán importantes para lo que queremos hacer. De igual manera $\beta$ está conformado por un 3-ciclo y dos 1-ciclos.\\
También es importante que notemos lo siguiente:
$$
\alpha \circ \beta= \left\{ \begin{array}{c}\xymatrix{
1\ar[r]&1\ar[rd]&1\\
2\ar[r]&2\ar[ru]&2\\
3\ar[ddr]&3\ar[r]&3 \\ 
4\ar[ur]&4\ar[r]&4 \\ 
5\ar[ur]&5\ar[r]&5} 
\end{array} \right. 
=  
\beta \circ \alpha = \left\{ \begin{array}{c}\xymatrix{
1\ar[rd]&1\ar[r]&1\\
2\ar[ru]&2\ar[r]&2\\
3\ar[r]&3\ar[rdd]&3 \\ 
4\ar[r]&4\ar[ru]&4 \\ 
5\ar[r]&5\ar[ru]&5} 
\end{array} \right. 
= \left\{ \begin{array}{c}
\xymatrix{
1\ar[rd]&1\\
2\ar[ru]&2\\
3\ar[ddr]&3 \\ 
4\ar[ur]&4 \\ 
5\ar[ur]&5} 
\end{array} \right.= \sigma . 
$$
Si denotamos $\alpha$ y $\beta$ por sus ciclos $(12)$ el ciclo para $\alpha$ y $(354)$ el ciclo para $\beta$. Podemos escribir a $\sigma$ como $\sigma=(12)(354)$, omitiendo el símbolo de composición de funciones.

\end{ejemplo}
Notemos que el orden dentro de los ciclos no importa $(12)=(21)$, o $(354)=(543)=(435)$.  Sin embargo es importante notar que los ciclos siguen el orden de composición $3\rightarrow 5 \rightarrow 4$. Es decir que $(354)\neq(345)$. También veamos que los ciclos  conmutan para formar a $\sigma$, es decir $\sigma=(354)(12)$ es también una forma valida.

\begin{lema}
El producto de $n$-ciclos disjuntos conmuta.

\end{lema}

\begin{proof}
Sea $\sigma$ un $k$-ciclo y $\beta$ un $m$-ciclo. Verifiquemos esto por casos.
\begin{enumerate}
\item[] Caso 1. Sea $x\in X$ un elemento no parte de ninguno de los ciclos, entonces $\sigma(x)=x=\beta(x)$. Entonces tenemos que 
$$x=\sigma(x)=\sigma(\beta(x))=\sigma \circ \beta (x)$$
$$x=\beta(x)=\beta(\sigma(x))=\beta \circ \sigma (x).$$
\item[] Caso 2. Podemos suponer sin pérdida  de generalidad que $x\in X$ pertenece únicamente a uno de los ciclos, digamos $\sigma$, significa que $\sigma(x)$ es parte del ciclo y por tanto no es parte de $\beta$, es decir  $\sigma(x)=\beta(\sigma(x))$, además $x=\beta(x)$ y $\sigma(x)=\sigma(\beta(x))$.
\end{enumerate}

\end{proof}

Es fácil ver que toda permutación esta formada por $n$-ciclos, o es un $n$-ciclo. Entonces, podemos decir que el conjunto de todos los $n$-ciclos existentes es un generador de $Sym(X)$. También es importante mencionar que a los 2-ciclos también se les conoce como \emph{transposiciones}.

\begin{proposicion}
Toda permutación $\sigma \in Sym(X)$ puede ser escrita como el producto de $n$-ciclos.  
\end{proposicion}

Esto se sigue de la manera en la que construimos la notación de los $n$-ciclos. A continuación presentamos un resultado más relevante. Notemos que si $X$ es un conjunto con un único elemento, su monoide de transformaciones esta formado solo por la identidad.

\begin{proposicion}
Sea $X$ un conjunto de cardinal $n\geq 2$, entonces $Sym(X)$ es generado por sus transposiciones.
\end{proposicion}

\begin{proof}
Dado el enunciado anterior, basta con verificar que los $n$-ciclos son generados transposiciones. Los 1-ciclos son la función identidad y debemos ver que $(1)=(12)^{2}$. Así que son generados. Los 2-ciclos son generados por ellos mismos. Se puede probar esto por inducción. Supongamos que los $k$-ciclos son generados por transposiciones. Dado un $k+1$-ciclo $(a_{1}a_{2}\cdots a_{k+1})$, basta con mostrar que $(a_{1}a_{2}\cdots a_{k+1})=(a_{1}\cdots a_{k})(a_{k}a_{k+1})$. Pero esto es bastante sencillo ver en el siguiente diagrama.
$$\begin{array}{c}  
\xymatrix{  & (a_{k}a_{k+1}) &  & (a_{1}\cdots a_{k})   \\
a_{1}\ar[rr] & & a_{1}\ar[rrd] & &  a_{1} \\
a_{2}\ar[rr] & & a_{2}\ar[rrd] & & a_{2} \\
\vdots & & \vdots & & \vdots \\
a_{k}\ar[rrd] & & a_{k}\ar[uuurr]& & a_{k} \\
a_{k+1}\ar[rru] & & a_{k+1}\ar[rr] & & a_{k+1}
}
\end{array}$$
como $(a_{1}\cdots a_{k})$ es generado por transposiciones, por hipótesis, entonces se cumple la afirmación.

\end{proof}

Entonces el conjunto de todas las transposiciones es un conjunto generador para $Sym(X)$. Notemos que el conjunto de transposiciones es un conjunto de $\frac{n(n-1)}{2}$ elementos.

\begin{proposicion}
Toda transposición es generada por elementos de la forma $(i\ i+1)$.
\end{proposicion}

\begin{proof}
Dada una transposición $(ab)$. Como se tiene que $(ab)=(ba)$, podemos suponer sin pérdida  de generalidad que $a < b$. Probaremos este hecho por inducción sobre $b-a$, que $(ab)$ es el producto de transposiciones de la forma $(i\ i+1)$. Esto es obvio cuando $b-a=1$, pues se tiene que $(ab)=(a\ a+1)$, la cual es una transposición de la forma deseada. Ahora, supongamos que se cumple para $b-a=k$. Solo debemos notar que se tiene que $(ab)=(a\ a+1)(a+1\ b)(a\ a+1)$, lo cual está descrito el en siguiente diagrama.
\begin{scriptsize}
$$\begin{array}{c}  
\xymatrix{
 & (a\ a+1) &  & (a+1\ b) &  & (a\ a+1) & \\
a\ar[rrd] & & a\ar[rr] & & a\ar[rrd] & & a\\
a+1\ar[rru] & & a+1\ar[rrd] & & a+1\ar[rru] & & a+1\\
b\ar[rr] & & b\ar[rru] & & b\ar[rr] & & b
}
\end{array}$$
\end{scriptsize}
Es evidente que $(a\ a+1)$ es una transposición de la forma deseada y $(a+1\ b)$ es una transposición cuya diferencia es igual a $k$, entonces por hipótesis de inducción es generada por transposiciones de la forma deseada.

\end{proof}

\begin{ejemplo}
Si tenemos un conjunto de cardinal finito $X$, podemos identificar sus elementos únicamente como $X=\{1,2,\dots,n\}$. Sea $n=6$, entonces tenemos que dada la transposición $(2\ 5)$ la podemos visualizar como sigue. 

$$\begin{array}{ccc}    
(2\ 5)&=&(2\ 3)(3\ 4)(4\ 5)(3\ 4)(2\ 3)\\
 \xymatrix{1\ar[rr]& &1\\
2\ar[rrddd]& &2\\
3\ar[rr]& &3\\
4\ar[rr]& &4\\
5\ar[rruuu]& &5\\
6\ar[rr]& &6}
&

&
\xymatrix{ 1\ar[r]&1\ar[r]&1\ar[r]&1\ar[r]&1\ar[r]&1\\
2\ar[rd]&2\ar[r]&2\ar[r]&2\ar[r]&2\ar[rd]&2\\
3\ar[ru]&3\ar[rd]&3\ar[r]&3\ar[rd]&3\ar[ru]&3\\
4\ar[r]&4\ar[ru]&4\ar[rd]&4\ar[ru]&4\ar[r]&4\\
5\ar[r]&5\ar[r]&5\ar[ru]&5\ar[r]&5\ar[r]&5\\
6\ar[r]&6\ar[r]&6\ar[r]&6\ar[r]&6\ar[r]&6}
\end{array}$$
\end{ejemplo}
Entonces podemos ver que este nuevo conjunto de transposiciones de la forma $(i\ i+1)$ es también un conjunto generador para $Sym(X)$, este en particular cuenta únicamente con $n-1$ elementos.\\

A continuación mencionaremos algunas propiedades que cumple el $n$-ciclo $(1\ 2\ \dots \ n)$, las cuales serán de utilidad para la construcción del siguiente conjunto generador. Sabemos que todas las transposiciones son sus propios inversos. Notemos que, para cualquier valor $x< n$, la permutación $(1\ 2\ \dots n)$ mapea $x \rightarrow  x+1$ y $n \rightarrow 1$. Entonces, si denotamos $\sigma=(1\ 2\ \dots n)$, entonces se tiene que 
$$\begin{array}{cccccccccccccc}
  & \sigma &  & \sigma^{2} &  & \sigma^{3} &  & \sigma^{n-2} &  & \sigma^{n-1} &  & \sigma^{n} &  \\

1& \longrightarrow & 2& \longrightarrow &3& \longrightarrow& \cdots& \longrightarrow & n-1& \longrightarrow& n& \longrightarrow &1 \\

2& \longrightarrow & 3& \longrightarrow &4& \longrightarrow& \cdots& \longrightarrow & n& \longrightarrow& 1& \longrightarrow &2 \\

& & & & & & \vdots \\

n& \longrightarrow & 1& \longrightarrow &2& \longrightarrow& \cdots& \longrightarrow & n-2& \longrightarrow& n-1& \longrightarrow &n \\

\end{array},$$
lo cual nos permite ver que $\sigma^{n}=Id$. Y por la asociatividad de la operación podemos fácilmente que $\sigma\circ \sigma^{n-1}=\sigma^{n-1}\circ \sigma = Id$, es decir que $\sigma^{n-1}$ es la permutación inversa de $\sigma$. También, en la notación de $n$-ciclos y después de desarrollarlo, se puede ver que $(1\ 2\ \dots \ n)^{-1}=(n \ n-1 \ \dots \ 2 \ 1)$. En consecuencia tenemos el siguiente lema que nos será de utilidad.

\begin{lema}
Sea $X$ un conjunto de $n$ elementos, dada una permutación $\alpha \in Sym(X)$ y un $k$-ciclo $(a_{1}\ \dots \ a_{k})$, entonces se cumple que:
$$\alpha (a_{1}\ \dots \ a_{k})\alpha^{-1}=(\alpha(a_{1})\dots\alpha(a_{k})).$$
\end{lema}

\begin{proof}
Basta notar que el elemento $\alpha(a_{i})$ es mapeado al elemento $\alpha(a_{i+1})$, lo cual es fácil de visualizar que cuando aplicamos $\alpha(a_{1}\cdots a_{n})\alpha^{-1}$ tenemos que:
$$\begin{array}{rl}\alpha(a_{1}\cdots a_{n})\alpha^{-1}(\alpha(a_{i}))&=
\alpha(a_{1}\cdots a_{n})\left( \cancel{\alpha^{-1}}(\cancel{\alpha}(a_{i})) \right)\\
&=\alpha(a_{1}\cdots a_{k})(a_{i}) \\
&=\alpha(a_{i+1}).
\end{array}
$$
\end{proof}

Si representamos $(1\ 2\ \cdots \ n)$ por $\sigma$, como consecuencia del lema anterior se tiene que:
\begin{enumerate}
\item[i)] $\sigma(a\ b)\sigma^{-1}=(a+1\ b+1)$.
\item[ii)] $ \sigma^{k}(a\ b)\sigma^{-k}=(a+k\ b+k)$.
\end{enumerate}
Tomando en cuenta que $n \mapsto 1$, y por consecuencia al aplicarse a la transposición $(1\ 2)$, tenemos que:
\begin{enumerate}
\item[iii)] $\sigma (1\ 2)\sigma^{-1}=(2 \ 3)$.
\item[iv)] $\sigma^{k}(1\ 2)\sigma^{-k}=(k+1\ k+2)$.
\end{enumerate}
Es decir que todas las transposiciones de la forma $(i\ i+1)$ son generadas por los elementos $(1\ 2)$ y $(1\ 2\ \cdots n)$. Además es fácil ver que $(1\ 2\ \cdots n)=(1\ 2)(2\ 3\ \cdots n)$. Por lo tanto $\{(1\ 2),(2\ 3\ \cdots n)\}$ es un conjunto generador de $Sym(X)$.

\begin{ejemplo}
Sea $X$ un conjunto con seis elementos los cuales denotaremos por $\{1,2,\cdots, 6\}$ entonces denotaremos por $\alpha$ a la transposición $(1\ 2)$, por $\beta$ al $n-1$-ciclo $(2 \ 3\ \cdots n)$ y por $\sigma$ al $n$-ciclo $(1\ 2\ \cdots n)$. Entonces tenemos que $\alpha^{-1}=\alpha$, $\beta^{-1}=\beta^{4}$, además $\sigma=\alpha \beta$ y $\sigma^{-1}=(\alpha \beta)^{-1}=\beta^{-1}\alpha^{-1}=\beta^{4}\alpha$. Entonces tenemos que para el $n$-ciclo $(2\ 4\ 3\ 1 \ 5)$ tenemos que:
$$\begin{array}{l}
(2\ 4\ 3\ 1\ 5)=(2\ 4)\ (4\ 3)\ (3\ 1)\ (1\ 5)\\
=(2\ 3)(3\ 4)(2\ 3)\ (3\ 4)\ (1\ 2)(2\ 3)(1\ 2)\ (1\ 2)(2\ 3)(3\ 4)(4\ 5)(3\ 4)(2\ 3)(1\ 2)
\end{array},$$
siendo $(k\ k+1)= (1\ 2)^{k-1} \beta ^{k-1}(1\ 2) \beta^{-(k-1)}(1\ 2)^{k-1}$ y al simplificar todos los cálculos obtenemos que
$$(2\ 4\ 3\ 1\  5 )=(1\ 2)\beta^{2}(1\ 2)\beta^{4}(1\ 2)\beta(1\ 2)\beta^{4}(1\ 2)\beta^{4}(1\ 2)\beta(1\ 2)\beta^{2}(1\ 2)\beta^{2}.$$
\end{ejemplo}

\subsection{El monoide de transformaciones de un conjunto infinito}

En esta sección mostramos como el rank relativo de $Trans(X)$ sobre su grupo de unidades es 2, para el caso en que $X$ es un conjunto infinito. \\
Recordamos algunos parámetros asociados a la medición de las cardinalidades de funciones. La idea de estos parámetros asociados a funciones es que queremos medir que tan alejada esta una función de ser de un tipo especial, por ejemplo inyectiva o sobreyectiva. \\

Llamamos \emph{rango} de una función $f:X\rightarrow X$ al cardinal de su imagen, i.e., $ran(f)=|Im(f)|$.\\

Se dice que dos funciones $f$ y $g$ son $\mathcal{J}$-equivalentes (o también $\mathcal{D}$-equivalentes) si y solo si tienen el mismo rango. Para el caso finito tenemos que $ran(f)=|X|$ si y solo si $f$ es sobreyectiva y  además, en este caso, el rango describe que tan diferente es $f$ de una función biyectiva.  De cualquier modo, cuando $X$ es un conjunto infinito existen funciones no sobreyectivas de rango $|X|$. Una mejor medida de la no sobreyectividad de una función es su \emph{defecto}, el cual es definido como el tamaño del complemento de la imagen de una función $f$, i.e., $def(f)=|X \setminus Im(f) |$. Es fácil ver que $def(f)=0$ si y solo si $f$ es sobreyectiva. Estos son los parámetros más usuales para trabajar con los cardinales de funciones, sin embargo para los objetivos de esta sección es necesario definir algunos más. \\

El \emph{shift} de una función mide que tan diferente es una función de la identidad y se define como el cardinal del conjunto de elementos transformados por la función, i.e., $shift(f)=|\{x\in X:\ f(x)\neq x\}|$. El \emph{colapso} de una función mide que tan diferente es una función $f$ de una función inyectiva y se define a partir de lo siguiente. Consideremos la \emph{relación kernel} de $\alpha$:

$$ker(\alpha):=\{(x,y)\in X\times X,\ \alpha(x)=\alpha(y)\}.$$

Es fácil ver que esta es una relación de equivalencia. Sea $T_{\alpha}$ el conjunto de representantes de clase  de dicha relación de equivalencia; en otras palabras, sea $T_{\alpha}$ un conjunto tal que para cada clase de equivalencia $C$ de $ker(\alpha)$ tenemos que $|C \cap T_{\alpha} |=1$. (También solemos referirnos a $T_{\alpha}$ como un transversal de $ker(\alpha)$.)\\
Entonces el \emph{colapso} de $\alpha$ está definido como:
$$col(\alpha)=|X\setminus T_{\alpha}|.$$
Es fácil ver que $$col(\alpha)=\sum_{y\in Im(\alpha)}{(|\alpha^{-1}(y)|-1)},$$
y que este número cardinal no depende de $T_{\alpha}$.\\

Definimos también el \emph{índice de contracción infinito}, una cantidad que nos ayuda a medir cuantas ''clases grandes'' tiene el kernel de  una transformación. Este se define como el número de clases en $ker(\alpha)$ de tamaño $|X|$. En otras palabras, sea el conjunto 
$$K(\alpha)=\{x\in X:\ |\alpha^{-1}(x)|=|X|\},$$
entonces el índice de contracción infinito esta dado por
$$k(\alpha)=|K(\alpha)|.$$

\begin{proposicion}
Sea $X$ un conjunto infinito y sea $\mu\in Trans(X)$ una transformación arbitraria. Entonces $\langle Sym(X),\mu \rangle\neq Trans(X)$.
\end{proposicion}
\begin{proof}
Supongamos que$\langle Sym(X),\mu \rangle= Trans(X)$. Sea $\alpha\in Trans(X)$ cualquier mapeo inyectivo de $X$ en un subconjunto propio de $X$. Descompongamos $\alpha$ como el producto de elementos en $Sym(X)\cup \{\mu\}$:
$$\alpha=\beta_{1}\beta_{2}\cdots\beta_{k}.$$

Como $\alpha$ no es una biyección, almenos un elemento $\beta_{i}$ debe ser igual a $\mu$; digamos que $\beta_{j}$ es el primer elemento con dicha cualidad. Entonces:
$$\beta_{j-1}^{-1}\dots \beta_{1}^{-1}\alpha= \mu \beta_{j+1}\cdots \beta_{k}.$$

Notemos que del lado izquierdo tenemos la composición de funciones inyectivas y por consecuencia, como la composición de funciones inyectivas sigue siendo una función inyectiva, $\mu$ debe ser inyectiva también. Entonces $\langle Sym(X),\mu\rangle$ consiste enteramente de funciones inyectivas y por tanto no puede ser todo $Trans(X)$. 
\end{proof}

\begin{proposicion}
Sea $X$ un conjunto infinito y sea $\mathcal{J}$ el conjunto de todas las funciones de $X$ en $X$ de rango $|X|$. Entonces $Trans(X)=\langle \mathcal{J}\rangle$. Definimos pues una transformación $\gamma:X\rightarrow X$ como:
$$\gamma(x)=\left\{ \begin{array}{cc} \alpha \beta^{-1} (x) & x\in Y\\
x & x\in X.\end{array} \right.$$
Como $Z \subseteq Im(\gamma)$ y $|Z|=|X|$ se tiene que $\gamma \in \mathcal{J}$. Ahora, para cualquier $x\in X$ se tiene que $\beta(x)\in Y$. Entonces se tiene que:
$$\gamma(\beta(x))=\alpha \beta^{-1}(\beta(x))=\alpha (x). $$
Asi, $\alpha=\gamma \beta \in \langle \mathcal{J} \rangle$.
\end{proposicion}

\begin{proof}
Sea $\alpha\in Trans(X)$ una transformación arbitraria. Como $X$ es infinito podemos expresarlo como la unión disjunta $X=Y\sqcup Z$, donde $|X|=|Y|=|Z|$. \\
Sea $\beta:X \rightarrow Y$ una biyección cualquiera entre $X$ y $Y$. Como $Im(\beta)=Y$ se tiene que $rank(\beta)=|X|$ y por tanto $\beta\in \mathcal{J}$.
\end{proof}

\begin{teorema}
Sea $X$ un conjunto infinito, sea $\mu\in Trans(X)$ una función inyectiva de defecto $def(\mu)=|X|$ y sea $\nu \in Trans(X)$ una función sobreyectiva con índice de contracción infinito $k(\nu)=|X|$. Entonces $\mathcal{J}\subset \langle Sym(X),\mu, \nu \rangle$.
\end{teorema}

\begin{proof}
Sea pues $\alpha \in \mathcal{J}$ y sea $\Lambda$ un conjunto de índices de cardinalidad $|X|$ y dividamos $\Lambda$ en 2 conjuntos ambos de cardinalidad $|X|$:
$$\Lambda=\Lambda_{1}\cup \Lambda_{2},\ \Lambda_{1} \cap \Lambda_{2}= \emptyset,\ |\Lambda_{1}|=|\Lambda_{2}|=|X|.$$
Como $k(\nu)=|X|$, podemos indexar los bloques de $ker(\nu)$ que tienen cardinal $|X|$ con el conjunto $\Lambda$. Entonces supongamos que $B_{\lambda}$, $\lambda\in \Lambda$, son dichos bloques. \\
Además, como $rank(\alpha)=|X|$ y $|\Lambda_{1}|=|X|$, se sigue que los bloques de $ker(\alpha)$ pueden ser indexados por $\Lambda_{1}$; sean $C_{\lambda}$, $\lambda \in \Lambda_{1}$, la colección de todos los bloques de $ker(\alpha)$. Definimos pues una permutación $\pi$ como sigue. Como $|B_{\lambda}|=|X|$, para cada $\lambda \in \Lambda_{1}$ existe una función inyectiva 
$$\pi_{\lambda}:\mu(C_{\lambda})\rightarrow B_{\lambda}.$$
Sea $$\overline{\pi}=\bigcup_{\lambda\in \Lambda_{1}}{\pi_{\lambda}},$$ 
notación personal
$$\overline{\pi}(x)=\left \{ \begin{array}{cc}  \pi_{\lambda}(x)&  x\in \mu(C_{\lambda})\\
\emptyset & \mbox{otro caso}\end{array} \right. ,$$
entonces es claro que $\overline{\pi}$ es una inyección parcial, por lo que tenemos que 
$$|X\setminus dom(\overline{\pi})|=|X\setminus im(\mu)|=def(\mu)= |X|.$$
Notemos también que $$im(\overline{\pi})\subseteq \bigcup_{\lambda\in \Lambda_{1}}{B_{\lambda}},$$
entonces tenemos que 
$$|X|=|\bigcup_{\lambda\in \Lambda_{2}}{B_{\lambda}}|\leq |X\setminus im(\overline{\pi})|.$$
Entonces $\overline{\pi}$ puede ser extendida a una biyección $\pi\in Sym(X)$.
$$\pi(x)=\left\{ \begin{array}{cc}\overline
{\pi}(x)& x\in dom(\overline{\pi})\\
x& \mbox{otro caso}\end{array} \right.=\left\{ \begin{array}{cc}\pi_{\lambda}(x)& x\in \mu(C_{\lambda})\\
x& \mbox{otro caso.}\end{array} \right.$$

Notemos pues que para cada $\lambda \in \Lambda$, se tiene que la imagen de cada bloque $B_{\lambda}$ bajo $\nu$, $\nu(B_{\lambda})$, es un conjunto unipuntual y podemos identificarlos con el elemento que contiene. Entoces ahora utilizamos el axioma de elección para seleccionar, para cada $\lambda\in \Lambda_{1}$, un elemento $z_{\lambda}$ en $\nu^{-1}(\alpha(C_{\lambda}))$. Notemos que $\nu^{-1}(\alpha(C_{\lambda}))$ es no vacío, debido a que $\nu$ es sobreyectiva. Entonces definimos una función $$\overline{\sigma}:\{\nu(B_{\lambda}): \lambda\in \Lambda_{1}\}\rightarrow X,$$
definida por:
$$\overline{\sigma}(\nu(B_{\lambda}))=z_{\lambda}\ (\lambda\in \Lambda_{1}).$$
Notación personal:
$$\overline{\sigma}: \nu(B_{\lambda})\rightarrow \alpha(C_{\lambda}).$$
Recordemos que $\Lambda=\Lambda_{1}\sqcup \Lambda_{2}$, es decir que si $\lambda\in \Lambda_{2}$, no está en $\Lambda_{1}$ y por tanto $B_{\lambda}$ no pertenece a $dom(\overline{\sigma})$ y por tanto:
$$|X\setminus dom(\overline{\sigma})|=|X|.$$

Notemos que $im(\overline{\sigma})$ contiene a lo mucho un elemento de cada uno de los bloques de $ker(\nu)$. Como $ker(\nu)$ tiene bloques de tamaño $|X|$, se sigue que $$|X\setminus im(\overline{\sigma})|=|X|,$$
y por tanto $\overline{\sigma}$ puede ser extendida a una biyección $\sigma\in Sym(X)$.

$$\sigma(x)=\left\{ \begin{array}{cc}\overline
{\sigma}(x)& x\in dom(\overline{\sigma})\\
x& \mbox{otro caso.}\end{array} \right.$$
Ahora afirmamos que $$\alpha=\nu \sigma \nu \pi \mu. $$
Sea $x\in X$ un elemento arbitrario. Entonces existe un único $\lambda \in \Lambda_{1}$ tal que $x\in C_{\lambda}$. Pero entonces $\mu(x)\in \mu(C_{\lambda})$ y así $\nu \pi \mu(x)\in \nu(\pi_{\lambda}\mu (C_{\lambda}))= \nu (B_{\lambda})$. Se sigue que $\sigma\nu\pi\mu(x)\in \nu^{-1}(\alpha(C_{\lambda}))$ y así 
$$\nu\sigma\nu\pi\mu(x)=\alpha(C_{\lambda})=\alpha(x).$$

\tikzset{every picture/.style={line width=0.75pt}} 

\begin{tikzpicture}[x=0.7pt,y=0.7pt,yscale=-1,xscale=1]

\draw   (21,107) .. controls (21,70.55) and (33.19,41) .. (48.22,41) .. controls (63.26,41) and (75.44,70.55) .. (75.44,107) .. controls (75.44,143.45) and (63.26,173) .. (48.22,173) .. controls (33.19,173) and (21,143.45) .. (21,107) -- cycle ;
\draw   (108.44,107) .. controls (108.44,70.55) and (120.63,41) .. (135.67,41) .. controls (150.7,41) and (162.89,70.55) .. (162.89,107) .. controls (162.89,143.45) and (150.7,173) .. (135.67,173) .. controls (120.63,173) and (108.44,143.45) .. (108.44,107) -- cycle ;
\draw   (197.89,107) .. controls (197.89,70.55) and (210.08,41) .. (225.11,41) .. controls (240.15,41) and (252.33,70.55) .. (252.33,107) .. controls (252.33,143.45) and (240.15,173) .. (225.11,173) .. controls (210.08,173) and (197.89,143.45) .. (197.89,107) -- cycle ;
\draw   (375.33,107) .. controls (375.33,70.55) and (387.52,41) .. (402.56,41) .. controls (417.59,41) and (429.78,70.55) .. (429.78,107) .. controls (429.78,143.45) and (417.59,173) .. (402.56,173) .. controls (387.52,173) and (375.33,143.45) .. (375.33,107) -- cycle ;
\draw   (552.78,109) .. controls (552.78,72.55) and (564.97,43) .. (580,43) .. controls (595.03,43) and (607.22,72.55) .. (607.22,109) .. controls (607.22,145.45) and (595.03,175) .. (580,175) .. controls (564.97,175) and (552.78,145.45) .. (552.78,109) -- cycle ;
\draw    (45,68) .. controls (70.57,46.44) and (108.54,47.93) .. (130.75,67.76) ;
\draw [shift={(132.09,69)}, rotate = 223.67000000000002] [color={rgb, 255:red, 0; green, 0; blue, 0 }  ][line width=0.75]    (10.93,-3.29) .. controls (6.95,-1.4) and (3.31,-0.3) .. (0,0) .. controls (3.31,0.3) and (6.95,1.4) .. (10.93,3.29)   ;
\draw    (136.09,69) .. controls (161.66,47.44) and (199.63,48.93) .. (221.84,68.76) ;
\draw [shift={(223.18,70)}, rotate = 223.67000000000002] [color={rgb, 255:red, 0; green, 0; blue, 0 }  ][line width=0.75]    (10.93,-3.29) .. controls (6.95,-1.4) and (3.31,-0.3) .. (0,0) .. controls (3.31,0.3) and (6.95,1.4) .. (10.93,3.29)   ;
\draw    (337.09,97) .. controls (351.79,70.81) and (367.45,70.28) .. (395.28,71.94) ;
\draw [shift={(397,72.05)}, rotate = 183.51] [color={rgb, 255:red, 0; green, 0; blue, 0 }  ][line width=0.75]    (10.93,-3.29) .. controls (6.95,-1.4) and (3.31,-0.3) .. (0,0) .. controls (3.31,0.3) and (6.95,1.4) .. (10.93,3.29)   ;
\draw   (290,97) -- (337.09,97) -- (337.09,121) -- (290,121) -- cycle ;
\draw   (467,96) -- (514.09,96) -- (514.09,120) -- (467,120) -- cycle ;
\draw  [fill={rgb, 255:red, 0; green, 0; blue, 0 }  ,fill opacity=1 ] (129.05,71) .. controls (129.05,68.77) and (130.86,66.95) .. (133.09,66.95) .. controls (135.33,66.95) and (137.14,68.77) .. (137.14,71) .. controls (137.14,73.23) and (135.33,75.05) .. (133.09,75.05) .. controls (130.86,75.05) and (129.05,73.23) .. (129.05,71) -- cycle ;
\draw  [fill={rgb, 255:red, 0; green, 0; blue, 0 }  ,fill opacity=1 ] (576,70.05) .. controls (576,67.81) and (577.81,66) .. (580.05,66) .. controls (582.28,66) and (584.09,67.81) .. (584.09,70.05) .. controls (584.09,72.28) and (582.28,74.09) .. (580.05,74.09) .. controls (577.81,74.09) and (576,72.28) .. (576,70.05) -- cycle ;
\draw  [fill={rgb, 255:red, 0; green, 0; blue, 0 }  ,fill opacity=1 ] (397,72.05) .. controls (397,69.81) and (398.81,68) .. (401.05,68) .. controls (403.28,68) and (405.09,69.81) .. (405.09,72.05) .. controls (405.09,74.28) and (403.28,76.09) .. (401.05,76.09) .. controls (398.81,76.09) and (397,74.28) .. (397,72.05) -- cycle ;
\draw  [fill={rgb, 255:red, 0; green, 0; blue, 0 }  ,fill opacity=1 ] (220.14,74.05) .. controls (220.14,71.81) and (221.95,70) .. (224.18,70) .. controls (226.42,70) and (228.23,71.81) .. (228.23,74.05) .. controls (228.23,76.28) and (226.42,78.09) .. (224.18,78.09) .. controls (221.95,78.09) and (220.14,76.28) .. (220.14,74.05) -- cycle ;
\draw  [fill={rgb, 255:red, 0; green, 0; blue, 0 }  ,fill opacity=1 ] (40.95,72.05) .. controls (40.95,69.81) and (42.77,68) .. (45,68) .. controls (47.23,68) and (49.05,69.81) .. (49.05,72.05) .. controls (49.05,74.28) and (47.23,76.09) .. (45,76.09) .. controls (42.77,76.09) and (40.95,74.28) .. (40.95,72.05) -- cycle ;
\draw    (235.09,44.27) -- (288.56,95.61) ;
\draw [shift={(290,97)}, rotate = 223.84] [color={rgb, 255:red, 0; green, 0; blue, 0 }  ][line width=0.75]    (10.93,-3.29) .. controls (6.95,-1.4) and (3.31,-0.3) .. (0,0) .. controls (3.31,0.3) and (6.95,1.4) .. (10.93,3.29)   ;
\draw    (232.09,171.27) -- (288.49,122.31) ;
\draw [shift={(290,121)}, rotate = 499.04] [color={rgb, 255:red, 0; green, 0; blue, 0 }  ][line width=0.75]    (10.93,-3.29) .. controls (6.95,-1.4) and (3.31,-0.3) .. (0,0) .. controls (3.31,0.3) and (6.95,1.4) .. (10.93,3.29)   ;
\draw    (412.09,44.27) -- (465.56,95.61) ;
\draw [shift={(467,97)}, rotate = 223.84] [color={rgb, 255:red, 0; green, 0; blue, 0 }  ][line width=0.75]    (10.93,-3.29) .. controls (6.95,-1.4) and (3.31,-0.3) .. (0,0) .. controls (3.31,0.3) and (6.95,1.4) .. (10.93,3.29)   ;
\draw    (409.09,171.27) -- (465.49,122.31) ;
\draw [shift={(467,121)}, rotate = 499.04] [color={rgb, 255:red, 0; green, 0; blue, 0 }  ][line width=0.75]    (10.93,-3.29) .. controls (6.95,-1.4) and (3.31,-0.3) .. (0,0) .. controls (3.31,0.3) and (6.95,1.4) .. (10.93,3.29)   ;
\draw    (570.83,171.93) -- (516.69,121.29) ;
\draw [shift={(515.23,119.92)}, rotate = 403.09000000000003] [color={rgb, 255:red, 0; green, 0; blue, 0 }  ][line width=0.75]    (10.93,-3.29) .. controls (6.95,-1.4) and (3.31,-0.3) .. (0,0) .. controls (3.31,0.3) and (6.95,1.4) .. (10.93,3.29)   ;
\draw    (572.17,44.9) -- (516.41,94.59) ;
\draw [shift={(514.92,95.93)}, rotate = 318.28999999999996] [color={rgb, 255:red, 0; green, 0; blue, 0 }  ][line width=0.75]    (10.93,-3.29) .. controls (6.95,-1.4) and (3.31,-0.3) .. (0,0) .. controls (3.31,0.3) and (6.95,1.4) .. (10.93,3.29)   ;

\draw (84,57) node [anchor=north west][inner sep=0.75pt]    {$\mu $};
\draw (173,58) node [anchor=north west][inner sep=0.75pt]    {$\pi $};
\draw (34,176) node [anchor=north west][inner sep=0.75pt]    {$C_{\lambda }$};
\draw (114,177) node [anchor=north west][inner sep=0.75pt]    {$\mu ( C_{\lambda })$};
\draw (214,178) node [anchor=north west][inner sep=0.75pt]    {$B_{\lambda }$};
\draw (264,54) node [anchor=north west][inner sep=0.75pt]    {$\nu $};
\draw (350,55) node [anchor=north west][inner sep=0.75pt]    {$\sigma $};
\draw (580.05,76.05) node [anchor=north west][inner sep=0.75pt]    {$x$};
\draw (29.05,76.05) node [anchor=north west][inner sep=0.75pt]    {$x$};
\draw (443,57) node [anchor=north west][inner sep=0.75pt]    {$\nu $};
\draw (523,57) node [anchor=north west][inner sep=0.75pt]    {$\alpha $};
\draw (390,77) node [anchor=north west][inner sep=0.75pt]    {$z_{\lambda }$};
\draw (362,175) node [anchor=north west][inner sep=0.75pt]    {$\nu ^{-1}( \alpha ( C_{\lambda }))$};
\draw (470,99) node [anchor=north west][inner sep=0.75pt]    {$\alpha ( C_{\lambda })$};
\draw (293,100) node [anchor=north west][inner sep=0.75pt]    {$\nu ( B_{\lambda })$};
\draw (570.83,175.93) node [anchor=north west][inner sep=0.75pt]    {$C_{\lambda }$};
\draw (261,229) node [anchor=north west][inner sep=0.75pt]    {$\alpha ( C_{\lambda }) =\alpha ( x)$};

\end{tikzpicture}
De aquí se sigue que $Trans(X)=\langle Sym(X),\mu,\nu\rangle$ y así se sigue que $rank(Trans(X):Sym(X))=2$.
\end{proof}

\newpage

\section{Las estructuras de $A^{G}$}
El espacio de configuraciones $A^{G}$ puede dotarse con diversas estructuras. En el estudio de los autómatas celulares, lo cual motivó este proyecto, es necesario dotar de una topología al espacio de configuraciones y estudiar sus propiedades. En esta sección se define y se estudia la topología de $A^{G}$. 
\subsection{La topología prodiscreta de $A^{G}$}
\subsubsection{La topología producto}

\begin{definicion}{(Topología producto)}\label{producto}

Sea $\{X_{\lambda}\}_{\lambda \in I}$ una familia de espacios topológicos indexados por un conjunto $I$, la topología básica del producto cartesiano $X=\prod_{\lambda \in I}{X_{\lambda}}$ asociada a los mapeos proyección $\pi_{\lambda} : X \rightarrow X_{\lambda}$ es llamada la \emph{topología producto} de $X$. 
\end{definicion}

 Ésta es la topología más pequeña para la cual las proyecciones son continuas. Definimos esta topología por su subbase, la cual es la familia 
$$\{\pi^{-1}_{\lambda}(U_{\lambda}): U_{\lambda} \mbox{ es un abierto en } X_{\lambda}\}.$$
En el caso donde cada $X_{\lambda}$ está dotada con su topología discreta, la topología producto en $X=\prod_{\lambda \in I}{X_{\lambda}}$ recibe el nombre de \emph{topología prodiscreta}.

\begin{proposicion}\label{hausdorff}
Sea $\{X_{\lambda}\}_{\lambda \in I}$ una familia de espacios topológicos Hausdorff. Entonces $X=\prod_{\lambda \in I}{X_{\lambda}}$ es Hausdorff para la topología producto.
\end{proposicion}

\begin{proof}
Sean $x=(x_{\lambda})$ y $y=(y_{\lambda})$ puntos distintos de $X$. Entonces existe un $\lambda_{0} \in I$ para el cual $x_{\lambda_{0}}\neq y_{\lambda_{0}}$. Como $X_{\lambda_{0}}$ es Hausdorff, podemos encontrar conjuntos abiertos $U, V \subseteq X_{\lambda_{0}}$ tales que $U\cap V =\emptyset$, que contienen a $x_{\lambda_{0}}$ y a $y_{\lambda_{0}}$ respectivamente. Entonces las preimágenes de $U$ y $V$ bajo el mapeo proyección $\pi_{\lambda_{0}}:\prod_{\lambda \in I}{X_{\lambda}} \rightarrow X_{\lambda_{0}}$ son conjuntos disjuntos y abiertos en $X$ que contienen a $x$ y $y$. Por lo tanto $X$ es un espacio topológico de Hausdorff.
\end{proof}

\begin{teorema}{(Teorema de Tychonoff)} \label{tychonoff}

Sea $(X_{\lambda})_{\lambda \in I}$ una familia de espacios topológicos. Entonces el espacio producto $X=\prod_{\lambda \in I}X_{\lambda}$ es un espacio topológico compacto si y solo si cada espacio topológico $X_{\lambda}$ es compacto.
\end{teorema}

Para una demostración detallada puede consultar \cite[capítulo 5]{Mun}.\\

\begin{proposicion}
Sea $\{X_{\lambda}\}_{\lambda \in I}$ una familia de espacios topológicos. Supongamos que $F_{\lambda}$ es un conjunto cerrado de $X_{\lambda}$ para cada $\lambda \in I$. Entonces el conjunto $$F=\prod_{\lambda \in I}{F_{\lambda}}$$ es un conjunto cerrado de $X=\prod_{\lambda \in I}{X_{\lambda}}$ para la topología producto.
\end{proposicion}

\begin{proof}
Bastará con demostrar que $$F=\bigcap_{\lambda \in I} \pi^{-1}_{\lambda}(F_{\lambda}).$$

Sea $x=(x_{\lambda}) \in F$. Por definición $x_{\lambda} \in F_{\lambda}$. \\
Debemos ver que $$x\in \pi^{-1}_{\lambda}(\{x_{\lambda}\}), \hspace{0.1in} \forall \lambda \in I,$$ entonces $$x \in \bigcap_{\lambda \in I}{\pi^{-1}_{\lambda}(\{x_{\lambda}\})}.$$
Luego tenemos que\\
\begin{center} 
$ \begin{array}{c}
x_{\lambda} \in F_{\lambda}, \hspace{0.1in}\forall \lambda \in I,\\
 \pi^{-1}_{\lambda}(\{x_{\lambda}\}) \subseteq \pi^{-1}_{\lambda}(F_{\lambda}),\\
\bigcap_{\lambda \in I} \pi^{-1}_{\lambda}(\{x_{\lambda}\}) \subseteq \bigcap_{\lambda \in I} \pi^{-1}_{\lambda}(F_{\lambda}).\end{array}$
\end{center}
Así $$x \in  \bigcap_{\lambda \in I} \pi^{-1}_{\lambda}(F_{\lambda})$$ 
y tenemos que \\
$$\prod_{\lambda \in I}{F_{\lambda}}\subseteq \bigcap_{\lambda \in I}{\pi^{-1}_{\lambda}}{F_{\lambda}}.$$
Ahora sea $$x \in \bigcap_{\lambda \in I}{\pi^{-1}_{\lambda}(F_{\lambda})},$$
tenemos que \\
\begin{center}
$\begin{array}{c}
  x \in \pi^{-1}_{\lambda}(F_{\lambda}), \hspace{0.1in} \forall \lambda \in I, \\
\pi_{\lambda}(x)=x_{\lambda} \in F_{\lambda}, \hspace{0.1in} \\
 x \in \prod_{\lambda \in I}{F_{\lambda}}.
\end{array}$
\end{center}
Lo cual implica que\\  
$$ F=\bigcap_{\lambda \in I} \pi^{-1}(F_{\lambda}).$$\\
Entonces, como $\pi^{-1}(F_{\lambda})$ es un cerrado en $X$, $\forall \lambda \in I$, $F$ es la intersección arbitraria de conjuntos cerrados en $X$ y por lo tanto es cerrado en la topología producto.

\end{proof}

\subsubsection{Caracterización de la topología producto}

\begin{teorema}{(Topología producto)}\label{tproducto}

Sea $\{X_{\lambda}\}_{\lambda \in I}$ una familia de espacios topológicos indexados por el conjunto $I$, dado el producto cartesiano $X=\prod_{\lambda \in I}{X_{\lambda}}$ y las proyecciones $\pi_{\lambda}: X \rightarrow X_{\lambda}$. Las siguientes afirmaciones son equivalentes y definen la topología producto de $X$:
\begin{enumerate}
\item[i)] $\tau$ es la topología producto.
\item[ii)] $\tau$ tiene como base la familia $$\left\{\bigcap_{k=1}^{n}\pi^{-1}_{\lambda_{k}}(U_{\lambda_{k}}): U_{\lambda_{k}} \mbox{ es abierto en } X_{\lambda}\right\}.$$
\item[iii)] $\tau$ es la topología en la cual los abiertos son de la forma $$U=\prod_{\lambda\in I}{U_{\lambda}},$$
donde $U_{\lambda}\subset X_{\lambda}$ es abierto y además $U_{\lambda} \neq X_{\lambda}$ solo para un número finito de índices $\lambda \in I$. 
\item[iv)] $\tau$ es la topología más pequeña en la cual todas las proyecciones son continuas.
\end{enumerate}
\end{teorema}

Estas se siguen directamente de las definiciones de base, subbase y topología generada por éstas. Recordemos que la intersección finita de elementos de una subbase genera la base para la topología y que los abiertos en una topología generada por una base son las uniones arbitrarias de elementos de ésta. Además, que sea la topología más pequeña solo nos dice que en cualquier otra topología en la cual las funciones $\pi_{\lambda}$ son continuas, la topología producto debe estar contenida en ésta, es decir que los conjuntos de la forma $\pi^{-1}_{\lambda}(U_{\lambda})$ están contenidos en ésta otra topología.\\

Sea $G$ un grupo y $A$ un conjunto, al que nos referiremos como alfabeto, es de particular interés el espacio de configuraciones, esto es, el conjunto de las funciones $x: G \rightarrow A$. Entonces definimos el espacio de configuraciones como $$A^{G}:=\{\alpha : G \rightarrow A\}.$$
Toda función en $A^{G}$ puede representarse como una túpla de elementos de $A$ como se muestra a continuación,

$$x \in A^{G}, \hspace{0.2in} x=(...x(g_{0}),x(g_{1}),x(g_{2}),...), \hspace{0.2in} g_{i}\in G,$$
notando que $x(g_{i})$ son elementos del alfabeto.\\

En particular podemos ver un elemento de la manera $x=(...a_{0},a_{1},a_{2},...)$ donde es fácil ver que $x(g_{i})=a_{i}$.\\
Esta representación funciona cuando $G$ es un conjunto contable, no lo generalizaremos, solo utilizamos esta notación para generar un mejor entendimiento del concepto de autómata celular.

\begin{lema}
Cualquier conjunto unipuntual no es abierto.
\end{lema}

\begin{proof}
Sabemos que lo abiertos en la topología producto son de la forma $$U=\prod_{\lambda \in \ I} U_{\lambda},$$ donde $U_{\lambda} \subset X_{\lambda}$ es abierto y $U_{\lambda} \neq X_{\lambda}$ para un número finito de índices $\lambda \in I$.\\

Es por eso que dada una configuración $x=(a_{\lambda})_{\lambda\in I} \in A^{G}$, el conjunto unipuntual ${x}$ no es un conjunto abierto.\\

Por ejemplo en $A^{\mathbb{Z}}$, siendo $A=\{0,1\}$, la configuración $x=(...0,0,0,0,...)$ es una configuración para la cual cada entrada está fija con el elemento $0$.\\
Podemos ver entonces que ${x}\in \Pi_{n}^{-1}\{0\}$ $\forall n \in \mathbb{Z}$. Entonces $${x}=\pi_{0}^{-1}\{0\}\cap\pi_{1}^{-1}\{0\}\cap\pi_{-1}^{-1}\{0\}\cap\pi_{2}^{-1}\{0\}\cap\pi_{-2}^{-1}\{0\}\cap...=\bigcap_{n\in \mathbb{Z}}\pi_{n}^{-1}\{0\},$$  pero esto no es la intersección finita de elementos $\Pi_{n}^{-1}\{0\}$ y no puede ser descrito como tal. Entonces el conjunto unipuntual ${x}$ no puede ser un conjunto abierto.
\end{proof}

Este ejemplo nos genera una deducción muy importante. La topología prodiscreta es la topología obtenida del producto de espacios equipados con la topología discreta. Y recordando que la topología discreta es aquella en la que todos los conjuntos son abiertos y también cerrados. Este ejemplo nos muestra que la topología prodiscreta para $A^{G}$ no es precisamente la topología discreta.

\subsubsection{La topología prodiscreta de $A^{G}$}

En esta sección visualizaremos la forma de la topología prodiscreta para $A^{G}$ y además mostraremos algunos resultados sencillos sobre esta topología.\\

Recordando que el Teorema \ref{tproducto} nos genera definiciones equivalentes para la topología producto, entonces podemos reestructurar la topología del espacio de configuraciones $A^{G}$.\\

Por la definición de la topología prodiscreta, tenemos que su subbase está dada por $$\{\pi^{-1}_{\lambda}(U_{\lambda}): U_{\lambda} \mbox{ es un abierto en } X_{\lambda}\},$$
pero sabemos que en $A^{G}$ se tiene que  $X_{\lambda}=A$ con la topología discreta para todo $\lambda$. Entonces la subbase está dada por $$\{\pi^{-1}_{\lambda}(U):U\subseteq A\}.$$
Recordemos que una configuración está dada como $$x=(...,x(g_{\lambda}),...),$$ donde indexamos cada elemento de $G$ como $g_{\lambda}$. Entonces si $x\in \pi^{-1}_{g}(U)$, nos referimos a las configuraciones tales que $x(g)\in U$.\\
Por otro lado, un abierto $U$ de $A$ no es otra cosa que un conjunto finito de elementos de $A$. Entonces para un abierto $U=\{a_{1},...,a_{n}\}$ podemos reescribirlo como $$U=\bigcup_{k=1}^{n}\{a_{k}\}.$$
Es decir que la familia de conjuntos unipuntuales $\{a_{k}\}$ genera a cualquier abierto en $A$.  Y por consiguiente la familia  $\{\pi_{k}^{-1}(\{a_{k}\})\}$ genera la topología prodiscreta en $A^{G}$. Como $a_{k}=x(g_{k})$ para algún $x\in A^{G}$ y algún $g_{k}\in G$, podemos escribir la subbase para $A^{G}$ como la familia $$\{\pi_{g}^{-1}(x):x\in A^{G}, g\in G\}$$
más aún, los abiertos son las uniones de intersecciones finitas de elementos de esta colección.
$$U=\bigcup_{\alpha}\left(\bigcap_{i=1}^{k}\pi^{-1}_{g_{i}}(x_{\alpha})\right).$$

\begin{proposicion}{(La estructura de $A^{G}$)}

Sea $G$ un grupo y $A$ el conjunto al que llamamos alfabeto.  El espacio de configuraciones $A^{G}$ es Hausdorff y completamente desconectado.
\end{proposicion}

\begin{proof}
Como lo prueba la Proposición \ref{hausdorff}, el producto de espacios de Hausdorff es Hausdorff. También el producto de espacios completamente desconectados es completamente desconectado. Para verificar este hecho ver \cite{Mun}.
\end{proof}

\begin{lema}
Sea $G$ un grupo finito, entonces la topología prodiscreta para $A^{G}$ es precisamente la topología discreta de éste.
\end{lema}

\subsubsection{La métrica en $A^{G}$}
En este apartado mostraremos que es posible definir una métrica para el espacio de configuraciones $A^{G}$ y no solo que sea posible definirla, sino que además la topología generada por dicha métrica es precisamente la topología prodiscreta de $A^{G}$. Gran parte del material que se necesita como base para el desarrollo de esta sección se encuentra en materiales como \cite{krey}, \cite{Lip} y \cite{Mun}.\\

\begin{definicion}{(Métrica)}

Sea $X$ un conjunto. Una métrica para $X$ es una función $d:X\times X \rightarrow \mathbb{R}$ que satisface las siguientes propiedades:
\begin{enumerate}
\item[i)] $d(x,y)>0$, $\forall x,y, \in X$.
\item[ii)] $d(x,y)=0\ \Leftrightarrow \  x=y$.
\item[iii)] $d(x,y)=d(y,x)$, $\forall x,y\in X$.
\item[iv)] $d(x,z)\leq d(x,y)+d(y,z)$, $\forall x,y,z \in X$.
\end{enumerate}
\end{definicion}

Llamamos \emph{bola abierta} con centro en $x$ y radio $r$ a el conjunto definido a continuación:  $$B(x,r):= \{y\in X:\ d(x,y)<r \}.$$

\begin{proposicion}
Dado un punto $x$ en un espacio métrico $M$, sean $r_{1}<	r_{2}$, entonces $$B(x,r_{1})\subset B(x,r_{2}).$$
\end{proposicion}

\begin{proof}
Dado un punto $y\in B(x,r_{1})$ se tiene que $d(x,y)<	r_{1}<	r_{2}$, por lo tanto $y\in B(x,r_{2})$.

\end{proof}
\begin{teorema}{(Topología generada por una métrica)}

Sea $X$ un conjunto con una métrica $d$. Entonces la familia $$\tau_{d}:=\{U\subseteq X: \ \forall x\in U, \  \exists r>0 : \ B(x,r)\subset U\}$$ es una topología para $X$.
\end{teorema}

\begin{proof}
El conjunto vacío y todo el espacio $X$ pertenecen trivialmente a la familia $\tau_{d}$. Ahora, consideremos una colección arbitraria $\{U_{\alpha}\}$ de conjuntos en $\tau_{d}$. Entonces para $x\in \bigcup_{\alpha}U_{\alpha}$, sabemos que $x\in U_{\alpha}$ para algún $\alpha$. 
Esto significa que existe $r>0$ tal que $B(x,r) \subset U_{\alpha}$ y por transitividad tenemos que $$B(x,r)\subset \bigcup_{\alpha}U_{\alpha}. $$

Ahora dados dos conjuntos $U_{1},U_{2}\in \tau_{d}$ y dado $x\in U_{1}\cap U_{2}$ tenemos que $\exists r_{1},r_{2}>0$ tales que $B(x,r_{1})\subset U_{1}$ y $B(x,r_{2})\subset U_{2})$. Podemos suponer sin pérdida de generalidad que $r_{1}<r_{2}$. Entonces, por la proposición anterior, tenemos que $$B(x,r_{1})\subset B(x,r_{2})\subset U_{2}.$$
Y finalmente vemos que $B(x,r_{1})\subset U_{1}\cap U_{2}$.
\end{proof}

Definimos pues una métrica para $A^{G}$.\\

Sean $A$ un conjunto y $G$ un grupo contable. Tomamos una familia de subconjuntos finitos $\{B_{n}\}_{n\in \mathbb{N}}$ de $G$ con la condición de que $$\emptyset=B_{0} \subseteq B_{1} \subseteq B_{2} \subseteq \ldots \subseteq B_{n} \subseteq \ldots $$ y además que cumple que $$\bigcup_{n=0}^{\infty} B_{n} = G.$$

Entonces para dos configuraciones $x,y\in A^{G}$ definimos una métrica $d:A^{G}\times A^{G} \rightarrow \mathbb{R}$:

 $$d(x,y)= \left\{ \begin{array}{cc} 0 & x=y \\ 2^{-k} & x\neq y \end{array} \right. ,$$
donde $k=max\{n\in \mathbb{N}: x|_{B_{n}}=y|_{B_{n}}\}$.\\

Observemos que $k$ es una cantidad entera que describe el punto más \emph{grande} dentro de la cadena de elementos de $\{B_{n}\}$ en donde $x$ y $y$ coinciden totalmente.\\

Únicamente para el siguiente lema y la Proposición \ref{metrica}, utilizaremos la siguiente notación. Sean $x,y\in A^{G}$ dos configuraciones, entonces denotaremos por $$n_{xy}=max\{n\in \mathbb{N}: x|_{B_{n}}=y|_{B_{n}}\}.$$

\begin{lema}\label{superlema}
Dadas tres configuraciones $x,y,z\in A^{G}$. Entonces una de las siguientes tres afirmaciones es verdadera:
\begin{enumerate}
\item[i)] $n_{xy}=n_{xz}\leq n_{yz}$.
\item[ii)] $n_{xy}=n_{yz}\leq n_{xz}$.
\item[iii)] $n_{yz}=n_{xz}\leq n_{xy}$.
\end{enumerate}
\end{lema}
\begin{proof}
Bastará con verificar uno de los casos, pues los otros 2 son casos completamente análogos, ordenando de manera adeacuada las configuraciones  $x,$, $y$ y $z$.\\
Supongamos pues que $$n_{xy}\neq n_{xz}> n_{yz}.$$
Sin perder generalidad, una de las dos cantidades de la izquierda debe ser estrictamente mayor que la otra. Supongamos pues que $n_{xy} >n_{xz}>n_{yz}$. Entonces existe $g \in B_{n_{xz}}$ tal que $x(g)=z(g)$. Sin embargo, $g\notin B_{n_{yz}}$, lo cual implica que $y(g) \neq z(g)$. Por otro lado $g\in B_{n_{xy}}$ y esto significa que $x(g)=y(g)$, lo cual es una contradicción. Por lo tanto lo supuesto no es cierto y se tiene que $$n_{xy}=n_{xz}\leq n_{yz}.$$
\end{proof}

\begin{proposicion}{(Métrica)}\label{metrica}

La función $d:A^{G}\times A^{G} \rightarrow [0,1]$, es precisamente una métrica para $A^{G}$.
\end{proposicion}

\begin{proof}
Es fácil ver que con $k\in \mathbb{Z}$, $2^{-k}$ es una cantidad positiva y por la manera en la que se define $d$ se cumple que $d(x,y)=0$ $\Longleftrightarrow$ $x=y$.\\
Además, si $x\neq y$ se tiene que $$\begin{array}{cl} d(x,y)&=2^{-max\{n\in \mathbb{N}: x|_{B_{n}}=y|_{B_{n}}\}}\\&=2^{-max\{n\in \mathbb{N}: y|_{B_{n}}=x|_{B_{n}}\}}\\ &=d(y,x),\end{array} $$ por lo tanto es simétrica.\\

Ahora para tres configuraciones $x,y,z \in A^{G}$ tenemos entonces por el Lema \ref{superlema} que una de las tres cantidades $n_{xy},n_{xz},n_{yz}$ debe de ser mayor o por lo menos igual a las otras y que las otras dos cantidades deben de ser iguales. Supondremos pues que se cumple con $n_{xy}=n_{xz}\leq n_{yz}$. Podemos entonces visualizar dos casos. El primero en el que las tres cantidades son iguales, i.e. $n_{xy}=n_{xz}=n_{yz}$. Y el segundo, en el cual una cantidad es estrictamente mayor. Es obvio también que $d(x,y)=2^{-n_{xy}}$. Si las tres cantidades son iguales entonces la desigualdad $d(x,z)\leq d(x,y)+d(y,z)$ se cumple trivialmente, sin importar el orden de las configuraciones $x,y$ o $z$. Ahora si una de las cantidades es estrictamente mayor, supongamos $n_{xy}=n_{xz}< n_{yz}$, entonces es fácil ver que, sin importar la combinación, se cumple la desigualdad triangular. Para los casos donde $d(x,y)$ o $d(x,z)$ están del lado izquierdo de la desigualdad y como $n_{xy}=n_{xz}$ implica directamente que $d(x,y)=d(x,z)$ y al ser $d(y,z)\geq 0$. Tenemos entonces que $d(x,y)\leq d(x,z)+d(y,z)$ y también que $d(x,z)\leq d(x,y)+d(y,z)$. Para el caso en el que $d(y,z)$ está en la parte izquierda de la desigualdad debemos visualizar lo siguiente:

$$\begin{array}{l}
n_{xy}=n_{xz}\leq n_{yz}\\
\rightarrow-n_{xy}=-n_{xz}\geq -n_{yz},
\end{array}$$
entonces se cumplen que
$$2^{-n_{xy}}\geq 2^{-n_{yz}}$$
y también que
$$2^{-n_{xz}}\geq 2^{-n_{yz}},$$
o dicho de otra forma
$$\begin{array}{l}
d(x,y)\geq d(y,z)\\
d(x,z)\geq d(y,z)
\end{array}$$
$$\Rightarrow d(x,y)+d(x,z)\geq d(y,z).$$

Por todo lo anterior, podemos concluir entonces que la función $d$ es una métrica para el espacio de configuraciones $A^{G}$.
\end{proof}

\begin{definicion}

Dos métricas $d_{1}$ y $d_{2}$ se dicen \emph{fuertemente equivalentes} si existen cantidades $\alpha, \beta >0$ tales que $$\alpha d_{1}(x,y)\leq d_{2}(x,y) \leq \beta d_{1}(x,y) \ \forall x,y\in X.$$
\end{definicion}

\begin{proposicion}\label{equivalentes}
Si dos métricas $d_{1}$ y $d_{2}$ son fuertemente equivalentes, entonces generan la misma topología.
\end{proposicion}

La demostración de esta Proposición no se incluye en este documento pues no empata directamente con sus objetivos. Si se quiere ver mas a detalle las definiciones y herramientas necesarias para la demostración de esta proposición pueden encontrarse en materiales como \cite{krey}.

\begin{teorema}{(Topología prodiscreta)}\label{metricaprodiscreta}

La topología generada por la métrica $d$, es precisamente la topología prodiscreta de $A^{G}$.
\end{teorema}

\begin{proof}
Recordemos que en la topología generada por una métrica, los abiertos son aquellos conjuntos en los que para cada punto en ellos es posible determinar un radio para el cual las bolas abiertas $$B(x,r):=\{y\in A^{G}: d(x,y)<r \}$$
estén completamente contenidas en el abierto. Es decir $U\subseteq A^{G}$ es un conjunto abierto si y solo si $\forall x \in U$  existe$ r >0$ tal que $B(x,r)\subset U$.\\

Para probar este teorema bastará con mostrar que las bolas abiertas, las cuales son base para la topología generada por la métrica $d$, son abiertos en la topología prodiscreta y de igual manera, que los cilindros básicos $\Pi_{\alpha}^{-1}(\{a\})$, con $a\in A$,  son abiertos para la topología generada por $d$.\\

Dada una bola abierta 
$$\begin{array}{rcl}
B(x,r)&:=&\{y\in A^{G}: d(x,y)<r\}\\
&=&\{y\in A^{G}: 2^{-k}<r\}\\
&=&\{y\in A^{G}: k>log_{2}\left(\frac{1}{r}\right)\}
\end{array}$$
donde $k=max\{n\in \mathbb{N}:\  x|_{B_{n}}=y|_{B_{n}}\}$, entonces podemos describir las bolas abiertas de la siguiente manera:
$$B(x,r):=\{y\in A^{G}: x|_{B_{k}}=y|_{B_{k}}\}.$$
Esto significa que $x(g)=y(g)$, $\forall g\in B_{k}$. Recordando que hay un número finito de elementos $g\in B_{k}$. \\
Dicho con palabras más simples, los puntos en $B(x,r)$ están en biyección con todas las configuraciones $y\in A^{G}$ tales que $x(g)=y(g)$ $\forall g\in B_{k}$. Así $$B(x,r)=\bigcap_{g\in B_{k}} \pi_{g}^{-1}(\{x(g)\})$$
notando que $x(g)\in A$ y además $k$ es el mínimo entero tal que $k>log_{2}\left(\frac{1}{r}\right)$. Al ser esto una intersección finita de elementos de la forma $\pi_{\alpha}^{-1}(\{a\})$, tenemos pues que $B(x,r)$ es un conjunto abierto en la topología producto (la topología prodiscreta).\\

Ahora bien, dado un abierto en la topología prodiscreta:
$$U=\bigcup_{\beta\in B}(\bigcap_{\alpha}(\pi_{\alpha}^{-1}(a_{\alpha \beta})) ).$$
Basta con verificar para cada $\pi_{\alpha}^{-1}(\{a\})$. Entonces. sea $\pi_{\alpha}^{-1}(\{a\})$ con $\alpha \in G$, $a\in A$. Dada una configuración $y\in \pi_{\alpha}^{-1}(\{a\})$ se tiene que $y(\alpha)=a$. Sea $k=min\{n\in \mathbb{N}: \alpha \in B_{n}\}$ determinamos $r= 2^{-k}$. Entonces dada una configuración $z\in B(y,r)$, ésta cumple con $$d(y,z)< 2^{-max\{n\in \mathbb{N}: y|_{B_{n}}=z|_{B_{n}}\}}<2^{-min\{n: \alpha \in B_{n}\}}i$$
Esto implica que $$max\{n\in \mathbb{N}: y|_{B_{n}}=z|_{B_{n}}\} > min\{n\in \mathbb{N}: \alpha \in B_{n}\},$$
siendo $m=max\{n\in \mathbb{N}: y|_{B_{n}}=z|_{B_{n}}\}$, esto nos dice que $$B_{0}\subseteq B_{1} \subseteq ... \subseteq B_{k}\subseteq ... \subseteq B_{m} \subseteq ...$$
lo cual implica que $\alpha \in B_{m}$. Como $y|_{B_{m}}=z|_{B_{m}}$, tenemos que $y(\alpha)=z(\alpha)=a$.\\
Entonces $z\in \pi_{\alpha}^{-1}(\{a\})$. Y luego $$y\in B(x,r) \subset \pi_{\alpha}^{-1}(\{a\}).$$
Esto prueba que $\pi_{\alpha}^{-1}(\{a\})$ es un abierto en la topología generada por la métrica $d$.
\end{proof}

\begin{ejemplo}
Dado un grupo contable $G$ y un conjunto $A$, definimos las siguientes familias de conjuntos $\{B_{n}\}_{n\in \mathbb{N}}$ y $\{C_{n}\}_{n\in \mathbb{N}}$:
$$    \begin{array}{cc}
\begin{array}{ccl}
B_{0}&=&\emptyset\\
B_{1}&=&\{g_{1}\}\\
B_{2}&=&\{g_{1},g_{2}\}\\
 & \vdots&  \\
B_{i}&=&\{g_{1},g_{2},...,g_{i}\}\\
 & \vdots&  \\
\end{array} &

\begin{array}{c}
C_{0}=\emptyset \\
C_{i}=G, \ \forall i \geq 1.
\end{array}

\end{array}$$
\end{ejemplo}

Sea $d_{1}$ la métrica inducida por la familia $\{B_{n}\}_{n\in \mathbb{N}}$ y $d_{2}$ la inducida por la familia $\{C_{n}\}_{n\in \mathbb{N}}$. 
Debemos visualizar que para dos configuraciones $x \neq y \in A^{G}$, el máximo índice en el cual coinciden en la familia $\{C_{n}\}_{n\in \mathbb{N}}$ es cero. Por lo tanto $d_{2}(x,y)=1$. Entonces, para una configuración $x\in A^{G}$ y para cualquier $\alpha >0$, es posible determinar una configuración $y\in A^{G}$ tal que cumpla con $$\begin{array}{c}\alpha d_{2}(x,y) > d_{1}(x,y)= \frac{1}{2^{n_{y}}}\\
\alpha > d_{1}(x,y).\end{array}$$
Así las métricas $d_{1}$ y $d_{2}$ no son fuertemente equivalentes. \\

Es aquí donde analizamos al grupo $G$. Si $G$ es un grupo de cardinal infinito, la familia $\{C_{n}\}_{n\in \mathbb{N}}$ definida en este ejemplo no cumple con las características para definir la métrica que queremos para $A^{G}$, pues los $C_{n}$ deben de ser finitos. Sin embargo, si $G$ fuera infinito, la métrica $d_{2}$ no estaría bien definida. Además, volvemos a los casos triviales donde en $A^{G}$ la topología prodiscreta es la topología discreta.\\

Este ejemplo busca mostrar un poco más a detalle cómo funciona la métrica que acabamos de definir.
Una vez visualizada la métrica definida por una familia $\{B_{n}\}_{n\in \mathbb{N}}$, por el Teorema \ref{metricaprodiscreta}, cualquier familia $\{C_{n}\}_{n\in \mathbb{N}}$ distinta, pero bien definida, genera también la topología prodiscreta para $A^{G}$. También se debe ver que estas métricas no serían fuertemente equivalentes. Esto sería el contraejemplo que muestra como es que el recíproco de la Proposición \ref{equivalentes} no es cierto, pues al generar la misma topología son métricas topológicamente equivalentes, pero no son fuertemente equivalentes.\\

Recordemos que en un espacio topológico $(X,\tau)$ la clausura de un subconjunto $A$ también se define de la siguiente manera:
$$\overline{A}:= \{x\in X|\  \forall U \in \tau, x\in U: A \cap U \neq \emptyset \}.$$

\begin{ejemplo}\label{ejemploi}
\end{ejemplo}
Sea $\mathbb{R}^{\mathbb{Z}}$, con su topología prodiscreta. Un elemento $x\in \mathbb{R}^{\mathbb{Z}}$ es de la forma $x=(...,x_{-1},x_{0},x_{1},...)$. Tomemos la función $f:\mathbb{R}^{\mathbb{Z}} \rightarrow \mathbb{R}^{\mathbb{Z}}$ de tal manera que $f(x)(n)=x(n+1)-x(n)^{2}$ $\forall n \in \mathbb{Z}$. Entonces, el conjunto $f(\mathbb{R}^{\mathbb{Z}})$ es denso en $\mathbb{R}^{\mathbb{Z}}$.\\

\begin{proof}
Entonces queremos mostrar que $\overline{f(\mathbb{R}^{\mathbb{Z}})}=\mathbb{R}^{\mathbb{Z}}$. Pero la primer contención es automática, $\overline{f(\mathbb{R}^{\mathbb{Z}})}\subseteq \mathbb{R}^{\mathbb{Z}}$.\\

Ahora queremos mostrar que $\mathbb{R}^{\mathbb{Z}} \subseteq \overline{f(\mathbb{R}^{\mathbb{Z}})}$.
Y recordemos que $$\overline{f(\mathbb{R}^{\mathbb{Z}})}:=\{x\in \mathbb{R}^{\mathbb{Z}} |\  \forall U\subset \mathbb{R}^{\mathbb{Z}} : x\in U, \Rightarrow  U \cap f(\mathbb{R}^{\mathbb{Z}}) \neq \emptyset\}.$$
Sea $y\in \mathbb{R}^{\mathbb{Z}}$. Mostraremos que $y\in \overline{f(\mathbb{R}^{\mathbb{Z}})}$. Sea $U$ un abierto que contiene a $y$. $U$ tiene la forma $U=\bigcup(\bigcap\pi^{-1}_{i\in F}(\{a\}))$ donde $a$ es un elemento de $\mathbb{R}$. Basta con tomar un abierto de la forma $U=\bigcap\pi^{-1}_{i\in F}(\{a\})$, donde $F\subset \mathbb{Z}$ es un conjunto finito. Seleccionemos $m\in \mathbb{Z}$ de tal manera que $F\subset [m, \infty)$. Tomemos pues la configuración $x_{F} \in \mathbb{R}^{\mathbb{Z}}$ de tal manera que $x_{F}(n)=0$, $\forall n\leq m$ y $x_{F}(n+1)=y(n)+x_{F}(n)^{2}$, $\forall n \geq m $. Esto significa que $f(x_{F})(n)=y(n)$ $\forall n \geq m$, entonces las configuraciones $f(x_{F})$ y $y$ coinciden en $[m,\infty)$ y por lo tanto en $F$.\\
Esto significa que $f(x_{F})\in U$. O dicho en otras palabras $U\cap f(\mathbb{R}^{\mathbb{Z}})\neq \emptyset$ y por lo tanto $y$ está en la clausura de $f(\mathbb{R}^{\mathbb{Z}})$.\\
\end{proof}

\subsection{La estructura prodiscreta uniforme}
Como se vio en la sección anterior la topología prodiscreta no es siempre sencilla de imaginar, pues tiene sus propias características que se vuelven un tanto complicada. El Teorema de Curtis-Hedlind es un resultado importante el cual depende de la continuidad de una función para definirla como autómata celular, sin embargo esto no siempre es posible, nuevamente a causa de su topología, es por eso que es necesario encontrar una estructura donde las funciones sean ''continuas'' a pesar de que la topología no lo permita. En esta sección definimos la ''estructura uniforme'' la cual nos permite obtener una generalización del Teorema de Curtis-Hedlund. \\

Sea $X$ un conjunto. Denotamos por $\Delta_{X}$ la diagonal en $X\times X$, i.e., $$\Delta_{X}:=\{(x,x)\in X\times X:\ x\in X \}.$$
Supongamos que $R$ es un subconjunto de $X\times X$ (podemos decir que $R$ es una relación sobre $X$). Para $y\in X$ definimos $R[y]\subseteq X$ dado por:
$$R[y]:=\{x\in X:\ (x,y)\in R\}.$$
El inverso de $R$, denotado por $R^{-1}:=\{(x,y)\in X\times X:\ (y,x)\in R\}$. Decimos que $R$ es simétrica si $R=R^{-1}$.\\
Si $R$ y $S$ son relaciones sobre $X$, definimos su composición como: $$R\circ S:=\{(x,y)\in X\times X:\ \exists z\in X\ \mbox{ tal que }\ (x,z)\in R, \ (z,y)\in S\}.$$

\begin{definicion}{(Estructura uniforme)}

Sea $X$ un conjunto, una estructura uniforme $\mathcal{U}$ es una colección de subconjuntos de $X\times X$ que satisface las siguientes condiciones:
\begin{enumerate}
\item[i)] Si $V\in \mathcal{U}$, entonces $\Delta_{X} \subset V$. 
\item[ii)] Si $V\in \mathcal{U}$ y $V\subset V' \subset X\times X$, entonces $V' \in \mathcal{U}$.
\item[iii)] Si $V\in \mathcal{U}$ y $W\in \mathcal{U}$, entonces $V\cap W \in \mathcal{U}$.
\item[iv)] Si $V\in \mathcal{U}$, entonces $V^{-1}\in \mathcal{U}$.
\item[v)] Si $V\in \mathcal{U}$, entonces existe $W\in \mathcal{U}$ tal que $W\circ W \subset V$.
\end{enumerate}
\end{definicion}

Un conjunto $X$ equipado con una estructura uniforme $\mathcal{U}$ es llamado \emph{espacio uniforme} y los elementos de $\mathcal{U}$ son llamados \emph{séquitos} de $X$.

\begin{ejemplo}
Dado un conjunto $X$, entonces $\mathcal{U}=\{X\times X\}$ es una estructura uniforme sobre $X$. Esta estructura uniforme es llamada \emph{estructura uniforme trivial} sobre $X$ y es la estructura uniforme más pequeña sobre $X$.
\end{ejemplo}

\begin{ejemplo}
La \emph{estructura uniforme discreta} sobre un conjunto $X$ es la estructura uniforme cuyos séquitos consisten en todos los subconjuntos de $X\times X$ que contienen $\Delta_{X}$. Esta es la estructura uniforme más grande sobre $X$. Se sigue de la segunda propiedad de las estructuras uniformes que la estructura uniforme discreta de $X$ es la única estructura uniforme de $X$ que admite la diagonal $\Delta_{X}$ como séquito.
\end{ejemplo}

Una estructura uniforme se dice \emph{metrizable} si está asociada con alguna métrica sobre $X$. Por ejemplo, la estructura uniforme discreta está asociada a la métrica discreta sobre $X$.

\begin{ejemplo}
Supongamos que $d$ es una métrica sobre $X$. Para cada $\varepsilon > 0$ definimos $V_{\varepsilon}\subset X\times X$ como el conjunto de pares $(x,y)$ tales que $d(x,y)< \varepsilon$. Sea $\mathcal{U}$ el conjunto de todos los subconjuntos  $W\subset X\times X$ tales que podemos encontrar un $\varepsilon >0$ para el cual se cumple que $V_{\varepsilon}\subset W$. Entonces $\mathcal{U}$ es una estructura uniforme sobre $X$ y es llamada la estructura uniforme asociada a la metrica $d$.
\end{ejemplo}

Sea $X$ un espacio uniforme, es fácil verificar que es posible definir una topología sobre $X$ tomando como conjuntos abiertos los subconjuntos $U\subset X$ los cuales satisfacen la siguiente propiedad: para cada $x\in U$, existe un séquito $V\subset X\times X$ tal que $V[x]=U$. Decimos que esta es la topología asociada a la estructura uniforme de $X$. Un subconjunto $N\subset X$ es un entorno de un punto $x\in X$ para esta topología si y solo si existe un séquito $V$ tal que $N=V[x]$. Esta topología es Hausdorff si y solo si la intersección de los séquitos de $X$ coincide con la diagonal $\Delta_{X}\subset X \times X$. \\

Sea $\mathcal{U}$ una estructura uniforme sobre $X$ y $Y$ un subconjunto de $X$. Entonces podemos definir una estructura uniforme sobre $Y$ inducida por la estructura uniforme sobre $X$ dada por $$\mathcal{U}_{Y}:=\{V\cap(Y\times Y): V\in \mathcal{U}\}.$$ 

La topología asociada a la estructura uniforme asociada a una métrica $d$ es la topología asociada a la métrica $d$. La topología asociada a la estructura uniforme inducida sobre $Y$ por la estructura uniforme de $X$ coincide con la topología inducida sobre $Y$ por la topología de $X$.

\begin{proposicion}
Sea $X$ un conjunto y sea $\mathcal{B}$ un conjunto no vacío de subconjuntos de $X\times X$. Entonces $\mathcal{B}$ es una base para alguna (necesariamente única) estructura uniforme sobre $X$ si y solo si se satisfacen las siguientes propiedades:
\begin{enumerate}
\item[i)] Si $V\in \mathcal{B}$, entonces $\Delta_{X} \subset V$.
\item[ii)] Si $V\in \mathcal{B}$ y $W \in \mathcal{B}$, entonces existe $U\in \mathcal{B}$ tal que $U\subset V\cap W$.
\item[iii)] Si $V\in \mathcal{B}$, entonces existe $W\in \mathcal{B}$ tal que $W\subset V^{-1}$.
\item[iv)] Si $V\in \mathcal{B}$, entonces existe $W\in \mathcal{B}$ tal que $W\circ W \subset V$.
\end{enumerate}
\end{proposicion}

Dada una función $f:X\rightarrow Y$, definimos una función en los productos cartecianos $f\times f: X\times X \rightarrow Y\times Y$ dada por $(f\times f) (x_{1},x_{2})=(f(x_{1}),f(x_{2}))$.

\begin{definicion}{(Función uniformemente continua)}
Sean $X$ y $Y$ estructuras uniformes, una función se dice \emph{uniformemente continua} si para cada séquito $W$ de $Y$, existe un séquito $V$ de $X$ tal que $(f\times f)(V)\subset W$.
\end{definicion}

Si $\mathcal{B}$ y $\mathcal{B}'$ son bases para $X$ y $Y$ respectivamente, entonces una función $f:x\rightarrow Y$ es uniformemente continua si y solo si se satisface que para cada $W\in \mathcal{B}'$, existe un $V\in \mathcal{B}$ tal que $(f\times f)(V) \subset W$.  Notese que esta condición es equivalente a que $(f\times f)^{-1}(W)$ es un séquito de $X$ para cada séquito $W$ de $Y$.

\begin{proposicion}
Sean $X$ y $Y$ estructuras uniformes. Entonces toda función uniformemente continua $f:x\rightarrow Y$ es continua (con respecto a las topologías asociadas a las estructuras uniformes de $X$ y $Y$). 
\end{proposicion}

\begin{proof}
Supongamos que $f:X \rightarrow Y$ es uniformemente continua. Sea $x\in X$ y sea $N\subset Y$ una vecindad de $f(x)$. Entonces existe un séquito $W$ de $Y$ tal que $W[f(x)]=N$. Como $f$ es uniformemente continua, el conjunto $V=(f\times f)^{-1}(W)$ es un séquito de $X$. El conjunto $V[x]$ es una vecindad de $X$ y satisface que $f(V[x])\subset W[f(x)]=N$. Esto muestra que $f$ es continua. 
\end{proof}

Una función continua entre espacios uniformes puede no ser uniformemente continua. 
\begin{ejemplo}
La función $x\mapsto x^{2}$ no es uniformemente continua sobre $\mathbb{R}$, equipada con la estructura uniforme asociada con su métrica usual.
\end{ejemplo}

\begin{definicion}{(Isomorfismo uniforme)}

Sean $X$ y $Y$ estructuras uniformes. Una función $f:X \rightarrow Y$ es un \emph{isomorfismo uniforme} si $f$ es biyectiva y tanto $f$ como $f^{-1}$ son uniformemente continuas.
\end{definicion}

Un isomorfismo uniforme de un conjunto $X$ en su mismo recibe el nombre de \emph{automorfismo}.

\subsubsection{Producto de espacios uniformes}

Sea $X$ un conjunto. Supongamos que tenemos una familia $(X_{\lambda})_{\lambda\in \Lambda}$ de espacios uniformes y una familia $(f_{\lambda})_{\lambda\in \Lambda}$ de funciones $f_{\lambda}:X \rightarrow X_{\lambda}$. Entonces la \emph{estructura uniforme inicial} asociada con estas es la estructura uniforme sobre $X$ más pequeña tal que todas las funciones $f_{\lambda}:X\rightarrow X_{\lambda}$ son uniformemente continuas. \\
En el caso particular cuando $x=\prod_{\lambda\in\Lambda}X_{\lambda}$ y $f_{\lambda}:X\rightarrow X_{\lambda}$ es el mapeo proyección, la estructura uniforme inicial asociada sobre $X$ es llamada \emph{estructura uniforme producto}. Una base de séquitos para la estructura uniforme producto sobre $X$ es obtenida al tomar todos los subconjuntos de $X\times X$ los cuales son de la forma:
$$\prod_{\lambda\in \Lambda}V_{\lambda}\subset \prod_{\lambda \in \Lambda} X_{\lambda}\times X_{\lambda}= \left( \prod_{\lambda \in \Lambda}X_{\lambda} \right) \times \left( \prod_{\lambda \in \Lambda}X_{\lambda} \right) = X\times X ,$$
donde $V_{\lambda}\subset X_{\lambda}\times X_{\lambda}$ es un séquito de $X_{\lambda}$ y $V_{\lambda}=X_{\lambda}\times X_{\lambda}$ para todos salvo un número finito de $\lambda \in \Lambda$. \\
Cuando cada $X_{\lambda}$ esta dotado con la estructura uniforme discreta, la estructura uniforme producto sobre $X=\prod_{\lambda\in \Lambda}$ es llamada \emph{estructura uniforme prodiscreta}.

\subsubsection{Teorema de Curtis-Hedlund generalizado}

Sea $G$ un grupo y $A$ un conjunto. La \emph{estructura uniforme prodiscreta} sobre $A^{G}$ es la estructura uniforme producto obtenida por tomar la estructura uniforme discreta sobre cada factor $A$ de $A^{G}=\prod_{g\in G}{A}$.\\
Una base de séquitos para la estructura uniforme prodiscreta sobre $A^{G}$ está dada por los conjuntos $W_{\Omega}\subset A^{G} \times A^{G}$, donde
$$W_{\Omega}=\{(x,y)\in A^{G}\times A^{G}:\ x|_{\Omega}=y|_{\Omega}\}$$
y $\Omega$ corre sobre todos los subconjuntos finitos de $G$.\\
Observemos que $$V(x,\Omega)=\{y\in A^{G}:\ (x,y)\in W_{\Omega}\},$$
para todo $x\in A^{G}$.

\begin{teorema}{(Teorema de Curtis-Hedlund generalizado)}\label{curtis-hedlund gen}

Sea $G$ un grupo y $A$ un conjunto. Sea $\tau:A^{G} \rightarrow A^{G}$ una función y equipamos $A^{G}$ con la estructura uniforme prodiscreta . Entonces las siguientes afirmaciones son equivalentes.

\begin{enumerate}
\item[i)] $\tau$ es un autómata celular.
\item[ii)] $\tau$ es uniformemente continua y $G$-equivariante.

\end{enumerate}

\end{teorema}
\begin{proof}
Supongamos que $\tau:A^{G} \rightarrow A^{G}$ es un autómata celular. Ya sabemos que $\tau$ es $G$-equivariante. Debemos mostrar que es uniformemente continua. Sea $S$ un conjuto memoria para $\tau$. Sabemos que si dos configuraciones $x,y\in A^{G}$ coinciden en $gS$ para algún $g\in G$, entonces $\tau(x)(g)=\tau(y)(g)$. Consecuentemente, si las configuraciones $x$ y $y$ coinciden en $\Omega S=\{gs:\ g\in \Omega,\ s\in S\}$ para algún subconjunto $\Omega \subset G$, entonces $\tau(x)$ y $\tau(y)$ coinciden en $\Omega$. Observemos que $\Omega S$ es finito siempre que $\Omega$ sea finito. Deducimos entonces que $$(\tau \times \tau)(W_{\Omega S}) \subset W_{\Omega S}$$ para cada subconjunto finito $\Omega$ de $G$. Como los conjuntos $W_{\Omega}$, donde $\Omega$ corre por todos los subconjuntos finitos de $G$, forman una base de séquitos para la estructura uniforme prodiscreta sobre $A^{G}$, se sigue que $\tau$ es uniformemente continua. Esto muestra que i) implica ii).\\
Conversamente, si suponemos que $\tau$ es uniformemente continua y $G$-equivariante. Probemos pues que $\tau$ es un autómata celular. Consideremos el subconjunto $\Omega=\{1_{G}\}\subset G$. Como $\tau$ es uniformemente continua, existe un subconjunto finito $S \subset G$ tal que $(\tau \times \tau )(W_{S})\subset W_{\Omega})$. Esto significa que $\tau(x)(1_{G})$ solo depende de la restricción de $x$ sobre $S$. Entonces, existe una función $\mu:A^{S}\rightarrow A$ tal que $$\tau(x)(1_{G})=\mu(x|_{S}),$$ para todo $x\in A^{G}$. Usando la $G$-equivarianza de $\tau$ obtenemos que $$\tau(x)(g)=[g^{-1}\tau(x)](1_{G})=\tau(g^{-1}x)(1_{G})=\mu((g^{-1}x)|_{S})$$ para todo $x\in A^{G}$. Esto muestra que $\tau$ es un autómata celular con conjunto memoria $S$ y regla local $\mu$. 
\end{proof}

Cada función uniformemente continua entre espacios uniformes es continua con respecto a las topologías asociadas y el converso es cierto cuando el dominio es un espacio compacto. La topología definida por la estructura uniforme prodiscreta en $A^{G}$ es la topología prodiscreta. En el caso de que A sea finito, la topología prodiscreta en $A^{G}$ es compacta según el Teorema de Tychonoff. Por tanto, este teorema se reduce al Teorema de Curtis-Hedlund en este caso.

\begin{corolario}[Autómatas celulares invertibles]
Sea $G$ un grupo y $A$ un conjunto. Sea $\tau:A^{G} \rightarrow A^{G}$ una función y equipemos a $A^{G}$ con su estructura uniforme prodiscreta. Entonces las siguientes afirmaciones son equivalentes:
\begin{enumerate}
\item[a)] $\tau$ es un autómata celular invertible.
\item[b)] $\tau$ es una automorfismo uniforme $G$-equivariante de $A^{G}$. 
\end{enumerate}
\end{corolario}

\begin{proof}
Es fácil ver que la función inversa de una función $G$-equivariante de $A^{G}$ en su mismo es también $G$-equivariante. Entonces, la condición de equivalencia entre las afirmaciones se sigue de la caracterización de autómatas celulares dada por el Teorema de Curtis-Hedlund generalizado \ref{curtis-hedlund gen}.
\end{proof}

\end{document}